\documentclass[a4paper, 11 pt, reqno]{amsart}
\usepackage{amsmath}
\usepackage{amsfonts}
\usepackage{amssymb}
\usepackage{amsthm}
\usepackage[all, cmtip]{xy}
\usepackage{tikz-cd}
\usepackage{mathrsfs}
\usepackage{textcomp}
\usepackage{hyperref}

\newcommand{\un}{\underline}
\newcommand{\ov}{\overline}

\newcommand{\bA}{\mathbb{A}}
\newcommand{\bC}{\mathbb{C}}

\newcommand{\bK}{\mathbb{K}}

\newcommand{\cC}{\mathcal{C}}
\newcommand{\cE}{\mathcal{E}}
\newcommand{\cF}{\mathcal{F}}

\newcommand{\cL}{\mathcal{L}}

\newcommand{\cT}{\mathcal{T}}

\newcommand{\sC}{{\mathscr{C}}}
\newcommand{\sE}{{\mathscr E}}
\newcommand{\sF}{{\mathscr F}}
\newcommand{\sG}{{\mathscr G}}

\newcommand{\sI}{{\mathscr I}}
\newcommand{\sK}{{\mathscr K}}
\newcommand{\sL}{{\mathscr L}}

\newcommand{\sN}{{\mathscr N}}
\newcommand{\sO}{{\mathscr O}}

\newcommand{\sT}{{\mathscr T}}

\newcommand{\bsF}[1]{\ov{\sF}_{#1}}

\newcommand{\mm}[1]{\mathfrak m_{#1}}

\newcommand{\y}[1]{\bar{y}_{#1}}
\newcommand{\x}[1]{\bar{x}_{#1}}
\newcommand{\M}[1]{\ov{M}_{#1}}

\newcommand{\be}[1]{b_{#1}}
\newcommand{\bet}[2]{b_{#1}(#2)}
\newcommand{\bp}[2]{b_{#1,#2}}
\newcommand{\btp}[3]{b_{#1,#2}(#3)}

\newcommand{\coker}[1]{\operatorname{coker}(#1)}
\renewcommand{\deg}[1]{\operatorname{deg}(#1)}
\newcommand{\Deg}[1]{\operatorname{Deg}(#1)}
\newcommand{\Ext}[1]{\operatorname{Ext}^{#1}}
\renewcommand{\ker}[1]{\operatorname{ker}(#1)}
\newcommand{\Gr}[2]{\operatorname{Gr}_{#1}(#2)}
\newcommand{\Hom}{\operatorname{Hom}}
\newcommand{\im}[1]{\operatorname{im}(#1)}
\newcommand{\Pic}[1]{\operatorname{Pic}(#1)}
\newcommand{\rk}[1]{\operatorname{rk}(#1)}
\newcommand{\Rk}[1]{\operatorname{R}(#1)}
\newcommand{\Spec}[1]{\operatorname{Spec}(#1)}

\newcommand{\SEnd}[1]{\un{\operatorname{End}}(#1)}
\newcommand{\SExt}[2]{\un{\operatorname{Ext}}^{#1}_{#2}}
\newcommand{\Specrel}[1]{\un{\operatorname{Spec}}\big(#1\big)}

\newcommand{\inj}{\hookrightarrow}

\newcommand{\surj}{\twoheadrightarrow}

\theoremstyle{plain}
\newtheorem{thm}{Theorem}[section]
\newtheorem{lemma}[thm]{Lemma}

\newtheorem{prop}[thm]{Proposition}
\newtheorem{cor}[thm]{Corollary}
\newtheorem{fact}[thm]{Fact}
\newtheorem{conj}[thm]{Conjecture}
\newtheorem{teor}{Theorem}

\theoremstyle{remark}
\newtheorem{rmk}[thm]{Remark}

\theoremstyle{definition}
\newtheorem{defi}[thm]{Definition}

\numberwithin{equation}{section}

\begin{document}

\title[Generalized line bundles on primitive multiple curves]{Generalized line bundles on primitive multiple curves and their moduli}

\author{Michele Savarese}
\address{Dipartimento di Matematica e Fisica,
Universit\`a Roma Tre, \\
Largo San Leonardo Murialdo 1,
00146 Roma (Italia)}
\email{msavarese@mat.uniroma3.it}

\keywords{Generalized Line Bundles, Primitive Multiple Curves, Generalized Divisors, Moduli Space of semistable sheaves, Compactified Jacobian}
\subjclass[2010]{14D20; 14H60}

\begin{abstract}
In this paper, we study generalized line bundles over $C_n$, a primitive multiple curve of arbitrary multiplicity $n$, where $n$ is a positive integer. In particular, we give a structure theorem for them and we characterize their semistability in terms of $n-1$ integral invariants, the indices. These results are used to describe the irreducible components that contain stable generalized line bundles in the Simpson moduli space of semistable sheaves of generalized rank $n$ and fixed generalized degree on $C_n$. We compute also the dimension of the Zariski tangent space to this moduli space in a point representing a generic generalized line bundle (any generalized line bundle for $n=3$). In the case $n=1$ all the results are completely trivial, while the case $n=2$ has already been treated in \cite{CK}.
\end{abstract}

\maketitle

\section*{Introduction}\label{SecIntro}
This work is devoted to the study of generalized line bundles on primitive multiple curves, which are a special kind of non-reduced curves, and of the moduli spaces of semistable ones. It extends the results already known on ribbons (cf. \cite{CK}), which are the easiest and most well-known type of primitive multiple curves, i.e. those of multiplicity $2$. It is also a partial answer to the first questions posed in \cite[\S 4]{DEL}, where it is suggested to study the moduli of sheaves on what is there called a ribbon of order $n$, which is precisely a primitive multiple curve of multiplicity $n$, with a special attention to a particular kind of sheaves which are exactly generalized line bundles.

\subsection*{Primitive multiple curves}
This is only an extremely brief introduction to the subject, for more details we refer to Section \ref{SecPmc}.

A primitive multiple curve $C_n=X$ is a Cohen-Macauley non-reduced but irreducible scheme of dimension $1$ over an algebraically closed field $\bK$ such that its reduced subscheme $C=X_{red}$ is a smooth projective curve and locally its nilradical is a principal ideal (or, equivalently, $X$ can be locally embedded in a smooth surface). Let $\sN\subset\sO_X$ be its nilradical, then $X$ is said to be of multiplicity $n$ if $\sN^n=0$ and $\sN^{n-1}\neq 0$; in the case $n=2$ it is just an irreducible ribbon (standard references about ribbons and generalized line bundles on them are \cite{BE} and \cite{EG}). $C_n$ admits a filtration $ C=C_1\subset C_2\subset\dotsb\subset C_{n-1}\subset X=C_n$, where $C_i$ is a primitive multiple curve of multiplicity $i$, for $2\le i\le n-1$. Multiple curves were introduced by B\u{a}nic\u{a} and Forster in \cite{BF} and primitive ones have been studied by Dr\'{e}zet in various articles, among which there are \cite{DR3}, parametrizing them and inspired by \cite{BE}, and \cite{DR1}, \cite{DR2} and \cite{DR4}, where coherent sheaves on them are studied. Note that in all Dr\'{e}zet articles $\bK$ is assumed to be $\bC$, but it seems that this hypothesis is not needed for the results we will use.

Any  coherent sheaf $\sF$ on $X$ has two fundamental invariants, introduced in \cite{DR1}: the generalized rank $\Rk{\sF}$ and the generalized degree $\Deg{\sF}$. A generalized line bundle $\sF$ on $X$ is a pure coherent sheaf that is generically a line bundle (i.e. $\sF_{\eta}\cong\sO_{X,\eta}$, where $\eta$ is the generic point of $X$); in particular, it has generalized rank $n$. It is relevant to observe that generalized line bundles coincide, in this context, with generalized divisors introduced by Hartshorne in \cite{H} (at a level of generality sufficient to comprehend primitive multiple curves); the coincidence is due to \cite[Proposition 2.8]{H}.

Primitive multiple curves and coherent sheaves on them are interesting objects of study mainly because they are a kind of non-reduced curves relatively easy to handle (particularly, in the case of multiplicity $2$) and non-reduced schemes and sheaves on them are still quite unknown. Indeed, according to my knowledge, there are only few papers about coherent sheaves on non-reduced schemes and their Simpson moduli space. In arbitrary dimension (but with special attention to curves and degenerate quadric surfaces) there is Inaba's article \cite{I}, while in the case of curves there are Dr\'{e}zet's studies \cite{DR1}, \cite{DR2} and \cite{DR4} for primitive multiple curves, that of Chen and Kass about the compactified Jacobian of a ribbon \cite{CK}, and Yang's one \cite{Y} about coherent sheaves on fat curves (within which there are ribbons and, more generally, ropes, but not primitive multiple curves of higher multiplicity). Some of the results of \cite{CK} had already been stated, without proofs and under more restrictive hypotheses, by Donagi, Ein and Lazarsfeld in \cite{DEL}. Finally, sheaves on ribbons are studied also in \cite{S3}, which completes the description of the irreducible components of the compactified Jacobian answering to \cite[Question 4.8]{CK}, and in \cite{S2}, which is essentially a reformulation of a long section of my forthcoming doctoral thesis \cite{S1}.

The Simpson moduli space of coherent sheaves of generalized rank $n$ on $C_n$ (including, in particular, generalized line bundles) is interesting also because it is a natural compactification of the Jacobian of $C_n$, when line bundles are stable (this happens, as we will see later, when the degree of the conormal sheaf of the reduced subcurve in the primitive multiple curve is negative).

Another reason of interest of primitive multiple curves is due to the fact that, when they have a retraction to the reduced subcurve, they are involved in the so-called spectral correspondence (for a brief introduction about twisted Higgs pairs and spectral correspondence see \cite[Appendix]{MRV2} or \cite{HP} for a longer one): if $C$ is a smooth projective curve, the spectral cover associated to nilpotent Higgs pairs of rank $n$ over $C$ is a primitive multiple curve of multiplicity $n$ with reduced subcurve $C$ and there is an isomorphism between the moduli space of (semistable) pure coherent sheaves of generalized rank $n$ on it and (semistable) nilpotent Higgs pairs of rank $n$ over $C$.
 
\subsection*{Structure of the work and main results}
This article begins with an introductory chapter about the theory of coherent sheaves on a primitive multiple curve, collecting the results and tools which will be used in the next ones; it is almost entirely based on \cite{DR1} and \cite{DR2}. It is divided in six subsections: the first one recalls the definition of a primitive multiple curve and its basic properties. The second one treats briefly line bundles and the Picard scheme of $C_n$, a primitive multiple curve of multiplicity $n$. The third one introduces the two canonical filtrations of a coherent sheaf on $C_n$ and their main properties. The forth section is about two fundamental invariants of a coherent sheaf: the generalized rank and the generalized degree. There we explain also their relation with ordinary rank and degree. The fifth one recalls the equivalent (on a primitive multiple curve) notions of pure sheaf of dimension $1$, torsion-free sheaf and reflexive sheaf. It treats also duality of sheaves and, in particular, the relations between the two canonical filtrations of dual sheaves. Finally, the sixth subsection is a brief overview about semistability of sheaves on a primitive multiple curve. We do not specify slope or Gieseker semistability because, as we will see in this subsection, these notions are equivalent on primitive multiple curves, as on smooth projective ones.

Section \ref{SecFirstProp} is concerned with various properties of generalized line bundles on $C_n$, with some results more generally about torsion-free sheaves of generalized rank $n$ on $C_n$. It is inspired by the case of ribbons treated in \cite[\S 2]{CK}. In particular, we introduce the indices $\bet{1}{\sF},\,\dotsc,\,\bet{n-1}{\sF}$ (which are non-negative integers) of a generalized line bundle $\sF$ on $C_n$, which will play a significant role throughout the whole work, and the associated torsion sheaves $\sT_i(\sF)$ for $1\le i\le n-1$ (see Definition \ref{Def:indices}).

Section \ref{SecStruc} studies the structure of a generalized line bundle on a primitive multiple curve. While it is quite easy to describe it on a ribbon (cf. \cite[Theorem 1.1]{EG}), the situation is much more complicated in higher multiplicity. The main result is the following: 
\begin{teor}\label{TeorIntroStruc}
Let $\sF$ be a generalized line bundle on $C_n$. Then $\sF$ is isomorphic to $\sI_{Z/C_n}\otimes\sG$, where $Z\subset C_{n-1}$ is a closed subscheme of finite support whose schematic intersection with $C$ is $\operatorname{Supp}(\sT_{n-1}(\sF))$, called the subscheme associated to $\sF$, and $\sG$ is a line bundle on $C_n$.\\
Moreover
\begin{enumerate}
\item $Z$ is unique up to adding a Cartier divisor.
\item Locally isomorphic generalized line bundles have the same associated subscheme, up to adding a Cartier divisor. In particular, if $\sF$ and $\sF'$ are locally isomorphic generalized line bundles, then there exists a line bundle $\sE$ such that $\sF=\sF'\otimes \sE$. Equivalently, there is a transitive action of $\Pic{X}$ on the set of locally isomorphic generalized line bundles.
\end{enumerate}
\end{teor}

In the text it appears as Corollary \ref{Cor:globstr}, because it is a consequence of the extremely involved local description, given in Theorem \ref{Thm:locstr}, and \cite[Proposition 2.12]{H}. The above cited action is studied with particular attention for some special types of generalized line bundles which will play a fundamental role in determining the irreducible components of the moduli space (see Corollaries \ref{Cor:globstrmolt3} and \ref{Cor:globstrmoltn}).

Section \ref{SecStab} studies semistability of generalized line bundles on $C_n$; the main result is Theorem \ref{Thm:ssinequ}:
\begin{teor}\label{TeorIntroSemistabglb}
Let $\sF$ be a generalized line bundle of generalized degree $D$ on $C_n$ with indices $\bet{1}{\sF},\,\dotsc,\,\bet{n-1}{\sF}$. Then $\sF$ is semistable if and only if the following inequalities hold:
\[
i\sum\limits_{j=i}^{n-1}\bet{j}{\sF}-(n-i)\sum\limits_{j=1}^{i-1}\bet{j}{\sF}\le-\frac{in(n-i)}{2}\deg{\sN/\sN^2}, \; \forall\; 1\le i\le n-1,  
\]
where $\sN$ is the nilradical of $\sO_{C_n}$.

It is stable if and only if all the inequalities are strict. 
\end{teor}
The case of ribbons had already been treated in \cite[\S 3]{CK}.
Another significant result of this section is the computation of a surprisingly canonical Jordan-Holder filtration of a strictly semistable generalized line bundle (see Proposition \ref{Prop:JH}).  

Section \ref{SecModuli} studies the irreducible components of the moduli space of stable generalized rank $n$ sheaves on $C_n$ that contain stable generalized line bundles. It extends some results of \cite[\S 4]{CK} to higher multiplicity and it is divided into two sections, the first about multiplicity $3$, about which we know something more, and the other about multiplicity greater than or equal to $4$. These irreducible components of stable generalized line bundles, which are all of the same dimension, when they exist, are completely described. There are also some results about the local geometry of the moduli space; in particular, we compute the dimension of the tangent space to points representing some special (any in multiplicity $3$) generalized line bundles (see Proposition \ref{Prop:DimSpTgn}), including, in particular, the generic elements of the irreducible components, which result to be generically smooth only when their generic element is a line bundle. The main results, which are Theorems \ref{Thm:compirrid} and \ref{Thm:connessioneluogoflgen} and Corollary \ref{Cor:SpTgpgen} for higher multiplicity, can be summarized in a simplified version as follows:

\begin{teor}\label{TeorIrrComp}
Let $C_n$ be a primitive multiple curve of arithmetic genus $g_n$ such that $\deg{\sN/\sN^2}<0$, where $\sN$ is the nilradical of $\sO_{C_n}$.
\begin{enumerate}
\item The closure of the locus of stable generalized line bundles of fixed indices $\be{1},\dotsc,\be{n-1}$ on $C_n$, $\bar{Z}_{\be{1},\dotsc,\be{n-1}}$, is a $g_n$-dimensional irreducible component of the moduli space of semistable sheaves of generalized rank $n$ (when this locus is not empty).
\item The union of these loci is connected for $n=3$ or for $n\ge 4$ and $\deg{\sN/\sN^2}$ sufficiently small.
\item The tangent space to the generic point of $\bar{Z}_{\be{1},\dotsc,\be{n-1}}$ has dimension $g_n+\sum_{i=[\frac{n+1}{2}]}^{n-1}\be{i}-\sum_{i=1}^{[\frac{n-2}{2}]}\be{i}$. 
\end{enumerate} 

\end{teor} 

It was not possible to get a complete description of the irreducible components of the moduli space, because within semistable pure coherent sheaves of generalized rank $n$ on $C_n$ there are also direct images of semistable pure coherent sheaves of generalized rank $n$ on $C_{i}$, for any $1\le i\le n-1$ and they are quite hard to handle, in general. In the case of ribbons, treated in \cite{CK}, this is not a real problem because there is only $C_1$, i.e. the reduced subcurve, to be considered and pure sheaves of generalized rank $2$ on it are just vector bundles of rank $2$, whose moduli space is well-known.
For a conjecture about the global picture, see \cite[Section 4.3]{S1}.

\section[Sheaves on primitive multiple curves]{Generalities on sheaves on primitive multiple curves}\label{SecPmc} 
As anticipated in the introduction, this section collects the definition and few properties of primitive multiple curves (in the first subsection) and the basis of the theory of coherent sheaves on them developed in \cite{DR1}, \cite{DR2} and \cite{DR4} (in the next ones). Here we fix also notations and conventions which will be used throughout the work.

\subsection{Primitive multiple curves}\label{SubSecPmc}
This subsection is based on \cite[\S 2.1]{DR1} and \cite[\S 2.1]{DR2}. 
\begin{defi}\label{Def:pmc}
A \emph{primitive multiple curve} of \emph{multiplicity} $n$ is an irreducible Cohen-Macaulay algebraic scheme $(X=C_n,\sO_X=\sO_{C_n})$ over an algebraically closed field $\bK$ such that:
\begin{enumerate}
\item\label{Def:pmc:1} its reduced subscheme $(X_{red},\sO_{X_{red}})$ is a smooth projective curve $(C,\sO_C)$ over $\bK$;
\item the multiplicity $n$ is the least natural number such that $\sN^n=0$, where $\sN=\ker{\sO_X\surj\sO_C}$ is the \emph{nilradical} ideal sheaf of $\sO_{X}$;
\item\label{Def:pmc:3} it is locally embedded in a smooth surface, i.e. any closed point admits a neighbourhood that can be embedded in a smooth surface, or, equivalently, the nilradical is locally a principal ideal.
\end{enumerate}
\end{defi}
\begin{rmk}\label{Rmk:pmc}
\noindent
If hypothesis \ref{Def:pmc:3} is omitted, then $X$ is called a \emph{multiple curve} or a \emph{multiple structure} over $C$, but in this work only primitive ones are treated.

\end{rmk}
Observe that a primitive multiple curve of multiplicity $1$ is just a smooth projective curve and a primitive multiple curve of multiplicity $2$ is just a ribbon over a smooth projective curve (which from now on will be called simply a ribbon). 
The topological space underlying a primitive multiple curve is homeomorphic to that of its reduced subcurve, but the structure sheaves are quite different. Indeed, if $P\in X$ is a closed point, it holds that $\sO_{X,P}=\sO_{C,P}\otimes_{\bK} (\bK[y]/(y)^n)$.

The above definition is not the original one, used in \cite{DR1} and given in terms of an ambient three-fold, but this abstract one, used e.g. in \cite{DR2}, is equivalent to the embedded one by \cite[Th\'{e}or\`{e}me 5.3.2]{DR3}.

The arithmetic genus of $C_n$, equal to $1-\chi(C_n,\sO_{C_n})$, will be denoted by $g(C_n)=g_n$ and will be called simply the \emph{genus} of $C_n$ (more generally for any curve $Y$ that will appear throughout the work its genus will be $g(Y)=1-\chi(Y,\sO_Y)$).

There is a \emph{canonical filtration} of $X$ by closed subschemes
\[
C_1=C\subset C_2\subset \dotsb \subset  C_{n-1}\subset C_n=X,
\] 
where $C_i$ is a primitive multiple curve of multiplicity $i$ (whose genus will be denoted by $g_i$) with reduced subcurve $C$ and such that its ideal sheaf in $X$ is $\sI_{C_i/X}=\sN^i$, for $1\le i\le n-1$. It holds that $\sN$ is a line bundle over $C_{n-1}$ and there exists a line bundle $\sC$ over $C_n$ extending it.

The \emph{conormal} sheaf of $C$ in $X$ is $\sN/\sN^2$ and is denoted by $\cC$. It is a line bundle over $C$ and it plays a quite important role in the study of $X$: if its degree is negative, $X$ has no non-constant global sections (cf. \cite[\S 2.6]{DR2}) and so in this case $g_n=\operatorname{h}^1(X,\sO_X)$. Moreover, $\sI_{C_i/X}/\sI_{C_{i+1}/X}=\sN^i/\sN^{i+1}$ is a line bundle on $C$ and it is equal to $\cC^{\otimes i}$, for $2\le i\le  n-1$. The nilradical ideal of $\sO_{C_i}$ is $\sN/\sN^i$, for any $2\le i\le n-1$. This implies that the conormal sheaf of $C$ with respect to $C_i$ is again $\cC$ (indeed, it is evident that $(\sN/\sN^i)/(\sN/\sN^i)^2=\sN/\sN^2$).

A primitive multiple curve is called \emph{split} if it admits a retraction to the reduced subcurve. A primitive multiple curve of multiplicity $n$ is \emph{trivial}, if it is the $n$-th infinitesimal neighbourhood of a smooth projective curve $C$ in the geometric vector bundle associated to $\cL^*$, where $\cL$ is a line bundle on $C$. These are the only primitive multiple curves that appear in the spectral correspondence. Any trivial primitive multiple curve is split. The converse holds in general only in multiplicity $2$ (essentially by \cite[Proposition 1.1]{BE}), while it is false in higher multiplicity (cf. \cite[\S 1.1.6]{DR3}).

\subsection{Line bundles and the Picard scheme}\label{SubSecLinBunPicGr}
In this short subsection we collect some useful facts about line bundles on $C_n$ and its Picard group.

The first relevant properties are the following, which are, respectively, \cite[Th\'{e}or\`{e}me 3.1.1 and \S 3.1.5]{DR1}:
\begin{fact}\label{Fact:extvb}
\noindent
\begin{enumerate}
\item\label{Fact:extvb:1}For $1\le i\le n-1$, any line bundle (and more generally any vector bundle) on $C_i$ extends to a line bundle (resp. vector bundle of the same rank) on $C_n$.
\item\label{Fact:extvb:2}Let $\sL$ be a line bundle over $C_{n-1}$ and $\cL=\sL|_C$. Then there is a short exact sequence
\[
0\to H^1(\cC^{n-1})\to\Ext{1}_{\sO_{C_n}}(\cL,\sL\otimes\sN)\overset{\pi}{\to}\bK\to 0.
\]
Moreover, the set $P_\sL$ of line bundles extending $\sL$ over $C_n$ is identified with $\pi^{-1}(1)$, which is an affine space isomorphic to $H^1(\cC^{n-1})$. In particular, the bijection between the latter and $P_{\sO_{C_{n-1}}}$ is an isomorphism of abelian groups.
\end{enumerate} 
\end{fact}

On a primitive multiple curve $C_n$ there exists, by, e.g., \cite[Theorem 8.2.3]{BLR}, the so-called Picard scheme $\Pic{C_n}$. It is a scheme locally of finite type parametrizing line bundles on $C_n$ and endowed with a tautological line bundle, called the Poincar\'{e} line bundle, over $\Pic{C_n}\times C_n$. A general introduction to the rich theory of relative and absolute Picard schemes can be found in \cite{K} or in \cite[Chapter 8]{BLR}.

The following fact is an application to our case of the general theory; in particular, the first point follows from \cite[Proposition 9.5.3]{K} for separateness, from \cite[Proposition 9.5.19]{K} for the smoothness and from \cite[Theorem 8.4.1]{BLR} or, equivalently, \cite[Corollary 9.5.13]{K} for the dimension. The second and the third point are contained in \cite[\S 3.3]{DR1}, while the last assertion follows from the general theory and from the previous points and is inspired by the case of ribbons treated in \cite[Fact 2.10]{CK}. 
\begin{fact}\label{Fact:Pic} Let $C_n$ be a primitive multiple curve. Then
\begin{enumerate} 
\item\label{Fact:Pic:1} The Picard scheme $\Pic{C_n}$ for $C_n$ is smooth, separated and of dimension $\operatorname{h}^1(C_n,\sO_{C_n})$.
\item\label{Fact:Pic:2} Its irreducible components are the varieties parametrizing the line bundles on $C_n$ whose restrictions to $C$ have fixed degree $j$.
\item\label{Fact:Pic:3} There are two short exact sequences of abelian group schemes:
\[
0\to P_{\sO_{C_{n-1}}}\simeq H^1(\cC^{n-1})\to\Pic{C_n}\to\Pic{C_{n-1}}\to 0,\]
\[
0\to \mathbf{P}_n\to\Pic{C_n}\to\Pic{C}\to 0;
\]
where $\mathbf{P}_n\subset\Pic{C_n}$ is the affine subgroup scheme of line bundles with trivial restriction to $C$. Moreover, there exists a filtration of group schemes $0=G_0\subset G_1\subset\dotsb\subset G_{n-1}=\mathbf{P}_n$ such that $G_i/G_{i-1}\simeq H^1(\cC^{i})$.
\item\label{Fact:Pic:4} The component of the identity, which is called the \emph{Jacobian variety}, is not proper if and only if $\mathbf{P}_n\neq 0$. The latter holds, in particular, if $\operatorname{h}^1(C,\cC)\neq 0$, and, thus, if $\deg{\cC}\le g_1-2$, where $g_1$ is the genus of $C$.
\end{enumerate}
\end{fact}
\begin{rmk}\label{Rmk:PicSch}
The first assertion of Fact \ref{Fact:Pic} is true for any projective curve over a field. In the following, we will consider the Picard scheme also of some blowing-ups of a primitive multiple curve which are not necessarily primitive multiple curves themselves.
\end{rmk}

\subsection{Canonical filtrations}\label{SubSecCanFiltr}
Here we introduce two tools, which are fundamental to study a coherent sheaf on a primitive multiple curve: the so-called canonical filtrations. The first one has been introduced by Dr\'{e}zet in \cite[\S 4.1]{DR1}, while the second one has been studied for the first time by Inaba, although in a more general context (cf. \cite[\S 1]{I}), as Dr\'{e}zet himself points out. Our presentation will essentially follow Dr\'{e}zet's article.  

Before starting with their definitions, which are given not only for sheaves but also for finitely generated $\sO_{C_n,P}$-modules, where $P$ is a closed point of $C_n$, it is useful to fix some more conventions.

Throughout the work, if $\sF$ will be coherent sheaf over $C_i$ for some $1\le i\le n-1$, its direct image over $C_n$ will be denoted again by $\sF$ and they will be treated as if they were the same object. All the sheaves studied throughout this paper will be coherent, so this attribute will be omitted. Vector (resp. line) bundle will be used as a synonym of locally free sheaf of finite rank (resp. of rank $1$). 
 
It is also convenient to fix the notation for the local set-up: set $A_i:=\sO_{C_i,P}$, for $1\le i\le n$, where $P$ is a closed point. This implies that $A_n=A_1\otimes_{\bK} \bK[y]/(y^n)$; moreover, $A_i=A_n/(y^i)$, for $1\le i\le n-1$.
The following definitions could more generally be made when $A_1$ is a DVR and $A_n$ a local ring whose nilradical is principal, generated by an element $y$ such that $y^n=0\neq y^{n-1}$, and whose reduced ring is $A_1$ and, in this more general context, the $A_i$ could be defined as $A_n/(y^i)$, but for this work we do not need that generality.\begin{defi}\label{Def:primafiltrcan}\noindent

\begin{enumerate}
\item The \emph{first canonical filtration} of a finitely generated $A_n$-module $M$ is
\[
\{0\}=M_n\subseteq M_{n-1}=y^{n-1}M\subseteq\dotsb\subseteq M_1=yM\subseteq M_0=M.
\]
Equivalently, $M_i$ is equal to $\ker{M_{i-1}\surj M_{i-1}\otimes_{A_n}A_1}$, for $1\le i\le n$.

The \emph{first graded object} of $M$ is $\Gr{1}{M}:=\bigoplus_{i=0}^{n-1}M_i/M_{i+1}$.
\item The \emph{first canonical filtration} of a sheaf of $\sO_{C_n}$-modules $\sF$ is, analogously,
\[
0=\sF_n\subseteq \sF_{n-1}=\sN^{n-1}\sF\subseteq\dotsb\subseteq \sF_1=\sN\sF\subseteq \sF_0=\sF.
\]
Equivalently, $\sF_i$ is equal to $\ker{\sF_{i-1}\surj \sF_{i-1}|_C}$, for $1\le i\le n$.

The \emph{first graded object} of $\sF$ is
\[
\Gr{1}{\sF}=\bigoplus_{i=0}^{n-1}G_i(\sF):=\bigoplus_{i=0}^{n-1}\sF_i/\sF_{i+1}.
\]

The \emph{complete type} of $\sF$ is
\[
\Big(\big(\rk{G_0(\sF)},\dotsc,\rk{G_{n-1}(\sF)}\big),\big(\deg{G_0(\sF)},\dotsc,\deg{G_{n-1}(\sF)}\big)\Big).
\]

\end{enumerate}
\end{defi}
The following remark collects some easy properties of the first canonical filtration.
\begin{rmk}\label{Rmk:PrimaFC}
For any $1\le i\le n-1$, $M_i/M_{i+1}=M_i\otimes_{A_n}A_1$, while $M/M_i\cong M\otimes_{A_n}A_i$ (analogously $\sF_i/\sF_{i+1}=\sF_i|_C$ and $\sF/\sF_i\cong\sF|_{C_i}$).

Again for any $1\le i\le n-1$, $M_i=\{0\}$ (respectively $\sF_i=0$) if and only if $M$ is an $A_i$-module (resp. $\sF$ is a sheaf of $\sO_{C_i}$-modules). $M_i$ (resp. $\sF_i$) is an $A_{n-i}$-module (resp. a sheaf of $\sO_{C_{n-i}}$-modules) with first canonical filtration $\{0\}\subseteq M_{n-1}\subseteq\dotsb\subseteq M_{i+1}\subseteq M_i$ (resp. $0\subseteq \sF_{n-1}\subseteq\dotsb\subseteq \sF_{i+1}\subseteq \sF_i$).

Any morphism of $A_n$-modules (resp. of sheaves over $C_n$) maps the first canonical filtration of the first module (resp. sheaf) to that of the second one.
\end{rmk}

The first canonical filtration (and thus also the related invariants of generalized rank and degree, cf. Definition \ref{Def:GenRkeDeg}) could be defined exactly in the same way for a multiple curve not necessarily primitive.

\begin{defi}\label{Def:secondafiltrcanon}
\noindent
\begin{enumerate}
\item The \emph{second canonical filtration} of a finitely generated $A_n$-module $M$ is
\[
\{0\}=M^{(0)}\subseteq M^{(1)}\subseteq\dotsb\subseteq M^{(n-1)}\subseteq M^{(n)}=M,
\]
where $M^{(i)}:=\{m\in M|y^i m=0\}$, for $1\le i\le n$.

The \emph{second graded object} of $M$ is $\Gr{2}{M}\!:=\!\bigoplus_{i=1}^{n}M^{(i)}/M^{(i-1)}$.

\item The \emph{second canonical filtration} of a sheaf of $\sO_{C_n}$-modules $\sF$ is defined analogously and is denoted by
\[
0=\sF^{(0)}\subseteq\sF^{(1)}\subseteq\dotsb\subseteq\sF^{(n-1)}\subseteq\sF^{(n)}=\sF.
\]
The \emph{second graded object} of $\sF$ is 
\[
\Gr{2}{\sF}=\bigoplus_{i=1}^{n}G^{(i)}(\sF):=\bigoplus_{i=1}^{n}\sF^{(i)}/\sF^{(i-1)}.
\]
\end{enumerate}
\end{defi}
For any $1\le i\le n-1$, it holds that $M_{n-i}\subset M^{(i)}$ (resp. $\sF_{n-i}\subset \sF^{(i)}$) and that $M^{(i)}$ (resp. $\sF^{(i)}$) is an $A_{i}$-module (resp. a sheaf of $\sO_{C_{i}}$-modules) with second canonical filtration $\{0\}\subseteq M^{(1)}\subseteq\dotsb\subseteq M^{(i-1)}\subseteq M^{(i)}$ (resp. $0\subseteq \sF^{(1)}\subseteq\dotsb\subseteq \sF^{(i-1)}\subseteq \sF^{(i)}$). Any morphism of $A_n$-modules (resp. of sheaves over $C_n$) maps the second canonical filtration of the first module (resp. sheaf) to that of the second one.

The following fact collects some properties, proved by Dr\'{e}zet, about the two canonical filtrations.
\begin{fact}\label{Fact:primaesecondfc}
Let $\sF$ be a sheaf of $\sO_{C_n}$-modules. 
\begin{enumerate}

\item\label{Fact:primaesecondfc:1}(\cite[Proposition 3.1(i)]{DR2}) There is a canonical isomorphism between $\sF_i$ and $(\sF/\sF^{(i)})\otimes\sC^{\otimes i}$, for any $0\le i\le n$.

\item\label{Fact:primaesecondfc:2}(\cite[Proposition 3.7]{DR2}) It holds that  $\deg{G^{(i+1)}(\sF)}\!=\!\deg{G_i(\sF)}+\big(\sum_{j=i+1}^{n-1}\rk{G_j(\sF)}-i\rk{G_i(\sF)}\big)\deg{\cC}$ and $\rk{G^{(i+1)}(\sF)}=\rk{G_i(\sF)}$, for any $0\le i\le n-1$,. 

\item\label{Fact:primaesecondfc:3} (\cite[Proposition 3.3 and Corollaire 3.4]{DR2}) Consider the canonical morphism defined by multiplication $\nu:\sF\otimes\sC\to\sF$. Then:
\begin{enumerate}
\item\label{Fact:primaesecondfc:3:1} $\nu$ induces injective morphisms $\lambda_{i,k}=\lambda_{i,k}(\sF):G^{(i+1)}(\sF)\otimes\cC^k\inj G^{(i+1-k)}(\sF)$, for any integers $0<i<n$, $0<k\le i+1$;

\item\label{Fact:primaesecondfc:3:2} $\nu$ induces surjective morphisms $\mu_{j,m}=\mu_{j,m}(\sF):G_{j}(\sF)\otimes\cC^m\surj G_{j+m}(\sF)$, for any non-negative integer $j$ and positive one $m$ such that $j+m\le n-1$.
\end{enumerate}

\end{enumerate} 
\end{fact}

\subsection{Generalized rank and degree}\label{SubSecGenRkDeg}
This subsection is devoted to recall the definitions (cf. \cite[\S\S 4.1.3-4.1.4]{DR1} or \cite[\S 3.2]{DR2}) and main properties of two fundamental invariants of a sheaf on $C_n$. The notation adopted is the same of the previous sections.
\begin{defi}\label{Def:GenRkeDeg}
\noindent
\begin{enumerate}
\item Let $M$ be a finitely-generated $A_n$-module, then its \emph{generalized rank} is $\Rk{M}=\rk{\Gr{1}{M}}=\rk{\Gr{2}{M}}$.
\item Let $\sF$ be a sheaf on $C_n$. Its \emph{generalized rank} $\Rk{\sF}$ is, by definition, $\rk{\Gr{1}{\sF}}$, while its \emph{generalized degree} is $\Deg{\sF}=\deg{\Gr{1}{\sF}}$. This is equivalent, by Fact \ref{Fact:primaesecondfc}\ref{Fact:primaesecondfc:2}, to $\Rk{\sF}=\rk{\Gr{2}{\sF}}$ and $\Deg{\sF}=\deg{\Gr{2}{\sF}}$. 
\end{enumerate}
\end{defi}

The following are some basic but fundamental properties of the generalized rank and degree.
\begin{fact}\label{Fact:PropertiesGenRkeDeg}
\noindent
\begin{enumerate}
\item\label{Fact:PropertiesGenRkeDeg:1} (\cite[\S\S 4.1.3-4.1.4]{DR1}) If $\cF$ is a sheaf over $C$, then $\rk{\cF}=\Rk{\cF}$ and $\deg{\cF}=\Deg{\cF}$.

More generally, if $\sF$ is the direct image of a sheaf of $\sO_{C_i}$-modules, for any $1\le i<n$, its generalized rank and degree as a sheaf on $C_n$ and those as a sheaf on $C_i$ coincide by definition of the first canonical filtration. 

\item\label{Fact:PropertiesGenRkeDeg:2} (\cite[\S\S 4.1.3-4.1.4]{DR1}). Let $\sF$ be a locally free sheaf of rank $m$ over $C_n$, then $\Rk{\sF}=nm=n\rk{\sF|_C}$ and $\Deg{\sF}=n\deg{\sF|_C}+(n(n-1)/2)m\deg{\cC}$. In particular, any line bundle has generalized rank $n$ and $\Deg{\sO_{C_n}}=(n(n-1)/2)\deg{\cC}$.
 
\item\label{Fact:PropertiesGenRkeDeg:3} (\cite[Th\'{e}or\`{e}me 4.2.1]{DR1})
Let $\sF$ be a sheaf of $\sO_{C_n}$-modules. It verifies the so-called \emph{generalized Riemann-Roch theorem}, i.e. $\chi(\sF)=\Deg{\sF}+\Rk{\sF}\chi(\sO_C)$.

\item\label{Fact:PropertiesGenRkeDeg:4} (\cite[\S 4.2.2]{DR1})
Let $\sF$ be a sheaf and $\sO_{C_n}(1)$ be a very ample line bundle on $C_n$, let $\sO_C(1)$ be its restriction to $C$ and $d=\deg{\sO_C(1)}$. Then the Hilbert polynomial of $\sF$ with respect to $\sO_{C_n}(1)$ is
\begin{equation}\label{Eq:HilbPol}
P_{\sF}(T)=\Deg{\sF}+\Rk{\sF}\chi(\sO_C)+\Rk{\sF}dT.
\end{equation}

\item\label{Fact:PropertiesGenRkeDeg:5} (\cite[Corollaire 4.3.2]{DR1}) The generalized rank and degree are additive, i.e.:
\begin{enumerate}

\item\label{Fact:PropertiesGenRkeDeg:5:1}  if $0\to M'\to M\to M''\to 0$ is a short exact sequence of finitely generated $A_n$-modules, then $\Rk{M}=\Rk{M'}+\Rk{M''}$;

\item\label{Fact:PropertiesGenRkeDeg:5:2} if $0\to \sF'\to \sF\to \sF''\to 0$ is a short exact sequence of sheaves on $C_n$, then $\Rk{\sF}=\Rk{\sF'}+\Rk{\sF''}$ and $\Deg{\sF}=\Deg{\sF'}+\Deg{\sF''}$.
\end{enumerate}

\item\label{Fact:PropertiesGenRkeDeg:6} (\cite[Proposition 4.3.3]{DR1})
\begin{enumerate}

\item\label{Fact:PropertiesGenRkeDeg:6:1} The generalized rank of finitely generated $A_n$-modules is invariant by deformation.

\item\label{Fact:PropertiesGenRkeDeg:6:2} The generalized rank and degree of sheaves over $C_n$ are invariant by deformation.
\end{enumerate}
\end{enumerate}
\end{fact}

\begin{rmk}\label{Rmk:GenRkDeg}
\noindent
\begin{enumerate}
\item\label{Rmk:GenRkDeg:1} Fact \ref{Fact:PropertiesGenRkeDeg} \ref{Fact:PropertiesGenRkeDeg:1} is one fundamental reason for which it is possible to do not distinguish between a sheaf over $C_i$ and its direct image over $C_n$, for any $1\le i\le n-1$. It is also a significant reason to use generalized rank and degree instead of the usual ones, which do not have this very useful property.

\item\label{Rmk:GenRkDeg:2} It is possible to give also another characterization of  generalized rank and degree of a sheaf $\sF$ over $C_n$ without making use of the canonical filtrations: the generalized rank of $\sF$ can be seen as its generic length, i.e. the length of the $\sO_{C_n,\eta}$-module $\sF_\eta$, where $\eta$ is the generic point of $C_n$, while its generalized degree could be defined also as $\chi(\sF)-\Rk{\sF}\chi(\sO_C)$.
The equivalence of these characterizations to the original definitions is almost immediate (for the generalized degree it has to be used the additivity of the Euler characteristic, which implies that $\chi(\sF)=\chi(\Gr{1}{\sF})$.

\end{enumerate}
\end{rmk}
Now we will describe the relation of generalized rank and degree with the usual rank and degree. Indeed, the latter can be defined also in this context and are often used for sheaves on ribbons (as in \cite{CK} or, at least the degree, in \cite{EG}). First of all, we need to recall the classical definitions: if $\sF$ is a sheaf of $\sO_{C_n}$-modules then its rank, $\rk{\sF}$, and its degree, $\deg{\sF}$, are the rational numbers for which its Hilbert polynomial with respect to a fixed very ample line bundle has the form
\begin{equation}\label{Eq:HilbPol2}
P_{\sF}(T)=\deg{\sF}+\rk{\sF}\chi(\sO_{C_n})+nd\rk{\sF}T,
\end{equation} 
where $d$ is as in Fact \ref{Fact:PropertiesGenRkeDeg}\ref{Fact:PropertiesGenRkeDeg:4}, (for this definition, cf., e.g., \cite[Definition 3.7]{HP}). Observe that if $\sF$ is a sheaf of $\sO_{C_i}$-modules, then its rank and degree are not equal to those of its direct image over $C_n$.

The next lemma, which is implied by formulae \eqref{Eq:HilbPol} and \eqref{Eq:HilbPol2}, compares generalized rank and degree with the usual ones:
\begin{lemma}\label{Lem:ComprkdegandRkDeg}
Let $\sF$ be a sheaf of $\sO_{C_n}$-modules, then $\Rk{\sF}=n\rk{\sF}$ and $\Deg{\sF}=\deg{\sF}+\rk{\sF}\Deg{\sO_{C_n}}=\deg{\sF}+\rk{\sF}\frac{n(n-1)}{2}\deg{\cC}$.
\end{lemma}  
\begin{cor}\label{Cor:DegandTensProd}
Let $\sF$ be a sheaf on $C_n$ of generalized rank $R$ and $\sE$ a vector bundle of rank $m$ (i.e. generalized rank $nm$) on $C_n$. Then
\begin{equation}\label{Eq:DegandTensProd}
\Deg{\sF\otimes\sE}=\frac{R}{n}\Deg{\sE}+m\Deg{\sF}-\frac{Rm(n-1)}{2}\deg{\cC}.
\end{equation}
\end{cor}
\begin{proof}
It follows from the above Lemma and from \cite[Tag 0AYV]{SP}, asserting that $\chi(\sF\otimes\sE)=\rk{\sF}\deg{\sE}+\rk{\sE}\chi(\sF)$, in a wider context, i.e. if $\sF$ is a sheaf and $\sE$ a vector bundle on a proper irreducible curve over a field.
\end{proof}

\subsection{Purity and duality}\label{SubSec:PurandDual}
The next step is to introduce the key (equivalent in our case) notions of pure, torsion-free and reflexive sheaves. The distinction between pure and torsion-free is taken from \cite[Definition 2.1]{CK}; Dr\'{e}zet speaks only of reflexive and torsion-free sheaves (\emph{faisceaux sans torsion} in French), but he defines the latter as Chen and Kass define pure ones (cf. \cite[\S 3.3]{DR2}).

Let us begin with pure and torsion-free sheaves:
\begin{defi}\label{Def:pureandtorsfree}
Let $\sF$ be a sheaf on $X$, its \emph{dimension}, $\mathrm{d}(\sF)$, is the dimension of its support.
A sheaf $\sF$ on $X$ is \emph{pure} if it has dimension $1$ and $\mathrm{d}(\sG)=1$ for any non-zero subsheaf $\sG\subset\sF$.

Let $U$ be an open subscheme of $X$, a regular function $f\in H^0(U,\sO_X)$ is a \emph{nonzerodivisor on} $\sF$ if the multiplication map $f \cdot \_ :\sF|_U\to\sF|_U$ is injective and the sheaf $\sF$ is \emph{torsion-free} if every nonzerodivisor on $\sO_X$ is a nonzerodivisor also on $\sF$.
\end{defi}
\begin{rmk}\label{Rmk:Pure}
Our definition of pure is not completely equal to \cite[Definition 2.1]{CK}: there, as often in literature, it is only required that the dimension of any proper subsheaf equals that of the sheaf; hence, any sheaf of dimension $0$ would be pure according to that definition, but we are not interested in them.
\end{rmk}
The following result is the extension of \cite[Lemma 2.2]{CK} from the case of ribbons to the case of primitive multiple curves. Also the proof is almost identical to that of the place cited, which extends verbatim to our case (it holds also in wider generality, namely at least for sheaves over any irreducible algebraic scheme of dimension $1$).
\begin{lemma}\label{Lem:p=tf}
Let $\sF$ be a sheaf on a primitive multiple curve $X$. Then $\sF$ is pure if and only if it is torsion-free. 
\end{lemma}
\begin{proof}
By definition, $\sF$ is not pure if and only if there exists a non-zero subsheaf of $\sF$ with finite support. This is equivalent to the existence of an open affine subscheme $U\subset X$ and a non-zero $g\in H^0(U,\sF)$ with finite support. Equivalently, there exist an open affine subscheme $U\subset X$ and a non-zero $g\in H^0(U,\sF)$ such that $\operatorname{ann}(g)\not\subset\sN|_U$. This is equivalent to the fact that there exist an open affine subscheme $U\subset X$, a nonzerodivisor $f\in H^0(U,\sO_X)$ and a non-zero $g\in H^0(U,\sF)$ such that $fg=0$. The last assertion means that $\sF$ is not torsion-free, by definition.
\end{proof}
In order to introduce reflexiveness, we need first to recall the notion of dual of a sheaf on a primitive multiple curve.
\begin{defi}\label{Def:Dual}
Let $\sF$ be a sheaf on $C_n$. Its \emph{dual} is $\sF^{\vee_n}=\sF^\vee=\un{\operatorname{Hom}}(\sF,\sO_{C_n})$.

The sheaf $\sF$ is \emph{reflexive} if the canonical morphism $\sF\to\sF^{\vee\vee}$ is an isomorphism.
\end{defi} 
\begin{rmk}\label{Rmk:Dual}
If $\sF$ is a sheaf on $C_i$, for $1\le i<n$, then $\sF^{\vee_i}\neq\sF^{\vee_n}$. But there is a canonical isomorphism $\sF^{\vee_n}\simeq\sF^{\vee_i}\otimes\sN^{n-i}$ (this is \cite[Lemme 4.1]{DR2}).
\end{rmk}
The following fact collects some properties of duality and of reflexive sheaves.
\begin{fact}\label{Fact:Dual} Let $\sF$ be a sheaf over $C_n$. 
\begin{enumerate}
\item\label{Fact:Dual:1}(\cite[Proposition 3.8 and Th\'{e}or\`{e}me 4.4]{DR2}) The following are equivalent:
\begin{enumerate}
\item $\sF$ is torsion-free;
\item $\sF^{(1)}=G^{(1)}(\sF)$ is a vector bundle on $C$;
\item $\sF$ is reflexive;
\item $\SExt{1}{\sO_{C_n}}(\sF,\sO_{C_n})=0$.
\end{enumerate}
Moreover, if the above conditions hold, $\SExt{i}{\sO_{C_n}}(\sF,\sO_{C_n})=0$ for any $i\ge 1$ and $G^{(j)}(\sF)$ is a vector bundle on $C$ for $1\le j\le n$. 
\item\label{Fact:Dual:2} (\cite[Proposition 4.2]{DR2}) For any $1\le i<n$, $(\sF^\vee)^{(i)}=(\sF|_{C_i})^{\vee}$.
\item\label{Fact:Dual:3}(\cite[Proposition 4.4.1]{DR4})$\Rk{\sF^\vee}\!=\!\Rk{\sF}$and $\Deg{\sF^\vee}\!=\!-\!\Deg{\sF}\!+\Rk{\sF}(n-1)\deg{\cC}+\operatorname{h}^0(\sT(\sF))$, where $\sT(\sF)$ is the torsion subsheaf of $\sF$, i.e. its greatest subsheaf with finite support.
\item\label{Fact:Dual:4} Assume, moreover, that $\sF$ is torsion-free. Then, for any $1\le i< n$:
\begin{enumerate}
\item\label{Fact:Dual:4:1}(\cite[Proposition 4.3.1.1]{DR4}) There is a canonical isomorphism between $\SExt{1}{\sO_{C_n}}(\sT(\sF|_{C_i}), \sO_{C_n})\otimes\sC^i$ and $\sT(\sF^{\vee}|_{C_i})$, where $\sT(\sF^{\vee}|_{C_i})$ and $\sT(\sF|_{C_i})$ are the torsion subsheaves of, respectively, $\sF^{\vee}|_{C_i}$ and $\sF|_{C_i}$. 
\item\label{Fact:Dual:4:2}(\cite[Proposition 4.3.1.2]{DR4}) There is a canonical isomorphism between $\ker{\sF\surj(\sF|_{C_i})^{\vee\vee}}$ and $(\sF^{\vee})_i\otimes\sC^{-i}$.
\end{enumerate}
\item\label{Fact:Dual:5}(\cite[Corollaire 4.5]{DR2}) If $0\to\sE\to\sF\to\sG\to 0$ is a short exact sequence of sheaves on $C_n$ with $\sG$ pure, then also the dual sequence $0\to\sG^\vee\to\sF^\vee\to\sE^\vee\to 0$ is exact.
\end{enumerate}
\end{fact}
\begin{rmk}\label{Rmk:Fact:Dual:5}
The hypothesis of the last point of the Fact is a bit weaker than that of the place cited, where it is required that also $\sE$ and $\sF$ are pure, but the only significant point for the proof is that $\SExt{1}{\sO_{C_n}}(\sG,\sO_{C_n})$ vanishes. So the assertion remains true also under our hypothesis. 
\end{rmk}
There is a special type of pure sheaves which plays a major role in the theory of sheaves over a primitive multiple curve: the so-called \emph{quasi locally free} sheaves.
\begin{defi}\label{Def:qll} (Cf. \cite[\S 5.1]{DR1}.)
A finitely-generated $A_n$-module M is \emph{quasi free} if there exist non-negative integers $m_1,\dotsc,\,m_n$ such that $M\cong\bigoplus_{i=1}^nA_i^{\oplus m_i}$. The $n$-tuple $(m_1,\dotsc,\,m_n)$ is called the \emph{type} of $M$.

Let $\sF$ be a sheaf on $C_n$; it is \emph{quasi locally free in a closed point} $P$ if there exists an open neighbourhood $U$ of $P$ and non-negative integers $m_1,\dotsc,\,m_n$ such that $\sF_Q$ is quasi free of type $(m_1,\dotsc,\,m_n)$ for any $Q\in U$. It is \emph{quasi locally free} if it is such in any closed point. 
\end{defi}
The following fact contains some significant results.
\begin{fact}\label{Fact:qll} Let $\sF$ be a sheaf on $\sO_{C_n}$ and let $P$ be a closed point of $C$.
\begin{enumerate}
\item\label{Fact:qll:1}(\cite[Th\'{e}or\`{e}me 3.9 and Corollaire 3.10]{DR2})
The following are equivalent:
\begin{enumerate}
\item $\sF$ is quasi locally free (resp. quasi locally free in $P$);
\item for $0\le i<n$, $G_i(\sF)$ is a vector bundle over $C$ (resp. is free in $P$);
\end{enumerate}
\item\label{Fact:qll:2}(\cite[Th\'{e}or\`{e}me 5.1.6]{DR1}) $\sF$ is generically quasi locally free, i.e. there exists a non-empty open $U$ of $C_n$ such that $\sF$ is quasi locally free in each point of $U$.
\end{enumerate}
\end{fact}
The last point of the above fact allows to give the following definition:
\begin{defi}\label{Def:type}
The \emph{type} of a sheaf $\sF$ on $C_n$ is the $n$-tuple of non-negative integers $(m_1,\dotsc,m_n)$ such that $\sF_\eta\cong \bigoplus_{i=1}^{n}\sO_{C_i,\eta}^{\oplus m_i}$, where, as usual, $\eta$ is the generic point of $C_n$.  
\end{defi}

Within quasi locally free sheaves there are those \emph{of rigid type}, studied in \cite{DR2} and \cite{DR4}:
\begin{defi}\label{Def:qllrt}
A sheaf $\sF$ on $C_n$ is said to be \emph{quasi locally free of rigid type} if there exist two non-negative integers $a>0$ and $j<n$ such that $\sF$ is locally isomorphic to $\sO_{C_n}^{\oplus a}\oplus\sO_{C_j}$.
\end{defi}
Observe that these include vector bundles (they are the quasi locally free sheaves of rigid type with $j=0$).
They are relevant because being quasi locally free of rigid type is an open condition in flat families of sheaves on $C_n$, as the name suggests (see \cite[Proposition 6.9]{DR2}). They are the only kind of pure sheaves on $C_n$ such that there are some results in literature (precisely \cite[Proposition 6.12]{DR2} and \cite[Th\'{e}or\`{e}me 5.3.3]{DR4}) about loci containing them in the moduli space of semistable sheaves.

\subsection{Semistability}\label{SubSecStab}
The last argument that we quickly treat in this section is semistability.
First of all, it is necessary to recall how it is defined on a primitive multiple curve (cf. \cite[\S 1.1]{DR1}).
\begin{defi}\label{Def:slopestability}
Let $\sF$ be a pure sheaf on $C_n$, then its \emph{slope} is $\mu(\sF)=\Deg{\sF}/\Rk{\sF}$. The definition of semistability of $\sF$ is the usual definition of (slope-)semistability: $\sF$ is (slope-)\emph{semistable} if for any non-trivial subsheaf $\sE$ it holds that $\mu(\sE)\le\mu(\sF)$ or, equivalently, if for any non-trivial pure quotient $\sG$ it holds that $\mu(\sG)\ge\mu(\sF)$. If the inequality is always strict, $\sF$ is said to be \emph{stable}.
\end{defi}
\begin{rmk}\label{Rmk:stability}
\noindent
\begin{enumerate}

\item\label{Rmk:stability:2} Thanks to the description of the Hilbert polynomial given by formula \eqref{Eq:HilbPol}, on a primitive multiple curve slope semistability coincides with Gieseker one, which considers the reduced Hilbert polynomial instead of the slope. So the latter results to be independent of the polarization, as in the case of vector bundles on a smooth projective curve.

\item\label{Rmk:stability:1} The equivalence of the condition about subsheaves and pure quotients is almost trivial and is a well-known property of semistability (cf. e.g. \cite[Proposition 1.2.6]{HL}). 

\item\label{Rmk:stability:3} It is evident, by definition, that if $\sF$ is a semistable sheaf on $C_i$, for some $1\le i\le n-1$, then it is semistable also on $C_n$.
\item\label{Rmk:stability:4} It is possible to verify (cf. e.g. \cite[\S 1.2]{DR4}) that there are interesting (i.e. different from direct images of stable vector bundles on $C$) stable sheaves on $C_n$ only if $\deg{\cC}<0$. It is quite easy to check the assertion for a vector bundle $\sF$ on $C_n$: by Fact \ref{Fact:PropertiesGenRkeDeg} \ref{Fact:PropertiesGenRkeDeg:2} it holds that $\mu(\sF)=\mu(\sF|_C)+((n-1)/2)\deg{\cC}$; hence it can be stable only if $\deg{\cC}<0$. Under this assumption all the line bundles on $C_n$ are stable (it is almost trivial, but it is also a consequence of Theorem \ref{Thm:ssinequ}, which, moreover, confirms the necessity of $\deg{\cC}<0$ in order to have stable generalized line bundles).  
\end{enumerate}
\end{rmk}
The following easy lemma concludes this quick overview about semistability.
\begin{lemma}\label{Lem:StabandDual}
Let $\sF$ be a pure sheaf on $C_n$. It is (semi)stable if and only if $\sF^\vee$ is (semi)stable.
\end{lemma}
\begin{proof}
The proof is done only in the case of semistability, because that for stabilty is essentially the same.

By the equivalence of purity and reflexiveness, it is sufficient to show that $\sF$ semistable implies $\sF^\vee$ semistable.

So, assume $\sF$ semistable and let $\sG$ be any pure quotient of $\sF^\vee$. Hence, $\sG^\vee$ is a subsheaf of $\sF$ by Fact \ref{Fact:Dual}\ref{Fact:Dual:5}. Thus,
$\mu(\sG)=-\mu(\sG^\vee)+(n-1)\deg{\cC}\ge-\mu(\sF)+(n-1)\deg{\cC}=\mu(\sF^\vee)-(n-1)\deg{\cC}+(n-1)\deg{\cC}=\mu(\sF^\vee)$, 
where the various equalities follow from Fact \ref{Fact:Dual}\ref{Fact:Dual:3} while the inequality is due to the semistabilty of $\sF$. 
\end{proof}
  
\section{Generalized line bundles}\label{SecFirstProp} 

This section is devoted to describe various properties of generalized line bundles on a primitive multiple curve, generalizing, as far as possible, \cite[\S 2]{CK} from multiplicity $2$ to arbitrary one. It is divided into two subsections. The first is more generally about pure sheaves of generalized rank $n$ on a primitive multiple curve of multiplicity $n$ while the second is more specific about generalized line bundles. 

Throughout this section $X$ will be a primitive multiple curve of multiplicity $n$.

\subsection{Pure sheaves of generalized rank n}\label{SubSecgenrkn}

All the results of this subsection are completely trivial for $n=1$, while the case $n=2$ is that treated in \cite{CK}. For $n\ge 3$ the properties described and also their proves are straightforward extensions of those by Chen and Kass.

First of all, we need to define properly the main object of study of this work.
\begin{defi}\label{Def:GenLB} 
A \emph{generalized line bundle} is a pure sheaf $\sF$ on $X$ such that $\sF_{\eta}$ is isomorphic to $\sO_{X,\eta}$, where $\eta$ is the generic point of $X$. 
\end{defi}
 \begin{rmk}
 Let $\sF$ be a generalized line bundle on $X$. By definition $\Rk{\sF}=n$ (or, equivalently, $\rk{\sF}=1$).
 
 According to my knowledge this definition of generalized line bundle is new for $n\ge 3$, but it is an obvious extension of the notion for ribbons (i.e. $n=2$). Furthermore, as in case $n=2$ (cf. \cite[beggining of page 759]{EG}), generalized line bundles coincide, by \cite[Proposition 2.12]{H}, with generalized divisors which have been introduced for Gorenstein schemes, including, in particular, primitive multiple curves, by Hartshorne in the cited article.
 \end{rmk}
Following Chen and Kass, we prove two easy lemmata about pure sheaves on $X$. The first one extends to the general case \cite[Lemma 2.3]{CK}.
\begin{lemma}\label{Lem:endpure}
If $\sF$ is a pure sheaf on $X$, then $\SEnd{\sF}$ is pure, too. 
\end{lemma}
\begin{proof}
The proof of \cite{CK} extends almost verbatim. We omit the trivial proof for sheaves of dimension $0$, because, as pointed out in Remark \ref{Rmk:Pure}, for us pure is a synonym of pure of dimension $1$.

It suffices to prove that if $\varphi$ is an element of $\SEnd{\sF}_x$ annihilated by $\mathfrak{m}_x\subset\sO_{X,x}$, where $x \in X$ is a closed point and $\mathfrak{m}_x$ is the maximal ideal of $\sO_{X,x}$, then $\varphi=0$. Indeed, given such a $\varphi$, it holds that for any $s\in\sF_x$ its image $\varphi(s)\in\sF_x$ is annihilated by $\mathfrak{m}_x$ and so $\varphi(s)=0$, because $\sF$ is pure (or, equivalently, torsion-free, thanks to Lemma \ref{Lem:p=tf}).   
\end{proof}
The next lemma is a generalization of \cite[Lemma 2.4]{CK} from the case of ribbons to any multiplicity.
\begin{lemma}\label{Lem:kernels}
Let $\sF$ be a pure sheaf over $X$. The kernel of the natural morphism $\varphi:\sO_X\to\SEnd{\sF}$ is equal to $\sN^i$ for some $1\le i\le n$. Hence the schematic support of $\sF$ is $C_i$.
\end{lemma}
\begin{proof}
Let $\sK=\ker{\varphi}$. By definition, there is an injection $\sO_X/\sK\inj\SEnd{\sF}$ and it is clear that $\sO_X/\sK\neq 0$ (e.g. any nonzerodivisor constant defines a non-zero endomorphism, being $\sF$ pure). Because $\SEnd{\sF}$ is pure (by Lemma \ref{Lem:endpure}), $\sO_X/\sK$ is such too. By primary decomposition, $\sK$ is contained in $\sN$ (indeed, over every open affine $U$, the prime ideals associated to $\sK(U)$ in $\sO_X(U)$ must have height zero, but the only prime of $\sO_X(U)$ with this property is $\sN(U)$).

Let $\eta$ be, as usual, the generic point of $X$. By definition of primitive multiple curve, $\sO_{X,\eta}$ is isomorphic to $\sO_{C,\eta'}[y]/(y^n)$ and $\sO_{C,\eta'}$, where $\eta'$ is the generic point of $C$, is a field. The only subideals of $\sN_\eta$ are its powers, i.e. $\sN_\eta^j$ with $1\le j\le n$ (with $\sN_\eta^n=0$) and thus $\sK_\eta=\sN_\eta^i$ for an $i\in\{1,\dotsc,n\}$. At this point the conclusion follows from the purity of $\sN^i/\sK$ (it is a subsheaf of $\sO_X/\sK$, which is pure by the above argument): the fact that $\sN_\eta^i/\sK_\eta=0$ implies that $\sN^i/\sK$ has finite support and so it must be zero, i.e. $\sK=\sN^i$.   
\end{proof}
\begin{rmk}
In the cited lemma of \cite{CK} there is the adjunctive hypothesis of generic length $2$ (equivalent to generalized rank $2$) of the sheaves involved (and that would correspond to generalized rank $n$), but it is superfluous. 
\end{rmk}
The previous lemma allows to get another characterization of generalized line bundles. The case of multiplicity $2$ is \cite[Lemma 2.5]{CK}.
\begin{cor}\label{Cor:genlb2}
Let $\sF$ be a pure sheaf of generalized rank $n$ on $X$. Then $\sF$ is a generalized line bundle if and only if the morphism $\varphi:\sO_X\to\SEnd{\sF}$ is injective.
\end{cor}
\begin{proof}
By the previous Lemma, it is sufficient to show that, if $\sF$ is pure, $\sF_\eta\cong\sO_{X,\eta}$, where, as usual, $\eta$ is the generic point of $X$, is equivalent to $\varphi$ injective.

Assume that $\sF$ is a generalized line bundle. The morphism $\varphi$ gives rise to the following commutative diagram:
\[
\begin{tikzcd}
\sO_X \ar[d, hook] \ar[r, "\varphi"] & \SEnd{\sF}_{\phantom{\eta}}\ar[d, hook]\\
\sO_{X,\eta} \ar[r, "\varphi_\eta"]  & \SEnd{\sF}_\eta,
\end{tikzcd}
\]
where the vertical arrows are injective because both $\sO_X$ and $\SEnd{\sF}$ are pure (the latter by Lemma \ref{Lem:endpure}). Moreover, being $\sF$ a generalized line bundle, $\varphi_\eta$ is an isomorphism, hence $\varphi$ is injective.

Conversely, assume that $\varphi$ is injective and let $y_\eta$ be a generator of the generic stalk $\sN_\eta$. By hypothesis, multiplication by $y_\eta^{n-1}$ on $\sF_\eta$ is not the zero map. Hence we can choose $s_0\in\sF_\eta$ such that $y_\eta^{n-1}s_0\neq 0$; let us consider the morphism $\psi:\sO_{X,\eta}\to\sF_\eta$ defined by $\psi(f)=fs_0$: it is the desired isomorphism. Indeed, $\ker{\psi}=\operatorname{Ann}(s_0)$ is a submodule of $\sN_\eta$, being $\sF$ pure. Moreover, the fact that $y_\eta^{n-1}s_0\neq 0$ implies that $\ker{\psi}\subsetneq\sN_\eta^{n-1}$ and thus $\ker{\psi}=0$. Surjectivity follows from the fact that $\sF_\eta$ and the submodule generated by $s_0$ have both length $n$.      
\end{proof}
These lemmata and the last corollary imply the following classification of pure sheaves of generalized rank $n$ on $X$, which extends \cite[\S 8.2]{DR1} and \cite[Proposition 2.6]{CK}.
\begin{prop}\label{Prop:granknsheaves}
Let $\sF$ be a pure sheaf of generalized rank $n$ on $X$. Then $\sF$ is either a generalized line bundle or the direct image of a pure sheaf of generalized rank $n$ defined on $C_{n-1}$ under the inclusion $C_{n-1}\subset X$ (in the following such sheaves were called sheaves of generalized rank $n$ defined on $C_{n-1}$). 
\end{prop} 
\begin{proof}
Consider again the natural morphism $\varphi:\sO_X\to\SEnd{\sF}$. If $\varphi$ is injective, $\sF$ is a generalized line bundle by Corollary \ref{Cor:genlb2}. Otherwise $\sN^{n-1}\subseteq\ker{\varphi}$ by Lemma \ref{Lem:kernels} and thus $\sF$ can be seen as an $\sO_X/\sN^{n-1}$-module, i.e. an $\sO_{C_{n-1}}$-module. 
\end{proof}
\begin{rmk}
When $n\ge 3$ this result is quite vague with respect to that for ribbons because it does not give a precise classification of pure sheaves of generalized rank $n$ defined on $C_{n-1}$, which include various different kinds of sheaves, from vector bundles of rank $n$ defined on $C$ to sheaves generically of the form $\sO_{C_{n-1}}\oplus\sO_C$. In general, they are pure sheaves generically of the form $\bigoplus_{i=1}^{n-1} \sO_{C_i}^{\oplus a_i}$, with the $a_i$ non-negative integers  such that $\sum_{i=1}^{n-1}ia_i=n$ (cf. Fact \ref{Fact:qll}\ref{Fact:qll:2}). This will be a huge complication in the study of the moduli space of generalized rank $n$ sheaves on $X$. 
\end{rmk}

\subsection{Generalities about generalized line bundles}\label{SubSecflgen}

In this subsection we will study generalized line bundles, so throughout it $\sF$ will denote a generalized line bundle on $X$.

The following lemma extends to morphisms between generalized line bundles a well-known property of those between line bundles.
\begin{lemma}\label{Lem:surj=iso}
If a morphism between generalized line bundles on $X$ is surjective, then it is an isomorphism. 
\end{lemma}
\begin{proof}
Let $\pi:\sF\surj\sG$ be a surjective morphism between generalized line bundles. It is evident that $\pi_\eta:\sF_\eta\to\sG_\eta$ is an isomorphism, hence $\ker{\pi}$ is generically zero and so has finite support. Therefore, the purity of $\sF$ implies that $\ker{\pi}=0$.
\end{proof} 

The two canonical filtrations of $\sF$ will play a crucial role in the following; in particular, we will use them to define other sheaves associated to $\sF$ and some invariants. So we need to study them in some detail.

First of all, observe that, being $\sF$ a generalized line bundle, all the containments in the two filtrations are strict, $\sF_i$ is a generalized line bundle defined on $C_{n-i}$ while $\sF^{(i)}$ is a generalized line bundle defined on $C_i$.
 However $\sF/\sF_i=\sF|_{C_i}$ is not necessarily a generalized line bundle on $C_i$, for any $1\le i\le n-1$: in general, it has a nonzero torsion subsheaf; on the other side, $\sF/\sF^{(j)}$, being isomorphic, up to tensor product with a line bundle, to $\sF_j$ (by Fact \ref{Fact:primaesecondfc}\ref{Fact:primaesecondfc:1}), is a generalized line bundle on $C_{n-j}$, for any $1\le j \le n-1$. This suggests to study the relation between these quotients and, in order to do that, it is useful to introduce some definitions and notations.
\begin{defi}\label{Def:purequotandtp}
For any $1\le i\le n$, the \emph{$i$-th pure quotient} of $\sF$ is $\bsF{i}:=(\sF|_{C_i})^{\vee\vee}$, while the kernel of the natural morphism $\sF|_{C_i}\surj\bsF{i}$ is denoted $\sT_i(\sF)$, or simply $\sT_i$ if it is clear which is the generalized line bundle involved. In order to avoid any risk of confusion between $\bsF{i}$ and $\sF_i$ in the following the latter will be always denoted by $\sN^i\sF$.   
\end{defi}
It makes sense to call $\bsF{i}$ \emph{the} $i$-th pure quotient of $\sF$ by the next lemma, asserting that it is the only pure quotient of $\sF$ supported exactly on $C_i$ (in the sense that it is an $\sO_{C_i}$-module but it does not have a structure of $\sO_{C_{i-1}}$-module). Clearly, for any generalized line bundle $\sF$, it holds that $\bsF{n}=\sF$ and $\sT_n=0$.

In the case of ribbons, the following lemma has not been explicitly enunciated in \cite{CK} but it is contained in the proof of \cite[Lemma 3.2]{CK}.
\begin{lemma}\label{Lem:purequot}
Let $\sG$ be a pure sheaf on $X$ and let $q:\sF\twoheadrightarrow\sG$ be a surjective morphism, then $\sG=\bsF{i}$ for some $1\le i\le n$.
\end{lemma}
\begin{proof}
There are two different cases to be discussed according to the generalized rank of $\sG$, i.e. $\Rk{\sG}=n$ and $\Rk{\sG}<n$. 

If $\Rk{\sG}=n$, it is sufficient to show that $\sG$ is a generalized line bundle on $X$ because this implies that $q$ is an isomorphism by Lemma \ref{Lem:surj=iso} (and thus, $\sG=\sF=\bsF{n}$). Hence, by Lemma \ref{Prop:granknsheaves} it is sufficient to prove that $\sG$ is not the direct image of a pure sheaf of $\sO_{C_{n-1}}$-modules of generalized rank $n$. This is the case because $\sG_\eta$ can be generated as $\sO_{X,\eta}$-module by a single element (the image of a generator of $\sF_\eta$) and the generic stalk of a $\sO_{C_{n-1}}$-module of generalized rank $n$ does not have this property (indeed, it is of the form $\bigoplus_{i=1}^{n-1}\sO_{C_{n-i}}^{\oplus a_i}$ with the $a_i$ non-negative integers such that $\sum_{i=1}^{n-1}ia_i=n$).

Now assume $\Rk{\sG}=r<n$. The morphism $q:\sF\twoheadrightarrow\sG$ induces an epimorphism $q_{\eta}:\sF_{\eta}\cong\sO_{X,\eta}\twoheadrightarrow\sG_\eta$, thus $\sG_\eta$ can be generated by a single element, say $s_0$. Let $\sK=\ker{\sO_X\to\SEnd{\sG}}$. By Lemma \ref{Lem:kernels}, $\sK\simeq\sN^i$ for some $1\le i\le n$. Moreover, $\sK_\eta=\operatorname{Ann}(s_0)$ and $\sG_\eta\cong\sO_{X,\eta}s_0\cong\sO_{X,\eta}/\sK_\eta$. The fact that the length of $\sG_\eta$ is $r$ implies that $\sK_\eta$ is isomorphic to $\sN_\eta^{r}$. Hence, $\sK=\sN^{r}$ and it follows that $\sG$ is a pure $\sO_X/\sN^r$-module, i.e. a $\sO_{C_r}$-module, or rather a generalized line bundle on $C_r$. Moreover, $q$ can be factorized as $\sF\to\sF|_{C_r}\surj\bsF{r}\overset{\bar{q}}{\twoheadrightarrow}\sG$. So $\bar{q}$ is a surjective morphism between generalized line bundles on $C_r$, hence an isomorphism, again by Lemma \ref{Lem:surj=iso}.          
\end{proof}

\begin{rmk}\label{Rmk:purequot}
The above lemma gives another useful characterization of the $i$-th pure quotient: it is isomorphic to $\sF/\sF^{(n-i)}$, because the latter is a pure quotient of $\sF$ on $C_i$.
\end{rmk}

A priori $\sT_i$ has a structure of $\sO_{C_i}$-module, but thanks to this remark it is possible to say something more, for $i>n/2$.
\begin{lemma}\label{Lem:supportofTi}
For any $1\le i\le n-1$, the torsion sheaf $\sT_i$ is isomorphic to $\sF^{(n-i)}/\sN^i\sF$; in particular, if $n/2<i<n$, it is an $\sO_{C_{n-i}}$-module.
\end{lemma}
\begin{proof}
The second assertion is an immediate consequence of the first one.

The following diagram is exact by definition of the various sheaves involved and by Remark \ref{Rmk:purequot}:
\[
\begin{tikzcd}
 & & &0 \ar[d] &\\
& 0 \ar[d] & &\sT_i \ar[d] &\\
0 \ar[r] & \sN^i\sF \ar[d] \ar[r] &\sF \ar[r] \ar[d, equal] &\sF|_{C_i} \ar[r] \ar[d] & 0\\
0 \ar[r] & \sF^{(n-i)} \ar[d] \ar[r] &\sF \ar[r] &\bsF{i} \ar[r] \ar[d] & 0\\
& \sF^{(n-i)}/\sN^i\sF \ar[d] & &0 &\\
&0 & & &     
\end{tikzcd}
\]
The first assertion follows from it by snake's lemma.
\end{proof}

Now we investigate when a generalized line bundle is a line bundle.

\begin{prop}\label{Prop:fliffT1=0}
A generalized line bundle on $X$ is a line bundle if and only if its restriction to $C$ is a line bundle. 
\end{prop}
\begin{proof}
The necessity is obvious; hence the only interesting part is sufficiency.

This proof proceeds by induction on $n$, the base is the completely trivial case $n=1$, although also the case $n=2$ is already known (cf., e.g, the proof of \cite[Theorem 1.1]{EG}).
Assume that the statement is true for $n-1\ge 1$ and that $\sF|_C$ is a line bundle. By Lemma \ref{Lem:purequot}, $\sF|_C$ is a line bundle if and only if $\sF|_C=\bsF{1}$ if and only if $\sT_1=0$ (by definition) if and only if $\sN\sF=\sF^{(n-1)}$ by Lemma \ref{Lem:supportofTi}.   
By Fact \ref{Fact:qll}\ref{Fact:qll:1} and by the trivial observation that a generalized line bundle is quasi locally free if and only if it is a line bundle, it is sufficient to show that $\sF^{(n-1)}=\sN\sF$ is a line bundle on $C_{n-1}$: indeed its second canonical filtration is the same of $\sF$ and the fact it is a line bundle implies that $\sF^{(j)}/\sF^{(j-1)}$ is a line bundle on $C$ for any $j\le n-1$, while $\sF/\sF^{(n-1)}$ is $\bsF{1}=\sF|_C$, which is locally free by hypothesis.
By Fact \ref{Fact:primaesecondfc}\ref{Fact:primaesecondfc:1}, $\sN\sF\simeq\bsF{n-1}\otimes\sN$, thus it is sufficient to prove that $\bsF{n-1}$ is a line bundle on $C_{n-1}$. The fact that $\sF\surj\bsF{n-1}$ implies that $\sF|_C\surj\bsF{n-1}|_C$, too. But $\sF|_C$ is a line bundle and $\bsF{n-1}|_C$ has rank $1$ on $C$, so the epimorphism has to be an isomorphism. Hence $\bsF{n-1}|_C$ is a line bundle on $C$ and, by inductive hypothesis, $\bsF{n-1}$ is a line bundle on $C_{n-1}$, as required.
\end{proof}
\begin{cor}\label{Cor:flifTi=0}
The following are equivalent:
\begin{enumerate}
\item\label{Cor:flifTi=0:1} $\sF$ is a line bundle on $X$;
\item\label{Cor:flifTi=0:2} $\sF|_{C_i}$ is a line bundle on $C_i$ for any $1\le i\le n-1$;
\item\label{Cor:flifTi=0:3} there exists $i\in\{1,\dotsc,n-1\}$ such that $\sF|_{C_i}$ is a line bundle on $C_i$.
\end{enumerate}
\end{cor}
\begin{proof}
It is immediate that \ref{Cor:flifTi=0:1} implies \ref{Cor:flifTi=0:2} and that the latter implies \ref{Cor:flifTi=0:3}.

If \ref{Cor:flifTi=0:3} holds, then also $(\sF|_{C_i})|_C=\sF|_C$ is a line bundle and then the Proposition allows to conclude that \ref{Cor:flifTi=0:3} implies \ref{Cor:flifTi=0:1}.  
\end{proof}

The next step is to introduce the generalizations of the index and of the local index sequence of a generalized line bundle on a ribbon (cf. \cite[Definition 2.7]{CK}).
\begin{defi}\label{Def:indices}
The \emph{$i$-th index} of $\sF$ is $\be{i}=\bet{i}{\sF}:=\operatorname{h}^0(C,\sT_{i}(\bsF{i+1}))$ and the \emph{indices-vector} of $\sF$ is $\be{.}=\bet{.}{\sF}:=(\be{1},\dotsc,\be{n-1})$.

Let $P\in X$ be a closed point, then the \emph{local $i$-th index} of $\sF$ at $P$ is $\bp{i}{P}=\btp{i}{P}{\sF}:=\operatorname{lenght}((\sT_{i}(\bsF{i+1}))_P)$ while its \emph{local indices-vector} at $P$ is $\bp{.}{P}=\btp{.}{P}{\sF}:=(\bp{1}{P},\dotsc,\bp{n-1}{P})$.

The \emph{local indices sequence} of $\sF$ is $\btp{.}{.}{\sF}=\bp{.}{.}=(\bp{.}{P_1},\dotsc,\bp{.}{P_k})$, where $P_1,\,\dotsc,\,P_k$ are the closed points supporting the torsion sheaves $\sT_i(\bsF{i+1})$, for $1\le i\le n-1$, i.e. the points in which $\sF$ is not locally free.
\end{defi}
The definition makes sense because $\bsF{i+1}$ is a generalized line bundle on $C_{i+1}$ and thus, thanks to Lemma \ref{Lem:supportofTi}, $\sT_i(\bsF{i+1})$ is an $\sO_C$-module.
\begin{rmk}
For any $1\le i\le n-1$, it holds that $\be{i}=\sum_{j=1}^{k}\bp{i}{P_j}$.
By definition, $\bet{j}{\bsF{i}}=\bet{j}{\sN^{n-i}\sF}=\bet{j}{\sF}$ for any $0< j<i$.
\end{rmk}
\begin{lemma}\label{Lem:propindices}
It holds that $\sT_i(\bsF{i+1})\subseteq\sT_{i+1}(\bsF{i+2})$, for any $1\le i\le n-2$. In particular, $\be{i}\le\be{i+1}$ and $\bp{i}{P}\le\bp{i+1}{P}$ for any $1\le i\le n-2$ and for any closed point $P\in X$. 
\end{lemma}
\begin{proof}
The second assertion is a straightforward consequence of the first one.

By the fact that $\ov{(\bsF{i})}_j=\bsF{j}$ for any $1\le j\le i\le n-1$, which is an obvious consequence of Lemma \ref{Lem:purequot}, it is sufficient to show that $\sT_{n-2}(\bsF{n-1})\subseteq\sT_{n-1}(\sF)$.

Indeed, $\bsF{n-1}\simeq\sN\sF\otimes\sN^{-1}$ by Fact \ref{Fact:primaesecondfc}\ref{Fact:primaesecondfc:1}. Therefore, $\sT_{n-2}(\bsF{n-1})=\sT_{n-2}(\sN\sF)$, because $\sN^{-1}$ is a line bundle on $C_{n-1}$. 
By the fact $\sN\sF$ is a generalized line bundle on $C_{n-1}$ and by Lemma \ref{Lem:supportofTi}, $\sT_{n-2}(\sN\sF)=(\sN\sF)^{(1)}/\sN^{n-1}\sF$ (thanks to the fact that the non-trivial terms of the two canonical filtrations of $\sN\sF$ seen as an $\sO_{C_{n-1}}$-module and seen as an $\sO_{X}$-module coincide). Moreover, $\sN\sF\subset\sF$ implies $(\sN\sF)^{(1)}\subset\sF^{(1)}$, hence
$\sT_{n-2}(\bsF{n-1})=\sT_{n-2}(\sN\sF)\subset\sF^{(1)}/\sN^{n-1}\sF=\sT_{n-1}(\sF)$, as wanted (the last equality holds again by Lemma \ref{Lem:supportofTi}).  
\end{proof}

\begin{cor}\label{Cor:fliffbetan-1=0}
The following are equivalent:
\begin{enumerate}
\item\label{Cor:fliffbetan-1=0:1} $\sF$ is a line bundle on $X$;
\item\label{Cor:fliffbetan-1=0:2} $\be{i}=0$ for any $1\le i\le n-1$;
\item\label{Cor:fliffbetan-1=0:3} $\be{n-1}=0$. 
\end{enumerate}
\end{cor}  

\begin{proof}
It is evident that \ref{Cor:fliffbetan-1=0:1} implies \ref{Cor:fliffbetan-1=0:2} which implies \ref{Cor:fliffbetan-1=0:3}; by the above Lemma \ref{Cor:fliffbetan-1=0:3} implies \ref{Cor:fliffbetan-1=0:2}.
The proof that \ref{Cor:fliffbetan-1=0:2} implies \ref{Cor:fliffbetan-1=0:1} is by induction. The basis is the case of ribbons, i.e. $n=2$, which is Proposition \ref{Prop:fliffT1=0}.

So let $n\ge 3$. By the fact $\be{i}=0$ for $1\le i\le n-2$, it holds that $\bsF{n-1}$ is a line bundle by inductive hypothesis; moreover, $\be{n-1}=0$ means that $\sF|_{C_{n-1}}=\bsF{n-1}$. Hence, $\sF$ is a line bundle by Corollary \ref{Cor:flifTi=0}. 
\end{proof}

An interesting problem, whose solution will be useful also in the study of stability conditions, is how to express the generalized degrees of the $\bsF{i}$'s in terms of that of $\sF$. The solution is the following:

\begin{prop}\label{Prop:DegFi}
Let $\sF$ be a generalized line bundle on $X$ of generalized degree $\Deg{\sF}=D$. Then 

\begin{equation}\label{Eq:Degquot}
\Deg{\bsF{i}}=\frac{1}{n}\Bigg[iD+(n-i)\sum\limits_{j=1}^{i-1}\be{j}-i\sum\limits_{j=i}^{n-1}\be{j}-\frac{in(n-i)}{2}\deg{\cC}  \Bigg],
\end{equation}

for any $1\le i\le n$.

\end{prop}

\begin{proof}
The proof is by induction on $n$. The basis is given by the trivial case $n=1$, where there is only the equality $D=D$.

In order to simplify the notation, let $D'=\Deg{\bsF{n-1}}$ and $\delta=-\deg{\cC}$.

By inductive hypothesis, it holds that  $\Deg{\bsF{i}}=\frac{1}{n-1}\Big[iD'+(n-1-i)\sum_{j=1}^{i-1}\be{j}-i\sum_{j=i}^{n-2}\be{j}+\frac{i(n-1)(n-1-i)}{2}\delta  \Big]$ for $1 \le i\le n-1$.

Now let us calculate $D'$ in terms of $D$: 
$D'=\chi(\bsF{n-1})-(n-1)\chi(\sO_{C})=\chi(\sF)-\chi(\sF^{(1)})-(n-1)\chi(\sO_C)=D-\chi(\sN^{n-1}\sF)
-\be{n-1}+\chi(\sO_C)=D-\chi(\bsF{1}\otimes \cC^{\otimes n-1})+\chi(\sO_C)-\be{n-1}=D-\deg{\bsF{1}\otimes \cC^{\otimes n-1}}-\be{n-1}=D-\deg{\bsF{1}}+(n-1)\delta-\be{n-1}=D-\frac{1}{n-1}\Big[D'-\sum_{j=1}^{n-2}\be{j}+\frac{(n-1)(n-2)}{2}\delta  \Big]+(n-1)\delta-\be{n-1}$;
 where the first equality holds by definition, the second by Remark \ref{Rmk:purequot} (and additivity of the Euler characteristic), the third by the definitions of $D$ and $\be{n-1}$, the forth by Fact \ref{Fact:primaesecondfc}\ref{Fact:primaesecondfc:1}, the fifth by definition of degree of a line bundle on $C$, the sixth by its additivity and the last by inductive hypothesis and by the fact that for a line bundle on $C$ degree and generalized degree coincide.
 
Thus we have that $D'=\frac{1}{n}\Big[(n-1)D+\sum_{j=1}^{n-2}\be{j} -(n-1)\be{n-1} +\frac{n(n-1)}{2}\delta\Big]$, as desired. To obtain the claim for $\Deg{\bsF{i}}$ with $1\le i\le n-2$, it is sufficient to substitute this value of $D'$ in the formulae obtained by inductive hypothesis. The case $i=n$ is a trivial identity.
\end{proof} 

\begin{cor}\label{Cor:DegF^i}
For any $1\le i\le n-1$, it holds that
\begin{equation}\label{Eq:Deg2filc}
\Deg{\sF^{(i)}}=\frac{1}{n}\Bigg[iD-i\sum\limits_{j=1}^{n-i-1}\be{j}+(n-i)\sum\limits_{j=n-i}^{n-1}\be{j}+\frac{in(n-i)}{2}\deg{\cC}\Bigg],
\end{equation}
where $D=\Deg{\sF}$, as in the proposition.
\end{cor}

\begin{proof}
The assertion follows from the Proposition, because $\sF/\sF^{(i)}$ is isomorphic to $\bsF{n-i}$ (cf. Remark \ref{Rmk:purequot}) and the generalized degree is additive (cf. Fact \ref{Fact:PropertiesGenRkeDeg}\ref{Fact:PropertiesGenRkeDeg:5}). 
\end{proof}

Proposition \ref{Prop:DegFi} can be used also to give another, apparently surprising, characterization of the indices of a generalized line bundle in terms of the torsion parts of the quotients of the first canonical filtration.

\begin{prop}\label{Prop:AltCarInd}
Let $\sF$ be a generalized line bundle on $X$. For any $1\le i\le n-1$, $\bet{i}{\sF}=\operatorname{h}^0(\cT_{n-1-i}(\sF))$, where $\cT_{n-1-i}(\sF)$ is the torsion part of $G_{n-1-i}(\sF)$. 
\end{prop} 
\begin{proof}
In order to simplify notations throughout the proof, we will set $\be{i}=\bet{i}{\sF}$ and $\beta_i=\beta_i(\sF)=\operatorname{h}^0(\cT_{n-1-i}(\sF))$.
We proceed by induction on $n$, the multiplicity of $X$.
The basis is constituted by $n=2$. In this case, it has to be considered only $\be{1}$ and the desired equality is verified by definition.

So let $n\ge 3$ and assume that the statement holds for $n-1$.
Let $\sF$ be a generalized line bundle on $X$ of generalized degree $D$ and let $d_i=\deg{G_i(\sF)}$. By definition and by additivity of the generalized degree, $D=\sum_{i=0}^{n-1}d_i$. 
It holds also that $d_i=d_{n-1}-(n-1-i)\deg{\cC}+\beta_{n-1-i}$, for any $0\le i\le n-2$. Indeed, for any $0\le i\le n-2$, by Fact \ref{Fact:primaesecondfc}\ref{Fact:primaesecondfc:3:2} there is a surjective morphism $\mu_{i,n-1-i}:G_i(\sF)\otimes\cC^{n-1-i}\surj G_{n-1}(\sF)=\sN^{n-1}\sF$; moreover, by the fact $\sF$ is a generalized line bundle, $G_{n-1}(\sF)$ is a line bundle over $C$, while $G_i(\sF)\otimes\cC^{n-1-i}$ has rank $1$ over $C$; hence, its locally free part is isomorphic to $G_{n-1}(\sF)$ and the kernel of $\mu_{i,n-1-i}$ is isomorphic to $\cT_{i}(\sF)$.
Therefore, $D=nd_{n-1}-n(n-1)/2\deg{\cC}+\sum_{i=1}^{n-1}\beta_{i}$.

Recall that by Fact \ref{Fact:primaesecondfc}\ref{Fact:primaesecondfc:2}, $d_{n-1}=\deg{\bsF{1}}+(n-1)\deg{\cC}$.

It follows that $D=n\deg{\bsF{1}}+n(n-1)/2\deg{\cC}+\sum_{i=1}^{n-1}\beta_{i}$. 
Substituting in this equality the value of $\deg{\bsF{1}}$ given by formula \eqref{Eq:Degquot} we get that $\sum_{i=1}^{n-1}\be{i}=\sum_{i=1}^{n-1}\beta_i$.

By Fact \ref{Fact:primaesecondfc}\ref{Fact:primaesecondfc:1}, $\sN\sF$ (which is a generalized line bundle over $C_{n-1}$) is isomorphic to $\bsF{n-1}\otimes \sN$, hence $\bet{i}{\sN\sF}=\bet{i}{\bsF{n-1}}=\be{i}$, for $1\le i\le n-2$, where the last equality holds by definition. Again by definition, $\beta_i=\beta_i(\sN\sF)$ for $1\le i\le n-2$. Thus we can use inductive hypothesis to assert that $\be{i}=\beta_i$, for any $1\le i\le n-2$. 

Therefore, the previous equality $\sum_{i=1}^{n-1}\be{i}=\sum_{i=1}^{n-1}\beta_i$ implies that also $\be{n-1}=\beta_{n-1}$.
\end{proof}

The next lemma and corollary describe some relations between a generalized line bundle and its dual.
\begin{lemma}\label{Lem:duality}
Let $\sF$ be, as usual, a generalized line bundle on $X$ and let $\sF^\vee$ be its dual (which is a generalized line bundle, too). Then there are the following canonical isomorphisms
\begin{enumerate}
\item\label{Lem:duality:1}$\sN^i(\sF^\vee)\simeq\big(\sF^{(n-i)}\big)^\vee\otimes\sC^{\otimes i}$, i.e. $\ov{(\sF^\vee)}_{n-i}\simeq\big(\sF^{(n-i)}\big)^\vee$;
\item\label{Lem:duality:2} $\sT_i(\sF^\vee)\simeq\un{\operatorname{Ext}}^1_{\sO_X}(\sT_i(\sF),\sO_X)\otimes\sC^{\otimes i}$, and then there is a non-canonical isomorphism between $\sT_i(\sF^\vee)$ and $\sT_i(\sF)$. 
\end{enumerate}
\end{lemma} 
\begin{proof}
The first assertion is Fact \ref{Fact:Dual}\ref{Fact:Dual:4:2}, thanks to the fact that for a generalized line bundle $(\sF|_{C_i})^{\vee\vee}=\bsF{i}$ by Lemma \ref{Lem:purequot} and, thus, $\ker{\sF\!\surj\!(\sF|_{C_i})^{\vee\vee}}$ coincides with $\sF^{(n-i)}$. The second one is Fact \ref{Fact:Dual}\ref{Fact:Dual:4:1}.   
\end{proof}
\begin{cor}\label{Cor:betaandduals}
For $1\le i\le n-1$, the following formula holds:
\begin{equation}\label{Eq:betadual}
\bet{i}{\sF^\vee}=\bet{n-1}{\sF}-\bet{n-1-i}{\sF}, 
\end{equation}
where $\bet{0}{\sF}$ is posed equal to $0$. 
\end{cor}
\begin{proof}
The case $i=n-1$ is implied by the second statement of the Lemma.

Thus, let $i\le n-2$.  
By the first point of the Lemma, $\bet{i}{\sF^\vee}=\bet{i}{\sF^{(i+1)}}$; hence, it is sufficient to show that $\bet{i}{\sF^{(i+1)}}=\bet{n-1}{\sF}-\bet{n-1-i}{\sF}$.

Consider the following commutative diagram:
\[
\begin{tikzcd}
0\ar[r] & \sN^{n-1}\sF \ar[d, "f"] \ar[r] & \sF^{(1)} \ar[d, equal] \ar[r] &\sT_{n-1}(\sF)\ar[d, "g"] \ar[r] & 0\\
0\ar[r] & \sN^{i}\sF^{(i+1)}  \ar[r] & \sF^{(1)} \ar[r] &\sT_{i}(\sF^{(i+1)}) \ar[r] & 0  
\end{tikzcd}
\]
By snake's lemma, $f$ is injective (as obvious), $g$ is surjective, as expected, and  $\ker{g}\simeq\coker{f}$; thus, it suffices to prove that $\coker{f}\simeq\sT_{n-1-i}(\bsF{n-i})$. The fact that, by their definitions, $\sN^{i}\sF^{(i+1)}\simeq(\sN^i\sF)^{(1)}$ implies that $\coker{f}\simeq\sT_{n-1-i}(\sN^i\sF)\simeq\sT_{n-1-i}(\bsF{n-i})$, where the latter isomorphism is due to Fact \ref{Fact:primaesecondfc}\ref{Fact:primaesecondfc:1}. 
\end{proof}

The next corollary will be useful in order to determine a surprisingly canonical Jordan-Holder filtration of a semistable generalized line bundle (see Proposition \ref{Prop:JH}).

\begin{cor}\label{Cor:betaF(i)}
For any $2\le i\le n-1$ and for any $1\le j\le i-1$, it holds that
\begin{equation}\label{Eq:beta2filtrcan}
\bet{j}{\sF^{(i)}}=\bet{n-i+j}{\sF}-\bet{n-i}{\sF}.
\end{equation} 
\end{cor}
\begin{proof}
Note that $\sF^{(i)}\simeq(\ov{(\sF^\vee)}_i)^\vee$ by Lemma \ref{Lem:duality}\ref{Lem:duality:1}. Thus, by a double application of the previous Corollary, it holds that $\bet{j}{\sF^{(i)}}=\bet{i-1}{\sF^\vee}-\bet{i-j-1}{\sF^\vee}=\bet{n-1}{\sF}-\bet{n-i}{\sF}-\bet{n-1}{\sF}+\bet{n-i+j}{\sF}=\bet{n-i+j}{\sF}-\bet{n-i}{\sF}$, as desired.
\end{proof}   

\section{Structure theorem}\label{SecStruc}
This section is devoted firstly to study the local and global structure of a generalized line bundle on a primitive multiple curve $X$ and then to describe the action of $\Pic{X}$ on the set of locally isomorphic generalized line bundles. It is not possible to extend straightforwardly to higher multiplicity \cite[Lemma 2.9]{CK}, which asserts that two generalized line bundles with the same local index sequence on a ribbon differ by the tensor product by a line bundle and makes also explicit the stabilizer of this action of the Picard group. Indeed we will show that having the same local indices sequence does not mean being locally isomorphic and that, moreover, in general, there is not a natural blow up on which a generalized line bundle becomes a line bundle (which in the case of ribbons is \cite[Theorem 1.1]{EG} and is the fundamental argument beyond the Lemma by Chen and Kass). However there is an action of the Picard group on the set of locally isomorphic generalized line bundles (see Corollary \ref{Cor:globstr}) whose stabilizer is completely known in multiplicity $3$ (see Corollary \ref{Cor:globstrmolt3}) and in some special cases in higher multiplicity (see Corollary \ref{Cor:globstrmoltn}). These special cases include those of the generic elements of irreducible components of stable generalized line bundles in the moduli space (see Theorem \ref{Thm:compirrid}).

The next lines recall the local set-up introduced before Definition \ref{Def:primafiltrcan}, adding also some more notation.
Let $P\in C$ be a closed point; then, in local arguments, $A_1=A_{\mathrm{red}}$ denotes $\sO_{C,P}$ (which is a DVR) and $\mm{A_1}=\mm{1}$ is its maximal ideal, while $A=A_n$ denotes $\sO_{X,P}$ with maximal ideal $\mm{A}=\mm{n}$ and $A_i$ denotes $\sO_{C_i,P}$ and $\mm{A_i}=\mm{i}$ is its maximal ideal, for $2\le i\le n-1$; moreover, $\pi_i$ denotes the projection $A\surj A_i$, for $1\le i\le n-1$. Let $y$ denote a generator of the nilradical of $A$ and let $\y{i}=\pi_i(y)$, for any $1\le i\le n-1$; fix a nonzero divisor $x$ such that $(x,y)=\mm{A}$ and let $\x{i}=\pi_i(x)$, for any $1\le i\le n-1$.   
\begin{defi}\label{Def:geninvmod}
An $A$-module $M$ is said to be \emph{generalized invertible} if there exists a generalized line bundle $\sF$ on $X$ such that the stalk $\sF_P$ is isomorphic to $M$. In algebraic terms, this means that $M$ is a torsion-free $A$-module (in the sense that $\operatorname{ann}(m)\subseteq (y)=\mathrm{Nil}(A)$ for any $0\neq m\in M$) and $M^{(i)}/M^{(i-1)}$ is an invertible $A_1$-module, i.e., being in a local context, it is isomorphic to $A_1$, for $1\le i\le n$ (in particular, $\Rk{M}=n$). In the trivial case $n=1$ generalized invertible is just invertible.

By the theory of generalized line bundles developed in the previous section, $M$ admits only one torsion-free quotient $\M{i}=M/(M^{(n-i)})$, which is an $A_i$ generalized invertible module, for any $1\le i\le n-1$.

The indices-vector of $M$ is $\be{.}=\bet{.}{M}:=\btp{.}{P}{\sF}$ and its $i$-th index is $\be{i}=\bet{i}{M}:=\btp{i}{P}{\sF}$, for $1\le i\le n-1$.
\end{defi}
The following theorem describes the structure of generalized invertible $A$-modules. It is called Local Structure Theorem because it describes all the stalks at closed points of a generalized line bundle on $X$.
\begin{thm}[Local Structure Theorem]\label{Thm:locstr}
Let $M$ be a generalized invertible $A$-module with indices-vector $\be{.}$. Then there exist elements (possibly equal to zero) $\alpha_{i,j}\in A$, with $3\le i\le n$ and $1\le j\le i-2$, well-defined modulo $(x^{\be{n-j}-\be{n-j-1}},y)$, and $m_i\in M$, with $1\le i\le n$ such that
\begin{equation*}
\begin{split}
M & \cong \frac{\bigoplus\limits_{i=1}^n m_iA}{\bigg(ym_1,ym_i-x^{\be{n-i+1}-\be{n-i}}m_{i-1}-\sum\limits_{j=1}^{i-2}\alpha_{i,j}m_j\Big|2\le i\le n\bigg)} 
 \\
& \cong \Biggl(y^{n-i} x^{\be{n-1}-\be{n-i}}+\sum\limits_{j=2}^{i-1}\Bigg(\sum\limits_{h=0}^{j-2}(-1)^h\alpha_{i-h,j-1-h}x^{\be{n-2(j+h)+5}-\be{n-2(j+h)+4}}\Bigg)\cdot \\ &\phantom{\cong\Biggl(} \cdot y^{n-j}\Big|
1\le i\le n \Biggl),
\end{split}
\end{equation*}
where $\be{0}=\be{n}=0$ and the last module is an ideal of $A$.
\end{thm}
\begin{proof}
First of all, observe that the second isomorphism is trivial: indeed, the relations between the generators of the ideal are those required for the $m_i$'s. So the only point is to show the first isomorphism.

The easiest way to prove such a statement is induction. The basis is the trivial case $n=1$, where generalized invertible modules are exactly invertible $A$-modules and there is no $y$: the statement reduces to the obvious observation that the only invertible modules on a local domain are free modules of rank $1$. In the case $n=2$ the statement is quite simpler than the general one: it reduces to the assertion that $M$ is isomorphic to $(x^{b},y)$. This is already known, although I do not know any explicit reference for it: it is a consequence of \cite[Theorem 1.1]{EG} and it is used various times in \cite{CK}.

So, let the statement hold for $n-1\ge 1$ and let us prove it for $n$.
Consider $M^{(n-1)}\subset M$: it is a generalized invertible module on $A_{n-1}$ with $\bet{i}{M^{(n-1)}}=\be{i+1}-\be{1}$, for $1\le i\le n-2$, by a local application of Corollary \ref{Cor:betaF(i)}. Thus, $\bet{n-i}{M^{(n-1)}}-\bet{n-i-1}{M^{(n-1)}}=\be{n-i+1}-\be{n-i}$ and, by inductive hypothesis, it holds that 
$M^{(n-1)}\cong \bigoplus_{i=1}^{n-1}  \tilde{m}_iA_{n-1}/(\y{n-1}\tilde{m}_{1},\y{n-1}\tilde{m}_i-\x{n-1}^{\be{n-i}-\be{n-i-1}}\tilde{m}_{i-1}-\sum_{j=1}^{i-2}\tilde{\alpha}_{i,j}\tilde{m}_j\big|1\le i\le n-2)$. The $\tilde{m}_i$'s belong to $M$; rename $\tilde{m}_i=m_{i}$ and choose $\alpha_{i,j}\in A$ over $\tilde{\alpha}_{i,j}$, for each pair $(i,j)$; so, we get that $M^{(n-1)}\cong
\bigoplus_{i=2}^n m_iA/(ym_1,ym_i-x^{\be{n-i+1}-\be{n-i}}m_{i-1}-\sum_{j=1}^{i-2}\alpha_{i,j}m_j\big|2\le i\le n-1)$.
 Moreover, $M/M^{(n-1)}=\M{1}\cong A_1$ by hypothesis. Therefore, choosing $m_n\in M$ over a generator of $\M{1}$ we obtain a set of generators of $M$, i.e. $m_1,\dotsc,m_n$. In order to complete the proof, we need to find the relations between $m_n$ and the other generators. Indeed, the submodule generated by $y^{n-1}m_n$ is isomorphic to $y^{n-1}M$; hence, substituting, if necessary, $m_n$ with another element with the same image in $\M{1}$, we have that $y^{n-1}m_n=x^{\be{n-1}}m_1$, by a local application of Lemma \ref{Lem:supportofTi}. Now using the other relations we get the desired one $ym_n-x^{\be{1}}m_{n-1}-\sum_{j=1}^{n-2}\alpha_{n,j}m_j$. By the fact we can substitute again $m_n$ with $m_n$ plus a linear combination of the other $m_i$'s we obtain that the $\alpha_{n,j}$'s are defined only modulo $(x^{\be{n-j}-\be{n-j-1}},y)$.
\end{proof}
\begin{rmk}\label{Rmk:Thm:locstr}
In order to apply \cite[Proposition 2.12]{H} to derive the global structure of generalized line bundles from the local one described in the previous theorem (see the proof of Corollary \ref{Cor:globstr}), generalized invertible $A$-modules should be classified up to \emph{linear equivalence} (i.e. up to product by an element of the total ring of fractions of $A$) and not up to isomorphism. A priori linear equivalence is stronger than being isomorphic. But, as on a DVR, also in this case linear equivalence and being isomorphic are equivalent, as we will see in the following lemma.

The classification of the ideals of $A_2$ up to linear equivalence had already been worked out in \cite[Example 3.9]{H} and in this case .
\end{rmk}
\begin{lemma}\label{Lem:iso=leq}
Two generalized invertible $A$-modules are isomorphic if and only if they are linearly equivalent.
\end{lemma}
\begin{proof}
As observed in the above Remark, linearly equivalence implies trivially being isomorphic. 
For the other implication, it is sufficient to work with ideals of $A$ (cf. \cite[Lemma 2.13]{H}); furthermore, one could consider only those of the completion of $A$ (cf. \cite[Proposition 2.14]{H}).

The proof is by induction, the basis is the elementary case of $A=A_1$ a DVR for which a generalized invertible ideal is just an invertible ideal, hence a principal ideal; principal ideals on $A$ are all both isomorphic and linearly equivalent to $A$ itself.

Hence, let us assume the statement holds for $A_{n-1}$, with $n\ge 2$, and let us prove it for $A=A_n$. Let $I$ and $J$ be two isomorphic ideals of $A$. In particular, they have the same indices that we denote simply as $\be{1}, \dotsc, \be{n-1}$. Let $i_1, \dotsc, i_n$ be the generators of $I$ verifying the relations given in the above Theorem and let $j_1,\dotsc, j_n$ be those generators of $J$; we can assume that one of them, say $I$, is precisely the ideal in the statement of the Theorem (in particular $i_1=y^{n-1}$ and $i_n=x^{\be{n}}+\alpha y$, for an appropriate $\alpha\in A$). It holds that $I^{(n-1)}$ and $J^{(n-1)}$, which are generated, by, respectively, $i_1,\dotsc,i_{n-1}$ and $j_1,\dotsc,j_{n-1}$, are isomorphic generalized invertible $A_{n-1}$-modules. Therefore, they are linearly equivalent on $A_{n-1}$ by inductive hypothesis; this means that there exists $\bar{f}\in F_{n-1}$ such that $\bar{f}J^{(n-1)}=I^{(n-1)}$, where $F_{n-1}$ is the total ring of fractions of $A_{n-1}$. Let $F$ be the total ring of fractions of $A$ and let $f\in F$ be an element restricting to $\bar{f}$; by the relations within the $j_l$'s and by the fact $f j_1,\dotsc,fj_{n-1}$ belong to $I$ and hence to $A$, $f$ can be chosen so that $fj_n$ belongs to $A$, too. It follows that $fJ=(i_1,\dotsc,i_{n-1},fj_n)$ and $fj_n$ verifies the same relation of $i_n$ with $i_1,\dotsc,i_{n-1}$; this imply that $fj_n-i_n$ belong to $(y^{n-1})=(i_1)$. Therefore, $fJ=I$.      
\end{proof}

Observe that sometimes a generalized invertible module can be generated by a smaller set of generators than those given by the local structure Theorem; a quite important case, which is fundamental to describe the irreducible components of generalized line bundles in the moduli space of semistable pure sheaves of generalized rank $n$ in Chapter \ref{SecModuli}, is that treated in the following corollary.
\begin{cor}\label{Cor:locstructspecial}
Let $M$ be a generalized invertible $A$-module with indices-vector $\be{.}$ such that there exists an integer $1\le j\le n-1$ such that $0=\be{j-1}<\be{j}=\be{n-1}=b$.
Then there exists $\alpha\in A$ such that $M\cong (x^{b}+\alpha y,\,y^j)$.
Moreover, there exist unique $z_{h,i}\in \bK$, for $1\le h\le \bar{\jmath}$, where $\bar{\jmath}=\min\{j,n-j\}-1$, and $0\le i\le b-1$, such that $M\cong\big(x^{b}+\sum_{h=1}^{\bar{\jmath}}\big(\sum_{i=0}^{b-1}z_{h,i}x^i\big)y^h,\,y^j\big)$. 
\end{cor}
\begin{proof}
The first assertion is a trivial consequence of the Theorem.

In order to simplify the notation in the proof of the second assertion, set $M(\beta)=(x^b+\beta y,\,y^j)$, for any $\beta\in A$. 
The existence of the $z_{h,i}\in \bK$ is equivalent to the fact that there exists $\alpha'\in A$, defined modulo $(x^b,\,y^{\bar{\jmath}+1})$, such that $M(\alpha)\cong M(\alpha')$.
First of all, observe that $M(\gamma+\beta y^{\bar{\jmath}})\cong M(\gamma)$, for any $\gamma,\beta\in A$, which implies that $\alpha$ is defined modulo $y^{\bar{\jmath}+1}$. Indeed, if $\bar{\jmath}=j-1$, the two modules are equal, while, if $\bar{\jmath}=n-j-1$, there is a trivial isomorphism, say $\varphi$ defined on the generators as $\varphi(x^b+(\gamma+\beta y^{\bar{\jmath}}) y) = x^b +\gamma y$ and $\varphi(y^j)= y^j$: in order to check that $\varphi$ is not only a bijection of sets but also a morphism of $A$-modules it is sufficient to verify that $y^j\varphi(x^b+(\gamma+\beta y^{\bar{\jmath}}) y)=(x^b+(\gamma+\beta y^{\bar{\jmath}}) y)\varphi(y^j)$, which is a trivial equality: in this case $\beta y^{\bar{\jmath}}y\varphi(y^j)=\beta y^n=0$.

The next step is to show that $M(\gamma+\beta x^b)$ is isomorphic to the module  $M\big(\sum_{l=1}^{\bar{\jmath}}(-\beta y)^{l-1}\gamma\big)$, for any $\gamma,\beta\in A$. Setting $\alpha=\bar{\alpha}+\beta x^b$, where $\bar{\alpha}\in A$ is an element without terms in $x^k$, with $k\ge b$ (looking for a while at $A$ as a $\bK$-vector space of infinite dimension), and iterating, if necessary, the proceeding, such an isomorphism is sufficient to conclude the existence of the desired $\alpha'$ (which is not necessarily congruent to $\alpha$ modulo $(x^b,y^{\bar{\jmath}+1})$).
The point is to check that $M(\gamma+\beta x^b)$ is equal to
$M\big(\sum_{l=1}^{2^c-1}(-\beta y)^{l-1}\gamma\big)$, where $c=[\log_2(\bar{\jmath})]+1$, because the latter is isomorphic to $M\big(\sum_{l=1}^{\bar{\jmath}}(-\beta y)^{l-1}\gamma\big)$, by the first step. The equality holds by the fact that $x^b+\sum_{l=1}^{2^c-1}(-\beta)^{l-1}y^l\gamma=(x^{b}+(\gamma+\beta x^b)y)\prod_{r=0}^{c}(1+(-1)^r(\beta y)^{2^r})$ and the latter is an invertible element of $A$.

It remains to prove the uniqueness of the $z_{h,i}$. So, we need to show that if  $M\big(\sum_{h=1}^{\bar{\jmath}}\big(\sum_{i=0}^{b-1}z_{h,i}x^i\big)y^{h-1}\big)$ and $M\big(\sum_{h=1}^{\bar{\jmath}}\big(\sum_{i=0}^{b-1}z'_{h,i}x^i\big)y^{h-1}\big)$ are isomorphic, then $z_{h,i}=z'_{h,i}$, for any $h$ and $i$. In order to simplify notations, set $s=x^{b}+\sum_{h=1}^{\bar{\jmath}}\big(\sum_{i=0}^{b-1}z_{h,i}x^i\big)y^{h}$ and $s'=x^{b}+\sum_{h=1}^{\bar{\jmath}}\big(\sum_{i=0}^{b-1}z'_{h,i}x^i\big)y^{h}$. 
Let $\psi$ be such an isomorphism and $\psi^{-1}$ its inverse. It holds that $\psi(y^j)=a_1 y^j+a_2 s'$ and $\psi(s)=a_3y^j+a_4s'$ and, analogously, $\psi^{-1}(y^j)=a'_1 y^j+a'_2 s$ and $\psi^{-1}(s')=a'_3 y^j+a'_4 s$. By the fact $y^{n-j}\psi{y^j}=0$ and $y^{n-1}\psi^{-1}(y^j)=0$, it follows that $a_1$ and $a'_1$ can be chosen so that $a_2=a_2'=0$; moreover, $y^j=\psi^{-1}(\psi(y^j))=a_1' a_1 y^j$ implies that $a_1$ and $a'_1$ are invertible and $a'_1=a_1^{-1}$ (set $a_1=u_1+m_1$ with $u_1\in \bK$ and $m_1\in\mm{A}$).
Moreover, $s=\psi^{-1}\psi(s)=a'_4 a_4 s+(a'_3 a_4 +a_1 a_3)y^j$ implies that $a'_4=a_4^{-1}$ and $(a'_3 a_4 +a_1 a_3)$ is a multiple of $y^{n-j}$.
As usual, $\psi$ is really a morphism if and only if $y^j\psi(s)=s\psi(y^j)$. But $y^j\psi(s)=a_3 y^{2j}+a_4y^js'=a_3 y^{2j}+a_4 y^j x^{b}+\sum_{h=1}^{\bar{\jmath}}\big(\sum_{i=0}^{b-1}a_4 z'_{h,i} x^i\big)y^{j+h}$ and $s\psi(y^j)=a_1 y^js=a_1 y^jx^{b}+\sum_{h=1}^{\bar{\jmath}}\big(\sum_{i=0}^{b-1}a_1(z_{h,i}-z'_{h,i}+z'_{h,i})x^i\big)y^{j+h}$.
Hence they are equal if and only if $(a_1-a_4)x^b y^j=\sum_{h=1}^{\bar{\jmath}}\sum_{i=0}^{b-1}\big((a_4-a_1)z'_{h,i}-(u_1+m_1)(z_{h,i}-z'_{h,i})\big)x^iy^{j+h}+a_3 y^{2j}$. So it has to be $a_1-a_4=\epsilon y$, for some $\epsilon\in A$, and the equality becomes $\epsilon x^b y^{j+1}=\sum_{h=1}^{\bar{\jmath}}\sum_{i=0}^{b-1}\big(-\epsilon z'_{h,i}y-u_1(z_{h,i}-z'_{h,i})-m_1(z_{h,i}-z'_{h,i})\big)x^iy^{j+h}+cy^{2j}$. Observing the powers of $x$ and $y$ in the right term (and remembering that $m_1$ belongs to $\mm{A}=(x,y)$, while $u_1$, $z_{h,i}$ and $z'_{h,i}$ belong to $\bK$), it has to hold that $z_{1,i}=z'_{1,i}$ for any $i$. But then all the terms on the right are divided by $y^{j+2}$ so $\epsilon=\eta y$, for some $\eta\in A$, and by the same considerations $z_{2,i}=z'_{2,i}$ for any $i$, and so on: the same argument continues to hold at any step and it follows that $z_{h,i}=z'_{h,i}$, for any $h$ and $i$, as wanted.
\end{proof}
\begin{rmk}\label{Rmk:Cor:locstructspecial}
\noindent
\begin{enumerate}

\item\label{Rmk:Cor:locstructspecial:1} 
The Corollary classifies these kind of modules also up to linear equivalence (cf. Remark \ref{Rmk:Thm:locstr}). It can be verified also directly, without using Lemma \ref{Lem:iso=leq}: indeed, if two modules are not isomorphic, they are also not linearly equivalent, and the only isomorphism used throughout the proof, i.e. $\varphi:M(\gamma+\beta y^{\bar{\jmath}})\overset{\sim}{\to} M(\gamma)$, in the case $\bar{\jmath}=n-j-1$, could be substituted with multiplication by $x^{-(2j-n)b}\prod_{l=0}^{2j-n-1}(x^{b}+(-\beta)^{l+1}\gamma^{l}y^{n-j+l})$.  

\item\label{Rmk:Cor:locstructspecial:2}
By the Theorem, in multiplicity greater than or equal to $3$ the local indices sequence is not always sufficient to characterize up to isomorphism stalks of generalized line bundles. Indeed, e.g. in the case $n=3$, using local notation, $(x^2+y,xy,y^2)$ has the same indices-vector of $(x^2,xy,y^2)$ but it is easy to show that they are not isomorphic.

Only in some special cases two generalized line bundles having the same local indices sequence are necessarily locally isomorphic; by the above Corollary it happens in particular for those having either $\bp{1}{P}=\bp{n-1}{P}$ or $\bp{1}{P}=\bp{n-2}{P}=0$, for any closed point $P\in C$. 
\end{enumerate}
\end{rmk}

It is easy to pass from the local description to the following affine picture.

\begin{cor}\label{Cor:affstruct}
Let $\sF$ be a generalized line bundle on $X$ and let $P$ be a closed point where $\sF$ has non-trivial local indices sequence $\bp{.}{P}=\be{.}$. There exists an affine neighbourhood $P\in U\subset X$, where $\sF(U)$ is isomorphic to the ideal $\Big(\!y^{n\!-i} x^{\be{n-1}\!-\be{n-i}}\!+\!\sum_{j=2}^{i-1}\!\big(\!\sum_{h=0}^{j-2}(-\!1)^h\alpha_{i\!-\!h,j\!-\!1\!-\!h}x^{\be{n-2(j+h)+5}\!-\be{n-2(j+h)+4}}\!\big)\!y^{n\!-j}\big|\\
1\le i\le n \Big)$, where $y$ is a generator of the nilradical of $\sO_X(U)$ and $x$ is a nonzerodivisor in $\sO_X(U)$ such that $(x, y)$ is the ideal of $P$ in $U$.

In the special case in which there exists an integer $1\le j\le n-1$ such that $0=\be{j-1}<\be{j}=\be{n-1}=b$, then there exist and are unique $z_{h,i}\in \bK$, for $1\le h\le \bar{\jmath}$, where $\bar{\jmath}=\min\{j,n-j\}-1$, and $0\le i\le b-1$, such that $\sF(U)\cong\Big(x^{b}+\sum_{h=1}^{\bar{\jmath}}\big(\sum_{i=0}^{b-1}z_{h,i}x^i\big)y^h,\,y^j\Big)$.    
\end{cor}

\begin{proof}
It is a trivial application of the Theorem and of the above Corollary, considering that there are only finitely many points on which the stalks of a generalized line bundle are not free.
\end{proof}

 In general it is possible to obtain only the following global description, which remains quite vague.

\begin{cor}[Global structure]\label{Cor:globstr}
Let $\sF$ be a generalized line bundle on $X$. Then $\sF$ is isomorphic to $\sI_{Z/X}\otimes\sG$, where $Z\subset C_{n-1}$ is a closed subscheme of finite support whose schematic intersection with $C$ is $\operatorname{Supp}(\sT_{n-1}(\sF))$, called the \emph{subscheme associated} to $\sF$, and $\sG$ is a line bundle on $X$.\\
Moreover, it holds that
\begin{enumerate}
\item $Z$ is unique up to adding a Cartier divisor.
\item Locally isomorphic generalized line bundles have the same associated subscheme, up to adding a Cartier divisor. In particular, if $\sF$ and $\sF'$ are locally isomorphic generalized line bundles, then there exists a line bundle $\sE$ such that $\sF=\sF'\otimes \sE$. Equivalently, there is a transitive action of $\Pic{X}$ on the set of locally isomorphic generalized line bundles.
\end{enumerate}
\end{cor}

\begin{proof}
Let $I_P$ denote the ideal isomorphic to $\sF_P$ described in the Theorem. Observe that the sum of $I_P$ with $\sN_P$ defines locally the support of $\sT_{n-1}(\sF)$. Moreover, it is evident that $\sO_{X,P}/I_P$ is an $\sO_{C_{n-1},P}$-module.

 Let $\sI\subset\sO_X$ be the ideal sheaf defined locally as $\sI_P=I_P$ for any closed point $P$ (hence it is isomorphic to $\sF\otimes\sG$ for some line bundle $\sG$ by \cite[Proposition 2.12]{H}): by the local observations, it defines a closed subscheme of finite support $Z\subset C_{n-1}$ such that $Z\cap C=\operatorname{Supp}(\sT_{n-1}(\sF))$.
 
 The two last assertions are trivial.
\end{proof}

As anticipated at the beginning of this section, it seems impossible to extend \cite[Theorem 1.1]{EG} to higher multiplicity for any generalized line bundle. However, it is possible to get something similar for some special choices of the local indices sequence. We will begin the study of this problem examining the case of multiplicity $3$, in which we get a quite complete picture about the action of the $\Pic$ of $C_n$ on the set of locally isomorphic generalized line bundles. 

\begin{lemma}\label{Lem:blowupmult3}
Let $\sF$ be a generalized line bundle on $X=C_3$, let $Z$ be the subscheme associated to it (cf. Corollary \ref{Cor:globstr}) and let $q:X'\to X$ be the blow up of $X$ along $Z$. Then
\begin{enumerate}
\item\label{Lem:blowupmult3:1}$\sF$ is the direct image of a line bundle $\sF'$ on $X'$ if and only if in any closed point $P$ such that $\sF_P$ is not a free $\sO_{X,P}$-module it holds that $2\bp{1}{P}\le\bp{2}{P}$.
\item\label{Lem:blowupmult3:2}$X'$ is a primitive multiple curve of multiplicity $3$ with reduced subcurve $C$ if and only if in any closed point $P$ such that $\sF_P$ is not a free $\sO_{X,P}$-module it holds that $2\bp{1}{P}\ge\bp{2}{P}$. 
\end{enumerate}
\end{lemma}
\begin{proof}
We can restrict our attention to the local setting, because there exists an affine cover where the situation is essentially equal to the local one (cf. Corollary \ref{Cor:affstruct}). This is due to the fact that a generalized line bundle is not free only in a finite set of closed points.

So, we study the local setting, using the same notation of the beginning of the section. If $M$ is an invertible generalized $A$-module for $A=A_3$, then by the Local Structure Theorem, i.e. Theorem \ref{Thm:locstr},  $M\cong(x^{\be{2}}+\alpha y,x^{\be{2}-\be{1}}y,y^2)$. 
If it were possible to extend the cited theorem by Eisenbud and Green, there would exist a natural blow up $A'$ of $A$, having the same reduced ring and possibly being again the local ring of a primitive multiple curve of multiplicity $3$, such that $M$ admits a structure of free $A'$-module of rank $1$. 
Being $M$ isomorphic to an ideal, there is only one natural blow up $A'$ of $A$ to consider: the one with respect to this ideal. 
By computations similar to those of the proof of \cite[Theorem 1.9]{BE} based on the fact that $y$ is nilpotent, it holds that $A'$ is isomorphic to $A[yx^{\be{2}-\be{1}}/(x^{\be{2}}+\alpha y),y^2/(x^{\be{2}}+\alpha y)]$, which reduces to $A[y/x^{\be{1}},y^2/x^{\be{2}}]$, in the simplest case of $\alpha=0$. 
By the fact $A'$ is contained in the total ring of fractions of $A$, $M$ admits a structure of $A'$-module if and only if it is closed under multiplication by $yx^{\be{2}-\be{1}}/(x^{\be{2}}+\alpha y)$ and $y^2/(x^{\be{2}}+\alpha y)$; this happens only if $2\be{1}\le\be{2}$. In this case, $M$ is isomorphic to $xA'$ and, thus, is free of rank $1$.

On the other hand, by definition such an $A'$ is the local ring of a primitive multiple curve when its nilradical is a principal ideal, and this happens only for $2\be{1}\ge\be{2}$.

Hence, we can summarize these results saying that $M$ admits a structure of free $A'$-module of rank $1$ for $2\be{1}\le\be{2}$ and that, with this condition, $A'$ is the local ring of a primitive multiple curve  (of multiplicity $3$) only if $2\be{1}=\be{2}$.
\end{proof}
\begin{rmk}
In particular, any generalized line bundle with $\be{1}=0$ verifies the hypotheses of the first point of the Lemma.
\end{rmk}
The following corollary is somehow an extension to multiplicity $3$ of \cite[Lemma 2.9]{CK}, although it requires more restrictive hypotheses.
\begin{cor}\label{Cor:globstrmolt3}
Let $\sF$ be a generalized line bundle on $X=C_3$, let $\sE$ be a line bundle on $X$ and let $q:X'\to X$ be the blow up of $X$ with respect to the ideal sheaf $\sI$ such that $\sI_P\cong\sF_P$ if $2\bp{1}{P}\le\bp{2}{P}$ and $\sI_P\cong\sF_P^\vee$ otherwise, for any closed point $P\in C$.
Then $\sF=\sF\otimes \sE$ if and only if $\sE$ belongs to $\ker{q^*:\operatorname{Pic}(X)\to\operatorname{Pic}(X')}$. In other words, the stabilizer of the action of $\Pic{X}$ on the set of locally isomorphic generalized line bundles (see Corollary \ref{Cor:globstr}) is $\ker{q^*:\operatorname{Pic}(X)\to\operatorname{Pic}(X')}$.
\end{cor}
\begin{proof}
First of all, observe that $\sI$ is the direct image of a line bundle on $X'$, by Lemma \ref{Lem:blowupmult3}\ref{Lem:blowupmult3:1}. So, the assertion follows for $\sI$ and any $\sF$ locally isomorphic to it from an easy application of the projection formula.

There are other two possibilities. The first one is that the dual of $\sF$ is locally isomorphic to $\sI$. In this case, we can conclude by the previous case and by the trivial observation that $(\sG\otimes\sE)^\vee\simeq\sG^\vee\otimes\sE^\vee$ if $\sE$ is a  line bundle and $\sG$ any sheaf.

The last case is the mixed one, in which nor $\sF$ neither $\sF^{\vee}$ are locally isomorphic to $\sI$. It follows easily from the previous ones. Indeed, in this case $\sF$ is locally isomorphic to $\sI_1\otimes\sI_2$, where $\sI_1$ is the ideal sheaf everywhere trivial except in the points for which $2\bp{1}{P}\le\bp{2}{P}$ where $\sI_{1,P}\cong\sF_P$ and $\sI_2$ is the ideal sheaf everywhere trivial except in the points for which $2\bp{1}{P}\ge\bp{2}{P}$ where $\sI_{2,P}\cong\sF_P$. The line bundles that fix $\sF$ are those fixing both $\sI_1$ and $\sI_2$. Hence, the assertion follows from the previous cases (essentially, because $\sI_1$ and $\sI_2$ are non-trivial in distinct points).  
\end{proof}
\begin{rmk}\label{Rmk:b1=b2}
It is not difficult to show that, for any generalized line bundle $\sF$ on $X$, it holds that $\SEnd{\sF}\cong q_*(\sO_{X'})$, where $q:X'\to X$ is the blow up of the previous Corollary. Indeed, if $\sF\cong \sI\otimes\sE$ or $\sF\cong\sI^\vee\otimes\sE$ (where $\sI$ is as in the statement of the Corollary and $\sE$ is a line bundle), it is immediate. Otherwise, $\SEnd{\sF}$ is locally isomorphic to $q_*(\sO_{X'})$ (because they are both locally isomorphic to $\sI$), so it is sufficient to show that there exists a morphism between the two sheaves. The latter is guaranteed by the universal property of the blow up, because, thanks to the local isomorphisms, the inverse image ideal sheaf of $\sI$ on the relative spectrum $\Specrel{\SEnd{\sF}}$ is invertible.         
\end{rmk}

This Remark will be useful in Chapter \ref{SecModuli} in order to study the dimension of the tangent space to a point corresponding to a stable generalized line bundle in the moduli space.

Now we turn our attention to multiplicity $n> 3$. In this case, the results obtained are less general; in particular, the action of the Picard group is described in detail only for some generalized line bundles, as we will see in Corollary \ref{Cor:globstrmoltn}. This lost of generality is not too dramatic, because the generalized line bundles covered are the generic ones, as we will see in Section \ref{SecModuli}. The basic ideas are essentially those of multiplicity $3$. Indeed, using local notation, if $M$ is a generalized invertible $A$-module, with $A\!=\!A_n$, by the Local Structure Theorem, a representative of its isomorphism class is the ideal $\Big(y^{n-i} x^{\be{n-1}-\be{n-i}}\!+\!\sum_{j=2}^{i-1}\big(\sum_{h=0}^{j-2}(-1)^h\alpha_{i-h,j-1-h}x^{\be{n-2(j+h)+5}-\be{n-2(j+h)+4}}\big) y^{n-j}\big|\\
1\le i\le n \Big)$ and for our purposes we can identify $M$ with it. It is quite complicate to write down explicitly the blow up of $A$ with respect to $M$ in full generality; hence, we restrict our attention to the case with all the $\alpha$'s zero. In this case, by the nilpotency of $y$, it holds that the blow up of $A$ with respect to $M$ is $A'=A[y/x^{\be{1}},y^2/x^{\be{2}},\dotsc,y^{n-2}/x^{\be{n-2}},y^{n-1}/x^{\be{n-1}}]$. By similar considerations to the case of $n=3$, we have that $M$ (under the hypothesis that all the $\alpha$'s are zero) admits a structure of $A'$-module (and, moreover, $M=xA'$) if and only if $\be{j}+\be{i}\le\be{j+i}$ for $1\le j\le n-2$ and $j\le i\le n-j-1$, while the nilradical of $A'$ is a principal ideal (and, thus, $A'$ can be seen as the local ring of a primitive multiple curve) if and only if $i\be{1}\ge\be{i}$ for $1\le i\le n-1$. The situation is more intricate when there are non-zero $\alpha$'s. Any case, the two different groups of inequalities are not dual one to the other, as it happened in multiplicity $3$, and there are many $A$-modules for which both of them do not hold.

However, the description is quite easy in the special case in which there exists a positive integer $h\le n-1$ such that $0=\be{h-1}<\be{h}=\be{n-1}=b$. In this case, as pointed out in Corollary \ref{Cor:locstructspecial}, there exists $\alpha\in A$ such that $M\cong(x^b+\alpha y,y^h)$ and the blow up results to be simply $A'=A[y^h/(x^b+\alpha y)]$: hence $M$ admits a structure of $A'$-module (and, moreover, it is a free $A'$-module of rank one) if and only if $h\ge n/2$, while the nilradical of $A'$ is a principal ideal if and only if $h=1$.    
By these observations and by essentially the same arguments of multiplicity $3$, the following lemma, corollaries and remarks hold:
\begin{lemma}\label{Lem:blowupmultn}
Let $\sF$ be a generalized line bundle on $X=C_n$ of local indices $\bp{.}{.}$, let $Z$ be the subscheme associated to it (cf. Corollary \ref{Cor:globstr}) and let $q:X'\to X$ be the blow up  of $X$ along $Z$. Then
\begin{enumerate}
\item\label{Lem:blowupmultn:1}If for any closed point $P$ such that $\sF_P$ is not a free $\sO_{X,P}$-module $\bp{j}{P}+\bp{i}{P}\le\bp{j+i}{P}$ for $1\le j\le n-2$ and $j\le i\le n-j-1$ and all the $\alpha$'s are zero or there exists a positive integer $n/2\le h(P)\le n-1$ such that $0=\bp{h(P)-1}{P}<\bp{h(P)}{P}=\bp{n-1}{P}$, then $\sF$ is the direct image of a line bundle $\sF'$ on $X'$.
\item\label{Lem:blowupmultn:2}If for any closed point $P$ such that $\sF_P$ is not a free $\sO_{X,P}$-module $i\bp{1}{P}\ge\bp{i}{P}$ for $2\le i\le n-1$ and all the $\alpha$'s are zero or $\bp{1}{P}=\bp{n-1}{P}$, then $X'$ is a primitive multiple curve of multiplicity $n$ with reduced subcurve $C$.  
\end{enumerate}
\end{lemma}

\begin{cor}\label{Cor:globstrmoltn}
Let $\sF$ be a generalized line bundle on $X=C_n$ of local indices $\bp{.}{.}$ and let $\sE$ be a line bundle on it.
\begin{enumerate}
\item\label{Cor:globstrmoltn:1} If $\sF$ verifies the hypotheses of the first point of the previous lemma, then $\sF\otimes\sE\simeq\sF$ if and only if $\sE$ belongs to $\ker{q^*:\operatorname{Pic}(X)\to\operatorname{Pic}(X')}$, where $q:X'\to X$ is the blow up of $X$ with respect to the ideal sheaf locally isomorphic to $\sF$. Equivalently, this kernel is the stabilizer of the transitive action of $\Pic{X}$ on the set of generalized line bundles locally isomorphic to $\sF$. 
\item\label{Cor:globstrmoltn:2} If $\sF^{\vee}$ verifies the hypotheses of the first point of the previous lemma, then $\sF\otimes\sE\simeq\sF$ if and only if $\sE$ belongs to $\ker{q^*:\operatorname{Pic}(X)\to\operatorname{Pic}(X')}$, where $q:X'\to X$ is the blow up of $X$ with respect to the ideal sheaf locally isomorphic to $\sF^{\vee}$. In other words, this kernel is the stabilizer of the transitive action of $\Pic{X}$ on the set of generalized line bundles locally isomorphic to $\sF$.
\item\label{Cor:globstrmoltn:3} If for any closed point $P$ there exists a positive integer $1\le h(P)\le n-1$ such that $0=\bp{h(P)-1}{P}<\bp{h(P)}{P}=\bp{n-1}{P}$, then $\sF\otimes\sE\simeq\sF$ if and only if $\sE$ belongs to $\ker{q^*:\operatorname{Pic}(X)\to\operatorname{Pic}(X')}$, where $q:X'\to X$ is the blow up of $X$ with respect to the ideal sheaf $\sI$ such that $\sI_P\cong\sF_P$ when $h(P)\ge n/2$ and $\sI_P\cong\sF_P^\vee$ otherwise, for any closed point $P$. Equivalently, this kernel is the stabilizer of the transitive action of $\Pic{X}$ on the set of locally isomorphic generalized line bundles whose local indices verify the hypothesis.
\end{enumerate} 
\end{cor}
\begin{proof}
The proof is essentially the same of Corollary \ref{Cor:globstrmolt3}, with Lemma \ref{Lem:blowupmult3} replaced by Lemma \ref{Lem:blowupmultn}.
\end{proof}
\begin{rmk}\label{Rmk:b1=bn-1}
It is not difficult to show that, for any generalized line bundle $\sF$ on $X$ whit the same hypotheses of the last point of the corollary, it holds that $\SEnd{\sF}\cong q_*(\sO_{X'})$, where $q:X'\to X$ is the blow up of $X$ with respect to the same ideal sheaf $\sI$ of the last point of the corollary. Indeed, if $\sF\cong \sI\otimes\sE$ or $\sF\cong\sI^\vee\otimes\sE$ (where $\sE$ is a line bundle on $X$), it is clear. Otherwise it is again evident that $\SEnd{\sF}$ is locally isomorphic to $q_*(\sO_{X'})$ (because they are both locally isomorphic to $\sI$), so it is sufficient to show that there exists a morphism between the two sheaves of $\sO_X$-algebras $\SEnd{\sF}$ and $\sO_{X'}$. The latter is guaranteed by the universal property of the blow up, because thanks to the local isomorphisms the inverse image ideal sheaf of $\sI$ on the relative spectrum $\Specrel{\SEnd{\sF}}$ is invertible.
\end{rmk}   
The Corollary and the Remark will be useful to replace as far as possible \cite[Lemma 2.9]{CK} (which is an essential tool in the proof of \cite[Lemma 4.4]{CK}) in the study of the moduli space in higher multiplicity.

\section{Semistable generalized line bundles}\label{SecStab}
This section, as the title suggests, is concerned with semistability of generalized line bundles; it extends to higher multiplicity the results of \cite[\S 3]{CK}. We will assume throughout this section that $\deg{\cC}<0$, for otherwise there will not be stable generalized line bundles, as pointed out in Remark \ref{Rmk:stability}\ref{Rmk:stability:4}.

We start with a quick remark about the slope and the Hilbert polynomial of a generalized line bundle and about an apparent discrepancy with \cite{CK}.
\begin{rmk}\label{Rmk:slopeandHpglb}
\noindent 
\begin{enumerate}

\item\label{Rmk:slopeandHpglb:1} Let $\sF$ be a generalized line bundle on $X=C_n$, then its slope is $\mu(\sF)=\Deg{\sF}/n$ and its Hilbert polynomial is $P_\sF(T)=\Deg{\sF}+n(1-g_1)+ndT$, while its reduced Hilbert polynomial $p_\sF(T)$ is equal to $T+(\mu(\sF)+1-g_1)/d$, where $d$ is the degree of a polarization on $C$, cf. Fact \ref{Fact:PropertiesGenRkeDeg}\ref{Fact:PropertiesGenRkeDeg:4}.

\item\label{Rmk:slopeandHpglb:2} In \cite[\S 3]{CK} there is a different definition of the slope of a generalized line bundle on a ribbon and, thus, an apparently different notion of its (semi)stability. However, it is equivalent to that used in this work, being both equivalent to Gieseker's semistability.  
\end{enumerate}
\end{rmk} 
The following theorem characterizes (semi)stability of a generalized line bundle on a primitive multiple curve in terms of a system of inequalities relating its indices and $\deg{\cC}$; it is the extension to higher multiplicity of \cite[Lemma 3.2]{CK}:
\begin{thm}\label{Thm:ssinequ}
Let $\sF$ be a generalized line bundle of generalized degree $D$ on $X$ and indices-vector $\be{.}$. Then $\sF$ is semistable if and only if the following inequalities hold:
\begin{equation}\label{Eq:ssin}
i\sum\limits_{j=i}^{n-1}\be{j}-(n-i)\sum\limits_{j=1}^{i-1}\be{j}\le-\frac{in(n-i)}{2}\deg{\cC}, \; \forall\; 1\le i\le n-1.  
\end{equation}
It is stable if and only if all the inequalities are strict. 
\end{thm}
\begin{proof}
It is a straightforward application of previous results. Indeed, by Lemma \ref{Lem:purequot} it is sufficient to verify that $\mu(\sF)\le\mu(\bsF{i})$, i.e. that $D\le n\Deg{\bsF{i}}/i$, for any $1\le i\le n-1$. The assertion is easily obtained by substituting in these inequalities the formulae \eqref{Eq:Degquot}: $\Deg{\bsF{i}}=\frac{1}{n}\Big[iD+(n-i)\sum_{j=1}^{i-1}\be{j}-i\sum_{j=i}^{n-1}\be{j}-\frac{in(n-i)}{2}\deg{\cC}  \Big]$.   
\end{proof} 
\begin{cor}\label{Cor:stabiffdualstab}
A generalized line bundles $\sF$ over $X$ is semistable (resp. stable) if and only if $\sF^\vee$ is semistable (resp. stable).
\end{cor}
\begin{proof}
This is a special case of Lemma \ref{Lem:StabandDual}, but it follows also from the Theorem. Indeed, using formulae \eqref{Eq:betadual}, the $i$-th inequality for $\sF$ is equivalent to the $(n-i)$-th for $\sF^\vee$.
\end{proof} 
\begin{rmk}
As anticipated in Remark \ref{Rmk:stability}\ref{Rmk:stability:4}, this Theorem implies that there can exist stable generalized line bundles if and only if $\deg{\cC}$ is negative because the left hand side of inequalities \eqref{Eq:ssin} is always non-negative (thanks to Lemma \ref{Lem:propindices} and to the obvious observation that $\be{1}\ge 0$). Line bundles are always stable, if $\deg{\cC}<0$, because their indices are all $0$, while they are the only type of strictly semistable generalized line bundles in the case $\deg{\cC}=0$.
\end{rmk}

The next step is to describe a Jordan-Holder filtration (which results to be in a certain sense canonical, being related to the second canonical filtration) and the Jordan-Holder graduate object of a strictly semistable generalized line bundle; it is an extension of \cite[Lemma 3.3]{CK} to higher multiplicity.
\begin{prop}\label{Prop:JH}
Let $\sF$ be a generalized line bundle of generalized degree $D$ on $X$ strictly semistable, i.e. such that in $k\ge 1$ of the inequalities \eqref{Eq:ssin} the equality holds. Let $0<i_1<\dotsb< i_k<n$ be the indices such that in the $i_h$-th inequality equality holds, for $1\le h\le k$. Then a Jordan-Holder filtration of $\sF$ is 
\[
0\subsetneq\sF^{(n-i_k)}\subsetneq\dotsb\subsetneq\sF^{(n-i_1)}\subsetneq\sF;
\]
and its Jordan-Holder graduate is
\[
\operatorname{Gr_{JH}}(\sF)=\bigoplus\limits_{h=0}^{k}\sF^{(n-i_h)}/\sF^{(n-i_{h+1})}=\bigoplus\limits_{h=0}^{k-1}\ov{(\sF^{(n-i_h)})}_{i_{h+1}}\oplus\sF^{(n-i_k)},
\]
where $i_0=0$ and $i_{k+1}=n$ .
\end{prop}
\begin{proof}
Set $\delta=-\deg{\cC}$ in order to simplify notation.

The proof is by strong induction. The basis is the trivial case $n=1$, i.e. the case of line bundles on a reduced smooth projective curve, which, as well-known, are all stable. 

So assume that the statement holds for generalized line bundles defined on $C_l$ with $1\le l\le n-1$.
First of all, observe that the greatest term in the Jordan-Holder filtration has to be a semistable pure subsheaf of $\sF$, having its same reduced Hilbert polynomial $p_\sF(T)$, i.e. having its same slope, and such that the quotient is a pure stable sheaf. Observe that $\sF^{(n-r)}$, being a generalized line bundle on $C_{n-r}$ has the same slope of $\sF$ if and only if $\Deg{\sF^{(n-r)}}=\frac{n-r}{n}D$, which is equivalent, by formula \eqref{Eq:Deg2filc}, to having the equality in the $r$-th inequality of $\sF$. Moreover, by similar considerations and by formula \eqref{Eq:Degquot}, in this case also the pure quotient $\bsF{r}$ has the same slope of $\sF$.

Hence, if $i_1$ is the greatest index $i$ such that the $i$-th inequality is an equality, $\sF^{(n-i_1)}$ is a plausible candidate as greatest term of the Jordan-Holder filtration of $\sF$. In order to check that it is really so, we need to verify that $\bsF{i_1}$ is stable and that $\sF^{(n-i_1)}$ is semistable.

The $i_1$-th pure quotient is stable if and only if $l\sum_{j=l}^{i_1-1}\!\be{j}-(i_1-l)\sum_{j=1}^{l-1}\!\be{j}\!\le\frac{li_1(i_1-l)}{2}\delta$ for any $1\le l\le i_1-1$, by Theorem \ref{Thm:ssinequ}. The choice of $i_1$ implies that $i_1\sum_{j=i_1}^{n-1}\be{j}-(n-i_1)\sum_{j=1}^{i_1-1}\be{j}=\frac{i_1n(n-i_1)}{2}\delta$ and $i\sum_{j=i}^{n-1}\be{j}-(n-i)\sum_{j=1}^{i-1}\be{j}<\frac{in(n-i)}{2}\delta$ for any $1\le i\le i_1-1$; substituting in the latter inequalities the value obtained for $\sum_{j=i_1}^{n-1}\be{j}$ one gets exactly those proving the stability of $\sF^{(i_1)}$.

On the other side, $\sF^{(n-i_1)}$ is semistable if and only if the inequalities $i\!\sum_{j=i}^{n-i_1-1}\!\bet{j}{\sF^{(n-i_1)}}-(n-i_1-i)\sum_{j=1}^{i-1}\bet{j}{\sF^{(n-i_1)}}\le\frac{i(n-i_1)(n-i_1-i)}{2}\delta$ hold for any $1\le i\le n-i_1-1$. The equality $i_1\sum_{j=i_1}^{n-1}\be{j}-(n-i_1)\sum_{j=1}^{i_1-1}\be{j}=\frac{i_1n(n-i_1)}{2}\delta$ and the fact that, by formulae \eqref{Eq:beta2filtrcan}, $\bet{j}{\sF^{(n-i_1)}}=\be{i_1+j}-\be{n-i}$, for any $i\le j\le n-i_1-1$, imply that the $i$-th of these inequalities is equivalent to the $(i_1+i)$-th of those giving the semistability of $\sF$; thus $\sF^{(n-i_1)}$ is semistable as wanted.

Now there are two distinct cases to consider.
 If $i_1=i_k$, i.e. all the other inequalities are strict, then $\sF^{(n-i_1)}$ is stable; therefore a Jordan Holder filtration of $\sF$ is simply $0\subset\sF^{(n-i_1)}\subset\sF$ and the graduate is $\operatorname{Gr_{JH}}(\sF)\cong\sF^{(n-i_1)}\oplus\bsF{i_1}$.
 
 Otherwise, $\sF^{(n-i_1)}$ is strictly semistable and one can conclude, getting the desired Jordan-Holder filtration and Jordan-Holder graduate of $\sF^{(n-i_1)}$ by induction (and hence those of $\sF$, for which one has to pay attention to the shift of indices: $i_h(\sF^{(n-i_1)})=i_{h+1}(\sF)-i_1(\sF)$). 
\end{proof}

\section{The moduli space of generalized line bundles}\label{SecModuli}
The aim of this section is to study the moduli space of semistable generalized line bundles on a primitive multiple curve.   
Throughout it, any primitive multiple curve $X$ will be such that $\deg{\cC}<0$, because, as observed in the previous one, only in this case there exist stable generalized line bundles on it.
We point out since now that some results are more precise in multiplicity $3$ than in arbitrary one.
The methods used in the following lines are inspired by the case of ribbons treated in \cite[\S 4.1]{CK} (where ordinary degree and rank are used instead of the generalized ones, but it is elementary to translate their results in terms of the latter). It seems very difficult to extend the main result of the cited section, i.e. \cite[Theorem 4.7]{CK}, to higher multiplicity: in the case of ribbons the involved sheaves which are not generalized line bundles are direct images of vector bundles of rank $2$ on a smooth integral curve, whose moduli spaces are well-known, while in the general case also the direct images of pure sheaves of generalized rank $n$ on $C_{i}$, for any $1\le i \le n-1$, are involved and their moduli spaces have never been studied in general (except, obviously, vector bundles of rank $n$ on $C=C_1$). The case of generalized rank $3$ sheaves on $C_2$ is treated, although not exhaustively, in \cite{S2} (as a particular case). This fact allows to formulate a plausible conjecture for the irreducible components in multiplicity $3$ (see Conjecture \ref{Conj:CompIrrn=3}).
 
After a brief introduction about the Simpson moduli space on a primitive multiple curve, useful mainly to fix notations, the section is divided into two parts. The first one is about the global structure of the moduli space; it describes the irreducible components of generalized line bundles and shows they are connected (at least for $\deg{\cC}$ small). The second one is quite shorter and studies the dimension of the Zariski tangent space to points of the moduli space corresponding either to generic generalized line bundles on $X$ or to rank $n$ vector bundles on $C$.

It is well-known (cf. e.g. \cite{HL}) that there exists a good moduli space parametrizing the semistable pure sheaves of fixed Hilbert polynomial $P$ on any projective scheme, and thus, in particular, on $X$. We will denote by $\mathrm{M}^\sharp(\sO_X,P)$ the moduli functor, by $\mathrm{M}(\sO_X, P)$ the projective scheme whose $\bK$-valued points parametrize the $S$-equivalence classes of semistable pure sheaves of Hilbert polynomial $P$ and by $\mathrm{M}_s(\sO_X, P)$ its subscheme whose $\bK$-valued points parametrize stable sheaves with the same Hilbert polynomial. The general theory works for polarized projective schemes, but, as observed in Fact \ref{Fact:PropertiesGenRkeDeg}\ref{Fact:PropertiesGenRkeDeg:4}, semistability on a primitive multiple curve is independent of the choice of a polarization. In the following, we will restrict our attention to Hilbert polynomials of the form $P_D(T)=D+n(1-g_1)+ndT$ (where $d$ is the degree of a polarization on $C$), i.e. to the Hilbert polynomials of generalized line bundles on $C_n$ of generalized degree $D$ (cf. Remark \ref{Rmk:slopeandHpglb}\ref{Rmk:slopeandHpglb:1}). As a consequence of Proposition \ref{Prop:granknsheaves}, the only other pure sheaves having the same Hilbert polynomial are direct images of sheaves of $\sO_{C_{n-1}}$-modules of generalized rank $n$ and generalized degree $D$.

\subsection{Global geometry: irreducible components}\label{SubSecGlobGeom}
As anticipated above, this subsection is concerned mainly with the irreducible components containing generalized line bundles.
 
First of all, we observe that a sheaf of rank $n$ on $C_{n-1}$ cannot specialize to a generalized line bundle on $X$. This implies that generalized line bundles are not contained in irreducible components whose generic element is supported on $C_{n-1}$.
 \begin{lemma}\label{Lem:openessofgenlinbund}
 Let $T$ be a $\bK$-scheme, let $\sF$ be a sheaf representing a $T$-valued point of $\mathrm{M}^\sharp(\sO_X,P)$ and let $T_0\subset T$ the locus of points $t\in T$ such that the restriction of $\sF$ to the fibre $X\times_{\bK} T\times_T\Spec{\bK(t)}$ is a generalized line bundle. Then $T_0\subset T$ is open.
 \end{lemma}
\begin{proof}
It is possible to prove this assertion in at least two different ways.

The first one is almost verbatim the proof of the cited Lemma by Chen-Kass:
by general results (cf., e.g., \cite[Theorem 4.3.4]{HL}) the moduli space of semistable pure sheaves on $C_{n-1}$ is projective, hence it is universally closed and this implies that $T\setminus T_0$ is closed.

The second one is maybe easier: it is well-known that the number of generators of a module can only decrease under specialization and generalized line bundles are the only sheaves of generalized rank $n$ whose generic stalk has only one generator.
\end{proof}
The above Lemma is a generalization of \cite[Lemma 4.2]{CK} from ribbons to the general case.

The next step consists in introducing some loci of generalized line bundles in $\mathrm{M}_s(\sO_X, P)$, among whose closures there are the irreducible components containing stable generalized line bundles, as we will show later.  
\begin{defi}\label{Def:Zbeta} Let $X$ be a primitive multiple curve of multiplicity $n$ and let $(b_{1,1},\dotsc,\,b_{1,r_1}),\dotsc,\: (b_{n-1,1},\dotsc,b_{n-1,r_{n-1}})$ be $n-1$ (possibly empty except one of them) sequences of positive integers such that the inequalities \eqref{Eq:ssin} are strictly verified by $\be{j}=\sum_{h=1}^{j}\sum_{l=1}^{r_h}b_{h,l}$, for $1\le j\le n-1$. Set
$\un{b}:=((\underset{j-1 \text{ times}}{\underbrace{0,\dotsc,0}},\underset{n-j \text{ times}}{\underbrace{b_{j,h},\dotsc,b_{j,h}}}))_{1\le j\le n-1,\,1\le h\le r_j}$. Define $Z_{\un{b}}\subset \mathrm{M}_s(\sO_X, P_D)$ as the subset of stable generalized line bundles of generalized degree $D$ and local indices sequence $\underline{b}$.
\end{defi}  
The inequalities in the above definition, which coincide with \cite[Definition 4.3]{CK} for $n=2$, are the stability conditions of Theorem \ref{Thm:ssinequ}. The following lemma is similar to \cite[Lemma 4.4]{CK}. 
\begin{lemma}\label{Lem:DimZbetamult}
If $n$ $\not\vert$ $D+(n(n-1)/2)\delta-\be{1}-\dotsb-\be{n-1}$, then $Z_{\un{b}}$ is empty. Otherwise, it is a constructible, irreducible subset of dimension $g_n-\be{n-1}+\sum_{h=1}^{n-1}r_h$, where $g_n$ is the genus of $X$.
\end{lemma}
\begin{proof}
In order to simplify notations set $\un{\beta}_j=(b_{j,1},\dotsc,b_{j,r_j})$ and $\beta_j=\sum_{l=1}^{r_j}b_{j,l}$, for $1\le j\le n-1$.

The first assertion follows from the fact that the first of formulae \eqref{Eq:Degquot} implies that $n|D+(n(n-1)/2)\delta-\bet{1}{\sF}-\dotsb-\bet{n-1}{\sF}$ for any generalized line bundle $\sF$.

So, assume $n|D+(n(n-1)/2)\delta-\be{1}-\dotsb-\be{n-1}$. The key point is to parametrize $Z_{\un{b}}$ with an irreducible variety of the required dimension.

Consider $C^{(\beta_j)}$, i.e. the $\beta_j$-th symmetric product of the reduced subcurve $C$, and within it the diagonal $\Delta_{\un{\beta}_j}$ associated to the partition $\un{\beta}_j$ of $\beta_j$ (for $j=1,\dotsc,n-1$), i.e. the image of the $r_j$-th direct product of $C$ with itself in $C^{(\beta_j)}$ under the morphism sending $(P_{j,1},\dotsc,P_{j,r_j})$ to $\sum_{l=1}^{r_j}\bp{j}{l}P_{j,l}$. Let $U\subset \Delta_{\un{\beta}_1}\times\dotsb\times\Delta_{\un{\beta}_{n-1}}$ be the locus such that the points $P_{j,l}$'s are all distinct (for $1\le j\le n-1$ and $1\le l\le r_j$). It is clear that $U$ is locally closed in $C^{(\beta_1)}\times\dotsb\times C^{(\beta_{n-1})}$ and irreducible of dimension $r_1+\dotsb+r_{n-1}$.

Set $m=\sum_{j=2}^{n-2}\sum_{h=1}^{\bar{\jmath}}\sum_{l=1}^{r_j}\sum_{i=0}^{b_{j,l}-1}1=\sum_{j=2}^{n-2}\bar{\jmath}\beta_j$, where $\bar{\jmath}=\min\{j,n-j\}-1$, and consider the affine space $\bA^m_k$; any of its closed points will be denoted in a completely non-standard way as $a=(z_{h,i}^{(j,l)})$, with $h$, $i$, $j$ and $l$ varying as in the definition of $m$.

For any $\Sigma\in U$ and $a\in\bA^m_k$, consider the ideal sheaf $\sI(\Sigma,a)$ defined as, using local notation, $\Big(x^{b_{j,l}}+\sum_{h=1}^{\bar{\jmath}}\sum_{i=0}^{b_{j,l}-1}z_{h,i}^{(j,l)}x^iy^h,y^l\Big)$ at the point $P_{j,l}$ for any $1\le j\le n-1$ and $1\le l\le r_j$. It holds that $\sI\Big(\Sigma,\Big(z_{h,i}^{(j,l)}\Big)\Big)$ is a stable generalized line bundle of generalized degree $-\sum_{j=1}^{n-1}j\beta_{j}-(n(n-1)/2)\delta$ and local indices sequence $\un{b}$.
So it is possible to define a map $\bA^{m}_k\times U\times\operatorname{Pic}^{D+\sum_{j=1}^{n-1}j\beta_{j}}(X)\to\mathrm{M}_s(\sO_X, P_D)$ by the rule $a\times\Sigma\times \sE\mapsto \sI(\Sigma,a)\otimes\sE$, where $\operatorname{Pic}^{D+\sum_{j=1}^{n-1}j\beta_{j}}(X)$ is the variety of line bundles on $X$ of generalized degree $D+\sum_{j=1}^{n-1}j\beta_{j}$ (it is the right generalized degree to be used by Corollary \ref{Cor:DegandTensProd}). 
Let $a\in\bA^m_j$; by the definition of $U$, for any set of $r_1+\dotsb+r_{n-1}$ points $P_{j,l}$, with $1\le j\le n-1$ and $1\le l\le r_j$, there is a unique closed subscheme $\Sigma\subset C$, corresponding to a point of $U$, such that  $\btp{1}{P_{j,l}}{\sI(\Sigma,a)}=\btp{j-1}{P_{j,l}}{\sI(\Sigma,a)}=0$ and $\btp{j}{P_{j,l}}{\sI(\Sigma,a)}=\btp{n-1}{P_{j,l}}{\sI(\Sigma,a)}=b_{j,l}$, for $1\le j\le n-1$ and $1\le j\le r_j$. Hence, by Corollary \ref{Cor:locstructspecial} and by Corollary \ref{Cor:globstrmoltn}\ref{Cor:globstrmoltn:3} (Corollary \ref{Cor:globstrmolt3} for multiplicity $3$), the image of the just defined map is $Z_{\un{b}}$ and, moreover, if $X'$ is the blow up described in the second of the cited Corollaries, the fibre over a point is an irreducible variety of dimension $\mathrm{h}^1(X,\sO_X)-\mathrm{h}^1(X',\sO_{X'})=g_n-g(X')=\sum_{j=1}^{n-1}(\bar{\jmath}+1)b_{j,l}$ (with $\bar{\jmath}$ as above), where the second equality is trivial and the first one holds because both $X$ and $X'$ do not have non-trivial global sections (for $X$ it is easily implied by $\deg{\cC}<0$, while for $X'$ it is Lemma \ref{Lem:glosecblup}). Hence, $Z_{\un{b}}$ is irreducible and constructible of dimension $m+r_1+\dotsb+r_{n-1}+g_n-\sum_{j=1}^{n-1}(\bar{\jmath}+1)b_{j,l}=g_n-\be{n-1}+\sum_{h=1}^{n-1}r_h$.
\end{proof}
In order to complete the above proof we need the following:  \begin{lemma}\label{Lem:glosecblup}
Let $q:X'\to X$ be the blow up considered in the proof of the previous Lemma. Then it has only trivial global sections, equivalently $g(X')=\mathrm{h}^1(X',\sO_{X'})$. 
\end{lemma}
\begin{proof}
The notation is as in the proof of the previous Lemma and, moreover, we set $\tilde{\jmath}=n-(\bar{\jmath}+1)$.

It follows from the definition that in any point $P$ different from the $P_{j,l}$'s
 $\sO_{X',P}\cong\sO_{X,P}$, while $\sO_{X',P_{j,l}}\cong \sO_{X,P_{j,l}}[y^{\tilde{\jmath}}/(x^{b_{j,l}}+\alpha y)]$ 
 for an appropriate $\alpha\in\sO_{X,P_{j,l}}$ which is not relevant to make explicit for the following counts. 
 So, we can consider the ideal sheaf $\sK_{n-1}\subset\sO_{X'}$ defined as $\sN^{n-1}_P$ for $P\notin\{P_{j,l}\}$ and as the ideal generated by $y^{n-1}/(x^{b_{j,l}}+\alpha y)$ in any $P_{j,l}$. The scheme $X'_{n-1}$ defined as $(C,\sO_{X'}/\sK_{n-1})$ is a multiple curve such that $\sO_{X'_{n-1},P}\cong\sO_{C_{n-1},P}$ for $P\notin\{P_{j,l}\}$ (for $2\le j\le n-2$ only) while $\sO_{X_{n-1}',P_{j,l}}\cong \sO_{C_{n-1},P_{j,l}}[\bar{y}_{n-1}^{\tilde{\jmath}}/(\bar{x}_{n-1}^{b_{j,l}}+\bar{\alpha}_{n-1} \bar{y}_{n-1})]$ (excluded the $P_{j,l}$ with $j=1$ or $n-1$, where $\bar{y}_{n-1}^{\tilde{\jmath}}=0$). Observe that if $r_j=0$ for $2\le j\le n-2$, then $X'_{n-1}$ is just $C_{n-1}$. If it is not the case, one can define $\sK_{n-2}\subset\sO_{X'_{n-1}}$ as the ideal isomorphic to $((\sN/\sN^{n-1})^{n-2})_P$ for $P\notin\{P_{j,l}\}$ (for $2\le j\le n-2$) and to the ideal generated by $\bar{y}_{n-1}^{n-2}/(\bar{x}_{n-1}^{b_{j,l}}+\bar{\alpha}_{n-1} \bar{y}_{n-1})$ in any $P_{j,l}$, with $2\le j\le n-2$. So it is possible to consider the scheme $X'_{n-2}$ defined as $(C,\sO_{X'_{n-1}}/\sK_{n-1})$. If it is not isomorphic to $C_{n-2}$, i.e. if there is at least one $r_j\neq 0$, with $3\le j\le n-3$, define similarly $\sK_{n-3}$ and $X'_{n-3}$ and continue in the same way defining $\sK_{n-i}$ and $X'_{n-i}$ for increasing $i$ until you get $X'_{n-\bar{\imath}}=C_{n-\bar{\imath}}$ ($\bar{\imath}$ is at most the integral part of $n/2$).
 
The point of the proof is to show that all the $\sK_{n-i}$ do not have global sections, for $1\le i\le n/2$, so that each $X'_{n-i}$ (and thus also $X'$) has only trivial global sections (because $C_{n-\bar{\imath}}$ has this property).

There are two distinct cases to be treated: $i<n/2$ and $i=n/2$ (the latter is possible only if $n$ is even).

For any $1\le i< n/2$, the sheaf $\sK_{n-i}$ is a line bundle on $C$ and there is an exact sequence $0\to(\sN/\sN^{n-i+1})^{n-i}\to\sK_{n-i}\to\sO_{D_i}\to 0$, where $D_i\subset C$ is an effective divisor of length $\be{n-i}-\be{i-1}$. Hence, it is sufficient to show that $\deg{\sK_{n-i}}=-(n-i)\delta+\be{n-i}-\be{i-1}<0$, i.e. that $\be{n-i}-\be{i-1}<(n-i)\delta$.

Consider the $i$-th and the $(n-i)$-th stability inequalities \eqref{Eq:ssin}, which hold strictly by hypothesis. They can be written as:  
\begin{equation*}
\left\{
\begin{aligned}
&i\Bigg(\sum\limits_{j=n-i}^{n-1}\be{j}-\sum\limits_{j=1}^{i-1}\be{j}\Bigg)+i\sum\limits_{j=i}^{n-i-1}\be{j}-(n-2i)\sum\limits_{j=1}^{i-1}\be{j}<\frac{in(n-i)}{2}\delta\\
&(n-i)\Bigg(\sum\limits_{j=n-i}^{n-1}\be{j}-\sum\limits_{j=1}^{i-1}\be{j}\Bigg)-i\sum\limits_{j=i}^{n-i-1}\be{j}+(n-2i)\sum\limits_{j=1}^{i-1}\be{j}<\frac{in(n-i)}{2}\delta.
\end{aligned}
\right. 
\end{equation*}

By Lemma \ref{Lem:propindices}, it holds that $\be{n-i}\le\be{j}$ for any $n-i\le  j\le n-1$, that $\be{i-1}\ge\be{h}$ for any $1\le h\le i-1$ and that $\be{n-1}\ge\be{n-i}-\be{i-1}$; thus, each of the above inequalities implies the corresponding one within the following
\begin{equation*}
\left\{
\begin{aligned}
&i\sum\limits_{j=i}^{n-i-1}\be{j}-(n-2i)(i-1)\be{i-1}<\frac{in(n-i)}{2}\delta-i^2(\be{n-i}-\be{i-1})\\
&(n-i)i(\be{n-i}-\be{i-1})<\frac{in(n-i)}{2}\delta+i\sum\limits_{j=i}^{n-i-1}\be{j}-(n-2i)(i-1)\be{i-1}.
\end{aligned}
\right. 
\end{equation*}

Hence, substituting the first one in the right hand term of the second one it follows that
\[
(n-i)i(\be{n-i}-\be{i-1})<in(n-i)\delta-i^2(\be{n-i}-\be{i-1}), 
\]
which is equivalent to the desired inequality.

Now assume $n$ even and consider the case of $i=n/2$. As in the previous case, it holds that
$\sK_{n/2}$ is a line bundle on $C$ and that there is an exact sequence $0\to(\sN/\sN^{n/2+1})^{n/2}\to\sK_{n/2}\to\sO_{D_{n/2}}\to 0$, where $D_{n/2}\subset C$ is an effective divisor of length $\be{n/2}-\be{(n-2)/2}$. Hence, it is sufficient to show that $\deg{\sK_{n/2}}=-(n/2)\delta+\be{n/2}-\be{(n-2)/2}<0$, i.e. that $\be{n/2}-\be{(n-2)/2}<(n/2)\delta$.
In this case, we need only the $(n/2)$-th stability inequality \eqref{Eq:ssin}, which can be written as 
\begin{equation*}
\frac{n}{2}\Bigg(\sum\limits_{j=n/2}^{n-1}\be{j}-\sum\limits_{j=1}^{(n-2)/2}\be{j}\Bigg)<\bigg(\frac{n}{2}\bigg)^3\delta.
\end{equation*}
Again by the basic inequalities between indices due to Lemma \ref{Lem:propindices}, the left hand term is greater than or equal to $(n/2)^2(\be{n/2}-\be{(n-2)/2})$. Therefore, it holds that $\be{n/2}-\be{(n-2)/2}<(n/2)\delta$, as wanted.
\end{proof}

Among the Zariski closures of the loci introduced in Definition \ref{Def:Zbeta} there are the irreducible components of the moduli space containing stable generalized line bundles. In order to prove this fact, it is convenient to study some deformations of generalized line bundles. The following lemma is inspired by \cite[Lemma 4.9]{CK}. 
\begin{lemma}\label{Lem:deformationsn}
Let $\sF$ be a generalized line bundle on $X=C_n$ of local indices sequence $\bp{.}{.}$. Let $P$ be a closed points of $C$, such that $\bp{n-1}{P}\ge 2$.
\begin{enumerate}
\item\label{Lem:deformationsn:1} If it does not exist an integer $h$ such that $0=\bp{h-1}{P}<\bp{h}{P}=\bp{n-1}{P}$, then $\sF$ is the specialization of a generalized line bundle $\sF'$ with the same local indices sequence of $\sF$ except in $P$, where $\btp{n-2}{P}{\sF'}=0$ and $\btp{n-1}{P}{\sF'}=\bp{n-1}{P}-\bp{n-2}{P}$, and in at most other $n-2$ closed points $Q_1,\dotsc,\,Q_{n-2}$, where $\bp{1}{Q_j}=\bp{n-1}{Q_j}=0$, while $\btp{j-1}{Q_j}{\sF'}=0$ and $\btp{j}{Q_j}{\sF'}=\btp{n-1}{Q_j}{\sF'}=\bp{j}{P}-\bp{j-1}{P}$, for any $1\le j\le n-2$.
\item\label{Lem:deformationsn:2} If there exists an integer $h$ such that $0\!=\!\bp{h-1}{P}\!<\!\bp{h}{P}\!=\!\bp{n-1}{P}$, then $\sF$ is the specialization of a generalized line bundle $\sF'$ with the same local indices sequence of $\sF$ except in $P$, where $0\!=\!\btp{h-1}{P}{\sF'}\!<\!\btp{h}{P}{\sF'}\!=\!\btp{n-1}{P}{\sF'}\!=\!\bp{n-1}{P}-1$, and in another closed point $Q$, where $\bp{1}{Q}=\bp{n-1}{Q}=0$ and $0=\btp{h-1}{Q}{\sF'}<\btp{h}{Q}{\sF'}=\btp{n-1}{Q}{\sF'}=1$.
\end{enumerate} 
\end{lemma} 
\begin{proof}
For both the two points we will exhibit explicit deformations, respectively over $\bK[t_1,\dotsc,t_{n-2}]$ and over $\bK[t]$.
In order to simplify notation, throughout the proof $\be{i}=\bp{i}{P}$, for $i=1,\dotsc,n-1$.

First of all, recall that by Corollary \ref{Cor:globstr}, $\sF$ is isomorphic to $\sI_{Z/X}\otimes\sL$, where $Z\subset C_{n-1}$ is a closed subscheme of finite support and $\sL$ is a line bundle on $X$. Thus, it is sufficient to find an appropriate deformation of $\sI=\sI_{Z/X}$, say $\sI'$, because, then, the desired deformation of $\sF$ would be $\sI'\otimes \sL$, where by a slight abuse of notation $\sL$ here denotes the constant family with fibre $\sL$. Deforming $\sI$ is equivalent to deforming $Z$.
In order to do that we will use the local affine description of generalized line bundles given in Corollary \ref{Cor:affstruct}. So let $U=\Spec{A}$ be an affine neighbourhood of $P$, in which the thesis of the cited Corollary holds (in particular, $P$ is the only closed point in $U$ where $\sI$ is not locally free). Now it is necessary to distinguish the two cases.

Start with \ref{Lem:deformationsn:1}: using the notation of the cited Corollary, it holds that $\sI(U)\cong(x^{\be{n-1}-\be{i-1}}y^{i-1}+\alpha_i y^i)_{i=1,\dotsc,n}$ (where $\alpha_i\in A$, for $1\le i\le n-1$, are not completely arbitrary, they could be written as in the cited Corollary) and the desired deformation is given by the extension of the ideal $\big(x^{\be{n-1}-\be{n-2}}y^i\prod\limits_{j=i}^{n-2}(x-t_j)^{\be{j}-\be{j-1}}+\alpha_i y^{i-1},y^{n-1}\big)_{i=1,\dotsc,n-1}$ to a proper flat family over $\Spec{\bK[t_1,\dotsc,t_{n-1}]}$ (it is possible to have such an extension by, e.g., the properness of the Hilbert scheme). This generic fibre is the expression over $U$ of a generalized line bundle having the desired local indices sequence, while the special fibre is $\sI(U)$.

The proof of \ref{Lem:deformationsn:2} is similar: $\sI(U)\cong(x^{\be{n-1}}+\alpha y,y^{n-h})$ (also this $\alpha\in A$ is not completely arbitrary, it can be expressed as in the last statement of Corollary \ref{Cor:affstruct}) and, similarly, the desired deformation is the extension of the ideal $(x^{\be{n-1}-1}(x-t)+\alpha y,y^{n-h})$ to a proper flat family over $\Spec{\bK[t]}$. Also in this case the needed verifies are almost trivial.
\end{proof}

The following theorem describes the irreducible components of the moduli space containing stable generalized line bundles. In multiplicity $2$ it is just \cite[Theorem 4.6]{CK}.
\begin{thm}\label{Thm:compirrid} Let $X$ be a primitive multiple curve of multiplicity $n$ and let $\be{1}\le\dotsb\le\be{n-1}$ be non-negative integers satisfying $n|D+(n(n-1)/2)\delta-\sum_{i=1}^{n-1}\be{i}$ and the strict inequalities \eqref{Eq:ssin}. 
Let $\bar{Z}_{\be{1},\dotsc,\be{n-1}}\subset \mathrm{M}(\sO_X,P_D)$ be the Zariski closure of the locus of stable generalized line bundles of indices-vector $(\be{1},\dotsc,\be{n-1})$.
If $(\be{1},\dotsc,\be{n-1})\neq(0,\dotsc,0)$, then $\bar{Z}_{\be{1},\dotsc,\be{n-1}}$ coincides with the Zariski closure of $Z_{\underline{b}}$, where $\underline{b}$ is the sequence $(\underset{\be{i}-\be{i-1} \text{ times}}{\underbrace{(\underset{i-1}{\underbrace{0,\dotsc,0}},\underset{n-i}{\underbrace{1,\dotsc,1}})}})_{1\le i\le n-1}$, where $\be{0}=0$.

Then any $\bar{Z}_{\be{1},\dotsc,\be{n-1}}$ is an irreducible component of $\mathrm{M}(\sO_X,P_D)$ of dimension $g_n=g(X)$. Moreover, any irreducible component containing a stable generalized line bundle is equal to $\bar{Z}_{\be{1},\dotsc,\be{n-1}}$, for a unique choice of $\be{1}\le\dotsc\le\be{n-1}$ satisfying the above conditions.
\end{thm}

\begin{proof}
The theorem is a straightforward application of the above lemmata.

By definition and by Theorem \ref{Thm:ssinequ}, $\cup\bar{Z}_{\be{1},\dotsc,\be{n-1}}$ contains the locus of stable generalized line bundles.
 
The first assertion is implied by a repeated application of Lemma \ref{Lem:deformationsn}. Combined with Lemma \ref{Lem:DimZbetamult} it implies that each $\bar{Z}_{\be{1},\dotsc,\be{n-1}}$ is irreducible of dimension $g_n$.

Now let $\bar{Z}$ be an irreducible component containing a stable generalized line bundle. Its subset consisting of stable generalized line bundles is open (by Lemma \ref{Lem:openessofgenlinbund}) and non-empty, hence it is dense. Thus, $\bar{Z}$ is contained in the union $\cup\bar{Z}_{\be{1},\dotsc,\be{n-1}}$. Hence, $\bar{Z}=\bar{Z}_{\be{1},\dotsc,\be{n-1}}$, for some $(\be{1},\dotsc,\be{n-1})$, because these loci are irreducible.

So, some of the $\bar{Z}_{\be{1},\dotsc,\be{n-1}}$ are irreducible components. Moreover, by their irreducibility and equidimensionality, each of them is a component.
Furthermore they are all distinct, because if $(\be{1},\dotsc,\be{n-1})\neq(\be{1}',\dotsc,\be{n-1}')$, the generic elements of $\bar{Z}_{\be{1},\dotsc,\be{n-1}}$ and $\bar{Z}_{\be{1}',\dotsc,\be{n-1}'}$ are different. 
\end{proof}

In order to study the connectedness of the locus of stable generalized line bundles, it is useful to introduce some other deformations of generalized line bundles (analogous to that of  \cite[Lemma 4.5]{CK})) 

\begin{lemma}\label{Lem:def-n}
Let $\sF$ be a generalized line bundle on $X=C_n$ of local indices sequence $\bp{.}{.}$ and let $P$ be a closed point of $C$ such that $\bp{n-1}{P}\neq 0$.

\begin{enumerate}

\item\label{Lem:def-n:1} Assume that $\sF_P\cong(x^{\bp{n-1}{P}},y^{n-h})$, for an integer $1\le h\le n-1$, and that $\bp{n-1}{P}\ge k$, where $k=n/\gcd(n,h)$. Then $\sF$ is the specialization of generalized line bundles $\sF'$, whose local indices sequence is equal to $\bp{.}{.}$ except in $P$, where $\btp{n-h-1}{P}{\sF'}=\bp{n-h-1}{P}=0$ and $\btp{n-h}{P}{\sF'}=\btp{n-1}{P}{\sF'}=\bp{n-1}{P}-l$.

\item\label{Lem:def-n:2} Assume that $\sF_P\cong(x^{\bp{n-1}{P}-\bp{i}{P}}y^i)_{i=0,\dotsc,n-1}$. Let $0=j_0\le j_1\le\dotsb\le j_{n-1}$ be integers whose sum is divided by $n$ and such that $\be{i}-j_i\ge\be{i-1}-j_{i-1}\ge 0$, for any $2\le i\le n-1$. Then $\sF$ is the specialization of generalized line bundles $\sF'$, whose local indices sequence is equal to $\bp{.}{.}$ except in $P$, where $\btp{i}{P}{\sF'}=\btp{i}{P}{\sF}-j_i$, for any $1\le i\le n-1$. 

\item\label{Lem:def-n:3} If $\sF_P\cong(x^{\bp{n-1}{P}-\bp{i}{P}}y^i)_{i=0,\dotsc,n-1}$ and $\bp{1}{P}\ge 2$, then $\sF$ is the specialization of generalized line bundles $\sF'$, whose local indices sequence is equal to $\bp{.}{.}$ except in $P$, where $\btp{i}{P}{\sF'}=\btp{i}{P}{\sF}-2$, for any $1\le i\le n-1$, and in another point $Q$ where $\bp{1}{Q}=\bp{n-1}{Q}=\btp{1}{Q}{\sF'}=0$ and $\btp{2}{Q}{\sF'}=\btp{n-1}{Q}{\sF'}=1$. 

\item\label{Lem:def-n:4} If $\sF$ is the dual of a generalized line bundle verifying the hypotheses of the previous point, then $\sF$ is the specialization of generalized line bundles $\sF'$, whose local indices sequence is equal to $\bp{.}{.}$ except in $P$, where $\btp{i}{P}{\sF'}=\btp{i}{P}{\sF}$, for any $1\le i\le n-2$ and $\btp{n-1}{P}{\sF'}=\btp{n-1}{P}{\sF}-2$, and in another point $Q$ where $\bp{1}{Q}=\bp{n-1}{Q}=\btp{n-3}{Q}{\sF'}=0$ and $\btp{n-2}{Q}{\sF'}=\btp{n-1}{Q}{\sF'}=1$. 
\end{enumerate}
\end{lemma}  
\begin{proof}
We will exhibit explicit deformations over $\bK[t]$. As for Lemma \ref{Lem:deformationsn}, it is sufficient to work with $\sI$, the ideal sheaf in the orbit of $\sF$ under the action of $\Pic{X}$. As there, deforming $\sI$ is the same thing of deforming the associated subscheme $Z\subset C_{n-1}$. It is also sufficient to work in an open neighbourhood, $U=\Spec{A}$, of $P$ in which $\sF$ (or, equivalently, $\sI$) is locally free in all points except $P$. In order to simplify notation throughout the proof $\be{i}=\bp{i}{P}$, for $1\le i\le n-1$, and $I=\sI(U)$. It is time to distinguish the various cases.

Start with \ref{Lem:def-n:1}: it holds by Corollary \ref{Cor:affstruct} that $I=(x^{\be{n-1}},y^{n-h})\subset A$ (by a slight abuse of notation, the affine one is identical to the local one used in the statement, but there is no risk of confusion because throughout the proof it will be used only the affine one; to be more precise in the local one $x$ and $y$ could be substituted by $x_P$ and $y_P$). As usual it is sufficient to give the generic fibre of the deformation, which, in this case, is the ideal $I'_t=(x^{\be{n-1}-k},y^{n-h})\cap((x-t)^k,y^{n-k}-t^{\be{n-1}-k+1}(x-t)^{k-1})$. Indeed by the fact $x^{\be{n-1}-k}(x-t)^k$ and $y^{n-h}-x^{\be{n-1}-k}t(x-t)^{k-1}=y^{n-h}-t^{\be{n-1}-k+1}(x-t)^{k-1}-(x-t)^kt\sum_{i=k+1}^{\be{n-1}}x^{i-k-1}t^{\be{n-1}-i}$ (if $\be{n-1}\ge k+1$; if $\be{n-1}=k$, instead of the latter consider $y^{n-j}-t(x-t)^{k-1}$) belong to $I'_t$, for any $t\neq 0$, it follows that $I$ is contained in the special fibre. The fact they coincide is due to degree considerations: $A/I$ has length $(n-h)\be{n-1}$ while $A/(x^{\be{n-1}-k},y^{n-h})$ has length $(n-h)(\be{n-1}-k)$ and $A/((x-t)^k,y^{n-k}-t^{\be{n-1}-k+1}(x-t)^{k-1})$ has length $(n-h)k$, for any non-zero value of the parameter $t$. Moreover, the ideal $((x-t)^k,y^{n-k}-t^{\be{n-1}-k+1}(x-t)^{k-1})$ defines a Cartier divisor of $X$ (for any fixed non-zero $t$), and so the part contributing to the local indices sequence of $I'_t$ is only $(x^{\be{n-1}-k},y^{n-k})$.

The proof of \ref{Lem:def-n:2} is similar. In this case $I=(x^{\be{n-1}-\be{i}}y^i)_{i=0,\dotsc,n-1}$ and the deformation has generic fibre $I'_t=J_t\cap H_t=(x^{\be{n-1}-\be{i}-j_{n-1}+j_i}y^i)_{i=0,\dotsc,n-1}\cap ((x-t)^{j_{n-1}-j_i}y^i)_{i=0,\dotsc, \tilde{\imath}}, y^{\tilde{\imath}+1}-t^{\be{n-1}-j_{n-1}+j_{\tilde{\imath}}+1}(x-t)^{j_{n-1}-j_{\tilde{\imath}}-1} \big)$, where $\tilde{\imath}$ is the greatest integer within $0$ and $n-2$ such that $j_{\tilde{\imath}}<j_{n-1}$. The special fibre is really $I$. Indeed, it is clear if $j_{n-1}=\be{n-1}$; otherwise, $(x-t)^{j_{n-1}-j_i}y^ix^{\be{n-1}-\be{i}-j_{n-1}+j_i}$ belongs to $I'_t$, for any $0\le i\le \tilde{\imath}$ and for any $t\neq 0$, hence $x^{\be{n-1}-\be{i}}y^i$ belongs to the special fibre for any $0\le i\le \tilde{\imath}$. Moreover, $x^{\be{n-1}-\be{i}}y^i$ belongs to the special fibre, also for $\tilde{\imath}<i\le n-1$, being the limit of the following element (which is in $I'_t$ for any $t\neq 0$):
$x^{\be{n-1}-\be{i}}y^{i-\tilde{\imath}-1}(y^{\tilde{\imath}+1}-t^{\be{n-1}-j_{n-1}+j_{\tilde{\imath}}+1}(x-t)^{j_{n-1}-j_{\tilde{\imath}}-1}-(x-t)^{j_{n-1}}w)=x^{\be{n-1}-\be{i}}y^{i}-y^{i-\tilde{\imath}-1}(x-t)^{j_{n-1}-j_{n-2}-1}tx^{2\be{n-1}-\be{i}-j_{n-1}}z$, where $w=\sum_{r=0}^{\be{n-1}-j_{n-1}-1}w_rt^{\be{n-1}-j_{n-1}-1-r}x^r$, in which  $w_0=(-1)^{j_{n-2}}$, and recursively $w_r=\sum_{l=0}^{r-1}(-1)^{r-l}\binom{j_{n-2}+1}{r-l}w_l$ for $1\le r\le j_{n-2}+1$ and $w_r=\sum_{l=r-j_{n-2}-1}^{r-1}(-1)^{r-l}\binom{j_{n-2}+1}{r-l}w_l$ for $j_{n-2}+1\le r\le \be{n-1}\!-\!j_{n-1}\!-\!1$, while $z\!=\!\sum_{r=1}^{j_{n-2}+1}x^{j_{n-2}+1-r}t^{r-1}\big(\sum_{l=1}^{r}w_{\be{n-1}-j_{n-1}-l}\binom{j_{n-2}+1}{r-l} \big)$.

Hence $I$ is contained in the special fibre; furthermore, they are equal by degree reasons: indeed, $A/I$ has length $(n-1)\be{n-1}-\sum_{i=1}^{n-2}\be{i}$, while $A/J_t$ has length $(n-1)\be{n-1}-\sum_{i=1}^{n-2}\be{i}-(\tilde{\imath}+1)j_{n-1}+\sum_{i=1}^{\tilde{\imath}}j_{i}$  and $A/H_t$ has length $(\tilde{\imath}+1)j_{n-1}-\sum_{i=1}^{\tilde{\imath}}j_{i}$, for any $t\neq 0$. The generalized line bundle defined by this deformation has the desired local indices sequence, because $H_t$ defines a Cartier divisor of $X$ and, hence, only $J_t$ contributes to the local indices.

Now let us prove \ref{Lem:def-n:3}. This time $I=(x^{\be{n-1}-\be{i}}y^i)_{i=0,\dotsc,n-1}$ and the deformation has generic fibre $I'_t=J_t\cap H_t=(x^{\be{n-1}-2},x^{\be{n-1}-\be{i}}y^i)_{i=1,\dotsc,n-1}\cap(y-t^{\be{n-1}-1}(x-t),t^{2(\be{n-1}-1)}y^2+(x-t)^2)$. 
The ideal $I$ is contained in the special fibre because both $x^{\be{n-1}-2}\big[(x-t)^2+t^{2(\be{n-1}-1)}y^2\big]$ and $x^{\be{n-1}-\be{i}}y^{i-1}\big\{y-t^{\be{n-1}-1}(x-t)-\big[t^{2(\be{n-1}-1)}y^2+(x-t)^2\big]t\sum_{j=0}^{\be{n-1}-3} t^j x^{\be{n-1}-3-j}\big\}\!=\!x^{\be{n-1}-\be{i}}y^i(1+\\t^{2(\be{n-1}-1)}y \sum_{j=0}^{\be{n-1}-3} t^j x^{\be{n-1}-3-j}) + x^{\be{n-1}-2}t x^{\be{n-1}-\be{i}}y^i(t -t x)$, for $1\le i\le n-2$ belong to both $H_t$ and $J_t$. The conclusion holds, as usual in these kind of proofs, by easy degree considerations. In the previous verifications it was implicitly assumed that $\be{n-1}\ge 3$; if $\be{n-1}=2$, the ideal $I$ is simply $(x^2,y)$ while the generic fibre $I'$ reduces to $((x-t)^2+t^2 y^2, y-t(x-t))$.

Finally, \ref{Lem:def-n:4} is simply the dual of \ref{Lem:def-n:3}.
\end{proof}

\begin{rmk}\label{Rmk:Lem:def-n}
The first point of the above Lemma could be seen as a special case of the second one; it is separated by its relevance, which will be perspicuous in the proof of the next theorem.
\end{rmk}

Now it is possible to prove the following theorem about the connectedness of the locus of stable generalized lien bundles in the moduli space. For multiplicity $n>3$, it remains a bit vague because the $\delta$ from which this locus is connected is not explicit.

\begin{thm}\label{Thm:connessioneluogoflgen}
Let $X$ be a primitive multiple curve of multiplicity $n$ and let $\delta=-\deg{\cC}$.
The locus of stable generalized line bundles in $\mathrm{M}(\sO_X,P_D)$ is connected for $\delta$ sufficiently large. If $n=2,3$ or $n|D-(n(n-1)/2)\delta$, then this locus is connected for any value of $\delta$.
\end{thm}

\begin{proof}
The case $n=2$ is Theorem \cite[Theorem 4.10]{CK}.
Let $\sF$ be a generalized line bundle of indices sequence $\be{.}$, not free in only one point $P$, where $\sF_P\!=\!(x^{\be{n-1}-\be{i}}y^i)_{i=0,\dotsc,n-1}$. By definition, $\sF$ belongs to $\bar{Z}_{\be{1},\dotsc,\be{n-1}}$. It belongs also to $\bar{Z}_{\overline{\be{1}},\overline{\be{1}}+\overline{\be{2}-\be{1}},\dotsc,\overline{\be{1}}+\overline{\be{2}-\be{1}}+\dotsb+\overline{\be{n-1}-\be{n-2}}}$, where $\overline{\be{i}-\be{i-1}}$ is the representative of the congruency class modulo $n$ of $\be{i}-\be{i-1}$ contained in $\{0,\dotsc,n-1\}$, for $1\le i\le n-1$ (as usual $\be{0}=0$). This follows applying, if needed, Lemma \ref{Lem:deformationsn}\ref{Lem:deformationsn:1}, then various times Lemma \ref{Lem:def-n}\ref{Lem:def-n:1} and, finally, also Lemma \ref{Lem:deformationsn}\ref{Lem:deformationsn:2}.

In particular, this implies that there are at most $n^{(n-2)}$ connected components for fixed $D$, i.e. those containing the irreducible components $\bar{Z}_{\be{1},\dotsc,\be{n-1}}$, with $0\le \be{i}-\be{i-1}\le n-1$, for $1\le i\le n-1$, and such that $n|D-(n(n-1)/2)\delta-\sum_{i=1}^{n-1}\be{i}$ (when $\delta$ is small, some of this components do not exist, because their indices are too big to satisfy the stability inequalities).
When $\delta$ is sufficiently large, the locus is connected: a repeated application of \ref{Lem:def-n}\ref{Lem:def-n:2} shows that $\sF$ is the generization of a generalized line bundle $\sG$ being not free only in $P$ and such that $D|\bet{i}{\sG}$ for any $1\le i\le n-2$ and $\bet{n-1}{\sG}\equiv \be{1}+\dotsb+\be{n-1} \pmod{n}$. If $\delta$ is sufficiently large, $\sG$ connects the above cited irreducible components where lies $\sF$, and so $\sG$, with $\bar{Z}_{0,\dotsc,0,\overline{\be{1}+\dotsb+\be{n-1}}}$ (to which $\sG$ belongs by a repeated application of \ref{Lem:deformationsn}\ref{Lem:deformationsn:2}) where $\overline{\be{1}+\dotsb+\be{n-1}}$ is the representative between $0$ and $n-1$ of the congruency class of $\be{1}+\dotsb+\be{n-1}$, i.e. of $D+(n(n-1)/2)\delta$.

The assertion about $n|D-(n(n-1)/2)\delta$ follows form the fact that, under this hypothesis, $\be{1}+\dotsb+\be{n-1}$ is a multiple of $n$ and, hence, $\sF$ is the specialization of a line bundle by Lemma \ref{Lem:def-n}\ref{Lem:def-n:2}.

We end this proof, treating explicitly the case $n=3$. We can exclude the case $D\equiv 0\pmod 3$, already covered in the previous lines. Thus, we can assume $D\not\equiv 0\pmod 3$. By what proved above, the locus of stable generalized line bundles of generalized degree $D$ has at most three connected components: that containing $\bar{Z}_{0,1}$, that containing $\bar{Z}_{1,3}$ and that containing $\bar{Z}_{2,2}$ if $D\equiv 1\pmod 3$ and the connected component of $\bar{Z}_{1,1}$, that of $\bar{Z}_{0,2}$ and that of $\bar{Z}_{2,3}$ if $D\equiv 2\pmod 3$. For $\delta=1$, in each case, only one of them really exists, because only the indices $(0,1)$ and $(1,1)$ verify the stability conditions \eqref{Eq:ssin}.

Now, we assume $\delta>1$ and we distinguish the two cases according the congruence class of $D$.

If $D\equiv 1\pmod 3$, there exists a stable generalized line bundle $\sG$ whose stalks are free in all closed points except one, say $P$, where it holds that $(\btp{1}{P}{\sG},\,\btp{2}{P}{\sG})=(1,\,3)$ and $\alpha_{3,1,P}=0$. By Lemma \ref{Lem:def-n}\ref{Lem:def-n:2} $\sG$ belongs to $\bar{Z}_{0,1}$ and so connects it with $\bar{Z}_{1,3}$, to which $\sG$ belongs by definition. Moreover, by Lemma \ref{Lem:def-n}\ref{Lem:def-n:4}, $\sG$ belongs also to $\bar{Z}_{2,2}$ and so the locus of stable generalized line bundles with $D\equiv 1\pmod 3$ is connected. 

If $D\equiv 2\pmod 3$, the situation is similar: there exists a stable generalized line bundle $\sG$ whose stalks are free in all closed points except one, say $P$, where $(\btp{1}{P}{\sG},\,\btp{2}{P}{\sG})=(2,\,3)$ and $\alpha_{3,1,P}=0$. By definition, $\sG$ belongs to $\bar{Z}_{2,3}$. By, respectively, Lemmata \ref{Lem:def-n}\ref{Lem:def-n:2} and \ref{Lem:def-n}\ref{Lem:def-n:3}, $\sG$ belongs also to $\bar{Z}_{1,1}$ and $\bar{Z}_{0,2}$.
\end{proof}
\begin{rmk}\label{Rmk:Thm:connessioneluogoflgen}
As observed in the proof, that there are at most $n^{(n-2)}$ connected components for fixed $D$. But many (maybe all) of them coincide, as in the case of $n|D-(n(n-1)/2)\delta$ or in that of $n=3$. The case $n=4$, where the number of candidates is not excessive, is relatively easy to be treated by hand and the result is the following: for $\delta=1$ or $2$ the locus of generalized line bundle is certainly connected for $D$ even (one has to use also Lemma \ref{Lem:def-n}\ref{Lem:def-n:4}) and it has at most $2$ connected components for $D$ odd, while for $\delta\ge 3$ it is always connected. In general the situation is not really easy to handle directly and I do not know explicitly from which value of $\delta$ the known deformations are sufficient to conclude the connection. However I think that there should be other deformations implying that this locus is always connected, although I had not been able to find them until now.
\end{rmk}

We end this subsection about the global geometry with the promised conjecture about the whole picture of the irreducible components of $\mathrm{M}(\sO_{C_3},P_D)$ (which, as already observed, is the compactified Jacobian of $C_3$ when $3$ divides $D$).

\begin{conj}\label{Conj:CompIrrn=3}
Let $X$ be a primitive multiple curve of multiplicity $3$ such that $\delta=-\deg{\cC}>0$ and such that $g_1\ge 2$, where $g_1$ is the genus of its reduced subcurve. 
\begin{enumerate}

\item\label{Conj:CompIrrn=3:1} If $\delta\le 2(g_1-1)$, then the irreducible components of $\mathrm{M}(\sO_{C_3},P_D)$ are the following:
\begin{enumerate}
\item $\mathrm{M}(\sO_{C},P_D)$, i.e. the moduli scheme of semistable rank $3$ vector bundles of degree $D$ over $C$;
\item The closures of the loci of quasi locally free sheaves of complete type $((2,1,0),(d_0,d_1,0))$ for any pair of integers $d_0$ and $d_1$ such that $d_0+d_1=D$ and $(d_0-3\delta)/2<d_1<d_0/2$;
\item $\bar{Z}_{\be{1},\be{2}}$ for any pair of non-negative integers $\be{1}\le\be{2}$ satisfying $3|D-\be{1}-\be{2}$, $0\le \be{2}+\be{1}< 3\delta$ and $0\le 2\be{2}-\be{1}< 3\delta$.
\end{enumerate}

\item\label{Conj:CompIrrn=3:2} If $\delta>2(g_1-1)$, the only irreducible components of
$\mathrm{M}(\sO_{C_3},P_D)$ are the $\bar{Z}_{\be{1},\be{2}}$, with $\be{1}$ and $\be{2}$ as above.
\end{enumerate}
\end{conj} 

The quasi locally free sheaves supported on $C_2$ cited in the statement are of rigid type so the loci above listed are surely irreducible components of dimension $1+2\delta+5(g_1-1)$ of $\mathrm{M}(\sO_{C_2},P_D)$ by \cite[Proposition 6.12]{DR2} if non-empty and their are non-empty by \cite[Corollary 2.7]{S2}.

The first part of this conjecture is implied by \cite[Conjecture 4.6(i)]{S2}, by Lemma \ref{Lem:openessofgenlinbund} and dimensional reasons: indeed, $1+9(g_1-1)\ge 1+2\delta+5(g_1-1)\ge g_3=1+3\delta+3(g_1-1)$ if $\delta\le 2g_1-2$. The second part would follow from \cite[Conjecture 4.6(ii)]{S2}, if one were able to show that, if $\sF$ is a sheaf of type $(1,1)$ on $C_2$, then a deformation of $\sF^{(1)}$, which is a rank $2$ vector bundle on $C$, to a generalized line bundle on $C_2$ (this deformations exists by \cite[Proposition 3.3]{S2}, or rather by its special case \cite[Theorem 1]{S3}, being under the hypothesis $\delta>2g_1-2$), induces a deformation of $\sF$ to a generalized line bundle on $C_3$. 
The whole conjecture is inspired by the case of generalized rank $2$ sheaves on ribbons (see \cite[Theorem 4.7]{CK} and \cite[Corollary 1]{S3}). 

In the special case in which $X$ is the spectral cover associated to nilpotent Higgs bundles of rank $3$ over $C$, the conjecture has to hold by the spectral correspondence: all the candidate irreducible components (i.e. the irreducible components of stable generalized line bundles and the closures of the loci of semistable vector bundles of rank $3$ over $C$ and of stable quasi locally free sheaves of rigid type of generalized rank $3$ over $C_2$), as we have already observed above, are really irreducible components by the fact they are all of the same dimension $g_3$ (in this case $\delta=2g_1-2$ and $g_3=9g_1-8$) and they have different generic elements. In order to understand if further components should exist, it is possible to compare their number (that can be computed using the stability conditions) with that of irreducible components of the moduli space of semistable Higgs bundles of rank $3$ over $C$, which is $2g_1(g_1-1)+g_1$ (when $D$ is coprime to $3$, see \cite[Examples at page 306]{Sc}): it follows that the above cited components should be all the components if the generalized degree is coprime to $3$.

Moreover, also without doing this computation, by \cite[Corollary 2.4]{Bo} about irreducible components of the nilpotent cone of Higgs bundles (at level of stacks) in our language, it holds that in the cases involved in the spectral correspondence any irreducible component is the closure of the locus of sheaves with fixed complete type and quasi locally free sheaves of rigid type of generalized rank $3$ over $C_2$ cover all the possible complete types for sheaves of generalized rank $3$ on the ribbon (excluding rank $3$ vector bundles over $C$).

For details about the relation of trivial primitive multiple curves with Higgs bundles see \cite[Appendix A]{S1}.
 
\subsection{Local geometry: Zariski tangent space}\label{SubSeclocgeom}
This subsection about the local geometry of $\mathrm{M}(\sO_X,P_D)$ is devoted to the computation of the dimension of the tangent space to points corresponding to generalized line bundles on $X$ and to vector bundles of rank $n$ on $C$. The results are quite similar to \cite[Proposition 4.11]{CK} and also the lemmata used to get them are similar, both in the enunciation and in the proof, to \cite[Lemmata 4.12 and 4.13]{CK}. This is a good point to observe that there is a little mistake in the second assertion of \cite[Lemma 4.12]{CK}: the right hypothesis to simplify the formula about the dimension of the $\Ext{1}$ of a generalized line bundle on a ribbon is that the associated blow up does not have non-trivial global sections, hence its genus (and not that of the ribbon) has to be greater than or equal to two times the genus of the reduced curve, i.e. $2g_1+\be{1}\le g_2$ (using their notation $2\bar{g}+b\le g$, and not $2\bar{g}\le g$, as asserted in the cited Lemma); in any case this error does not affect \cite[Proposition 4.11]{CK}, because it is about stable generalized line bundles, for which it holds also the right hypothesis.
\begin{prop}\label{Prop:DimSpTgn}
Let $X$ be a primitive multiple curve of multiplicity $n$ and let $x$ be a point of $\mathrm{M}_s(\sO_X, P_D)$. 
\begin{enumerate}

\item If $x$ corresponds to a a stable vector bundle $\cE$ of rank $n$ over $C$, then
\begin{align}
\dim T_x\mathrm{M}(\sO_X, P_D)&=n^2(g_1-1)+1+\operatorname{h}^0(C,\SEnd{\cE}\otimes\cC^{-1})\label{Eq:DimTgModVBn}\\
&=n^2\delta+1\text{ if }\delta>\deg{\omega_C}\;(\iff\, g_2>4g_1-3)\label{Eq:DimTgModVBnsp}.
\end{align}

\item If $x$ corresponds to a stable generalized line bundle $\sF$ of local indices sequence $\bp{.}{.}$ such that in each point $P$ where $\sF$ is not free there exists an integer $1\le h(P)\le n-1$ such that $0=\bp{h(P)-1}{P}<\bp{h(P)}{P}=\bp{n-1}{P}$, then
\begin{equation}\label{Eq:DimTgModFLGn}
\dim T_x\mathrm{M}(\sO_X, P_D)=g_n+\sum\limits_{j=1}^{r}\min\{h(P_j),\,n-h(P_j)\}\bp{n-1}{P_j},
\end{equation}
where $P_1,\dotsc,P_r$ are the points of $C$ where $\sF$ is not locally free.

\item If $n=3$ and $x$ corresponds to a stable generalized line bundle $\sF$ of indices sequence $\bp{.}{.}$, then
\begin{equation}\label{Eq:DimTgModFLGn=3}
\dim T_x\mathrm{M}(\sO_X, P_D)=g_3+\be{2}+\sum\limits_{j=1}^{r}\min\{\bp{1}{P_j},\,\bp{2}{P_j}-\bp{1}{P_j}\},
\end{equation}
where $P_1,\dotsc,P_r$ are the points of $C$ where $\sF$ is not locally free.
\end{enumerate}
\end{prop}
The case $n=3$ and the general case are equal for those generalized line bundles verifying the hypothesis of the second point of the statement, but for $n=3$ the result is the same for any other generalized line bundle.
Before proving the Proposition, we give its following immediate consequence.
\begin{cor}\label{Cor:SpTgpgen}
The tangent space to a generic point of the irreducible component $\bar{Z}_{\be{1},\dotsc,\be{n-1}}$ has dimension 
\[
g_n+\sum_{i=[\frac{n+1}{2}]}^{n-1}\be{i}-\sum_{i=1}^{[\frac{n-2}{2}]}\be{i}=g_n+\be{n-1}+\sum_{i=[\frac{n+1}{2}]}^{n-2}(\be{i}-\be{n-1-i}).
\]
In particular, only the component of line bundles, i.e. $\bar{Z}_{0,\dotsc,0}$, is generically reduced. 
\end{cor} 
\begin{proof}
The first assertion follows from formula \eqref{Eq:DimTgModFLGn} and Theorem \ref{Thm:compirrid}, which describes the generic points of $\bar{Z}_{\be{1},\dotsc,\be{n-1}}$.

It implies, in particular, that this generic dimension is always greater than or equal to $g_n+\be{n-1}$, by Lemma \ref{Lem:propindices}. Hence, recalling that, again by Theorem \ref{Thm:compirrid}, each $\bar{Z}_{\be{1},\dotsc,\be{n-1}}$ has dimension $g_n$, the second assertion is a consequence of the first one and of Corollary \ref{Cor:flifTi=0}.
\end{proof}

The Proposition is implied by the well-known fact that the Zariski tangent space to a point corresponding to a stable sheaf $\sG$ in the moduli space is canonically isomorphic to $\Ext{1}(\sG,\sG)$ (see, e.g., \cite[Corollary 4.5.2]{HL}) and by the next lemmata calculating this dimension for the various kinds of sheaves cited in the statement.

We start with vector bundles of rank $n$ over $C$. The first assertion of the lemma is essentially \cite[Remark 2.7(iii)]{I}, from which the first part of the proof is taken, too.

\begin{lemma}\label{Lem:ExtVBrkn}
Let $\cE$ be a stable vector bundle of rank $n\ge 2$ over $C$. It holds that
\begin{equation}\label{Eq:ExtVBrkn}
\dim(\Ext{1}_{\sO_{C_n}}(\cE,\cE))=n^2(g_1-1)+1+\operatorname{h}^0(C,\cC^{-1}\otimes\SEnd{\cE}). 
\end{equation}
If, furthermore, $\delta=-\deg{\cC}>2g_1-2$, then this formula simplifies to
\begin{equation}\label{Eq:ExtVBrknspecial}
\dim(\Ext{1}_{\sO_{C_n}}(\cE,\cE))=n^2\delta +1.
\end{equation}
\end{lemma}
\begin{proof}
The Ext-spectral sequence $H^p(\SExt{q}{\sO_{C_n}}(\cE,\cE))\Rightarrow\Ext{p+q}_{\sO_{C_n}}(\cE,\cE)$ implies that the following sequence is exact:
\[
0\to H^1(\SEnd{\cE})\to\Ext{1}_{\sO_{C_n}}(\cE,\cE)\to H^0(\SExt{1}{\sO_{C_n}}(\cE,\cE))\to 0.
\]

It is well-known that $H^1(\SEnd{\cE})=\Ext{1}_{\sO_C}(\cE,\cE)$, being $\cE$ a vector bundle on $C$. It holds also that $H^0(\SExt{1}{\sO_{C_n}}(\cE,\cE))\cong \Hom(\cC\otimes\cE,\cE)$ (it can be checked using, e.g., the locally free periodical resolution $\dotsb\to\sC^n\otimes\sE\to\sC\otimes\sE \to\sE\to\cE\to 0$, where $\sE$ is a vector bundle on $C_n$ extending $\cE$, cf. also the proof of \cite[Proposition 3.14]{DR2}). Hence, formula \eqref{Eq:ExtVBrkn} is implied by well-known properties of stable vector bundles over smooth projective curves and by the trivial identity $\Hom(\cC\otimes\cE,\cE)=H^0(C,\cC^{-1}\otimes\SEnd{\cE})$.

Assume now $\delta=\deg{\cC^{-1}}>2g_1-2$. Consider $\cC^{-1}\otimes\SEnd{\cE}$: it is a semistable vector bundle of rank $n^2$ and degree $n^2\delta$ on $C$, because $\cC^{-1}$ is a line bundle and $\SEnd{\cE}$ is a semistable vector bundle of rank $n^2$ and degree $0$. Hence, $\chi(C,\cC^{-1}\otimes\SEnd{\cE})=n^2(1-g_1)+n^2\delta$. Furthermore, by Serre duality, $\operatorname{h}^1(C,\cC^{-1}\otimes\SEnd{\cE})=\operatorname{h}^0(C,\omega_C\otimes\cC\otimes\SEnd{\cE})$ and the latter vanishes, because $\omega_C\otimes\cC\otimes\SEnd{\cE}$ is semistable of degree $n^2(2g_1-2)-n^2\delta<0$, by hypothesis. Therefore, $\operatorname{h}^0(\omega_C\otimes\cC\otimes\SEnd{\cE})=\chi(\omega_C\otimes\cC\otimes\SEnd{\cE})$ and formula \eqref{Eq:ExtVBrknspecial} holds.  
\end{proof}

The next step is the computation for generalized line bundles in arbitrary multiplicity.
\begin{lemma}\label{Lem:ExtFLGn}
If $\sF$ is a generalized line bundle on $X=C_n$ with local indices sequence $\bp{.}{.}$ such that in each point $P$ where $\sF$ is not locally free there exists an integer $1\le h(P)\le n-1$ such that $0=\bp{h(P)-1}{P}<\bp{h(P)}{P}=\bp{n-1}{P}$, then
\begin{equation}\label{Eq:ExtFLGn}
\dim(\Ext{1}(\sF,\sF))=g_{n}+\tilde{b}_{n-1}+\operatorname{h}^0(X',\sO_{X'})-1,
\end{equation}
where $\tilde{b}_{n-1}=\sum_{j=1}^{r}\min\{h(P_j),\,n-h(P_j)\}\bp{n-1}{P_j}$, where $P_1,\dotsc,P_r$ are the points of $C$ in which $\sF$ is not free and $X'$ is the blow up associated to $\sF$ as in Corollary \ref{Cor:globstrmoltn}\ref{Cor:globstrmoltn:3}.

If, moreover, $\sF$ is stable, then this formula simplifies to
\begin{equation}\label{Eq:ExtFLGnspecial}
\dim(\Ext{1}(\sF,\sF))=g_{n}+\tilde{b}_{n-1}.
\end{equation}
\end{lemma}
\begin{proof}

The Ext-spectral sequence $H^p(X,\SExt{q}{ }(\sF,\sF))\Rightarrow\Ext{p+q}_{\sO_{X}}(\sF,\sF)$ implies the existence of the following short exact sequence
\[
0\to H^1(X,\SEnd{\sF})\to\Ext{1}(\sF,\sF)\to H^0(X,\SExt{1}{ }(\sF,\sF))\to 0.
\]
Hence, in order to get the result, it suffices to compute the dimensions of the two external terms. By Remark \ref{Rmk:b1=bn-1}, we have $\SEnd{\sF}\simeq q_*(\sO_{X'})$, where $q:X'\to X$ is the blow up there studied. Therefore, $H^1(X,\SEnd{\sF})=H^1(X',\sO_{X'})$ and the latter has dimension $g(X')-\operatorname{h}^0(X',\sO_{X'})+1=g_n-\tilde{b}_{n-1}-\operatorname{h}^0(X',\sO_{X'})+1$ (this formula is due to the definition of the blow up $X'$). By Lemma \ref{Lem:glosecblup}, if $\sF$ is stable, then $\operatorname{h}^0(X',\sO_{X'})=1$, justifying the difference between formulae \eqref{Eq:ExtFLGn} and \eqref{Eq:ExtFLGnspecial}.

It remains to calculate $\operatorname{h}^0(X,\SExt{1}{ }(\sF,\sF))$. As in the case of ribbons treated in \cite[Lemma 4.12]{CK}, by the fact $\sF$ is not free only in a finite number of points, it is clear that $\SExt{1}{ }(\sF,\sF)$ is supported on $P_1,\dotsc,P_r$ and it can be decomposed as $\bigoplus_{j=1}^{r}\Ext{1}(\sF_{P_j},\sF_{P_j})$.

In the next lines we will show that $\dim(\Ext{1}(\sF_{P_j},\sF_{P_j}))=2\tilde{b}_{n-1}$ and, thus, formulae \eqref{Eq:ExtFLGn} and \eqref{Eq:ExtFLGnspecial} hold, as desired.

In order to do this computation, we will use local notation with $A=\sO_{X,P_j}$ and $\sF_{P_j}$ will be denoted by $I$, while $\bp{n-1}{I}=b$ and $h=h(P_j)$.

By Corollary \ref{Cor:locstructspecial}, $I$ is isomorphic to the ideal $(x^b,y^{h})$.

It has the following periodic free resolution:
\[
\dotsb\longrightarrow A^2\overset{M_2}{\longrightarrow}A^2\overset{M_1}{\longrightarrow}A^2\overset{f}{\longrightarrow}I\longrightarrow 0,
\] 
where
\[
M_1=
\begin{pmatrix}
y^{n-h} & -x^{b}-\alpha y\\
0^{\phantom{n-h}} & y^h 
\end{pmatrix}\!, 
M_2=
\begin{pmatrix}
y^{h} & x^{b}+\alpha y\\
0^{\phantom{h}} & y^{n-h} 
\end{pmatrix} \text{and}
\left\{
\begin{aligned}
&f((1,0))\!=\!y^h\\
&f((0,1))\!=\!x^{b}+\alpha y.
\end{aligned}
\right.
\]
From the resolution one gets the complex
\[
\dotsb\longleftarrow\Hom(A^2,I)\overset{\phantom{x}a_2}{\longleftarrow}\Hom(A^2,I)\overset{\phantom{x}a_1}{\longleftarrow}\Hom(A^2,I),
\] 
where $a_i$ is the homomorphism induced by multiplication by $M_i$, for $i=1,2$. By definition, $\Ext{1}(I,I)=\ker{a_2}/\im{a_1}$. It holds that
\[
\varphi\in\im{a_1}\iff \left\{
\begin{aligned}
&\varphi((1,0))=\beta_1y^{n-h}(x^b+\alpha y)\\
&\varphi((0,1))=-\beta_1(x^{b}+\alpha y)^2+\beta_2y^h(x^{b}+\alpha y)+\beta_3y^{2h},
\end{aligned}
\right.
\]
with $\beta_i\in A$, for $1\le i\le 3$. In order to study $\ker{a_2}$ it is convenient to distinguish two cases: $n-h\le h$ and $h<n-h$. 
In the first one, we have that
\[
\psi\in\ker{a_2}\iff \left\{
\begin{aligned}
&\psi((1,0))=\gamma_1(x^{b}+\alpha y)y^{n-h}+\gamma_2 y^h\\
&\psi((0,1))=\gamma_2(x^{b}+\alpha y)y^{2h-n}+\gamma_3y^{h},
\end{aligned}
\right.
\]
with $\gamma_i\in A$, for $1\le i\le 3$.\\
Otherwise, it holds that 
\[
\psi\in\ker{a_2}\iff \left\{
\begin{aligned}
&\psi((1,0))=\gamma_1y^{n-h}\\
&\psi((0,1))=(\gamma_2-\gamma_1)(x^{b}+\alpha y)+\gamma_3y^{2h},
\end{aligned}
\right.
\]
with $\gamma_i\in A$, for $1\le i\le 3$. In both cases the length of $\Ext{1}(I,I)$ is the desired one.
\end{proof}
\begin{rmk}\label{Rmk:Lem:ExtFLGn}
The beginning of the proof, i.e. the existence of the short exact sequence
$H^1(X,\SEnd{\sF})\inj\Ext{1}(\sF,\sF)\surj H^0(X,\SExt{1}{ }(\sF,\sF))$ and also the identification of the right hand term with $\bigoplus_{j=1}^{r}\Ext{1}(\sF_{P_j},\sF_{P_j})$, is true for any generalized line bundle $\sF$ over $X$. However, the interpretation of $\SEnd{\sF}$ in terms of an appropriate blow up is known only for those verifying the hypotheses of Lemma \ref{Lem:blowupmultn}\ref{Lem:blowupmultn:1} and their duals, within which there are those studied in the above Lemma. These are particularly significant because within them there are the generic elements of the irreducible components of the moduli space containing stable generalized line bundles (cf. Theorem \ref{Thm:compirrid}). Moreover, in this case the explicit calculation of the extensions of the stalks is not too hard, because they have only two local generators. 

On the other hand, in the case of multiplicity $3$, i.e. in the following Lemma, the interpretation of the endomorphism sheaf in terms of a blow up is always known and all the computations are not too difficult. 
\end{rmk}

\begin{lemma}\label{Lem:ExtFLGn=3} 
If $\sF$ is a generalized line bundle on $X=C_3$ of local indices sequence $\bp{.}{.}$, then
\begin{equation}\label{Eq:ExtFLGn=3}
\dim(\Ext{1}(\sF,\sF))=g_{3}+\be{2}+\tilde{b}_{1}+\operatorname{h}^0(X',\sO_{X'})-1,
\end{equation}
where $\tilde{b}_{1}=\sum_{j=1}^{r}\min\{\bp{1}{P_j},\,\bp{2}{P_j}-\bp{1}{P_j}\}$, being $P_1,\dotsc,P_r$ the points of $C$ where $\sF$ is not locally free and  $X'$ is the blow up associated to $\sF$ as in Corollary \ref{Cor:globstrmolt3}.

If, moreover, $\sF$ is stable, then this formula simplifies to
\begin{equation}\label{Eq:ExtFLGn=3special}
\dim(\Ext{1}(\sF,\sF))=g_{3}+\be{2}+\tilde{b}_{1}.
\end{equation}
\end{lemma}
\begin{proof}
The fundamental ideas of the proof, as pointed out in the previous Remark, are the same of the proof of Lemma \ref{Lem:ExtFLGn} about arbitrary multiplicity.

As there, the Ext-spectral sequence implies the existence of the following short exact sequence
\[
0\to H^1(X,\SEnd{\sF})\to\Ext{1}(\sF,\sF)\to H^0(X,\SExt{1}{ }(\sF,\sF))\to 0.
\]
Therefore, it is sufficient to compute the dimensions of the two external terms in order to get the result. By Remark \ref{Rmk:b1=b2}, it holds that $\SEnd{\sF}\simeq q_*(\sO_{X'})$, where $q:X'\to X$ is the blow up there studied. It follows that $H^1(X,\SEnd{\sF})=H^1(X',\sO_{X'})$ and the latter has dimension $g(X')-\operatorname{h}^0(X',\sO_{X'})+1=g_3-\be{2}-\tilde{b}_{1}-\operatorname{h}^0(X',\sO_{X'})+1$ (this formula is implied by the definition of the blow up $X'$). Again as in arbitrary multiplicity, by Lemma \ref{Lem:glosecblup}, if $\sF$ is stable, then $\operatorname{h}^0(X',\sO_{X'})=1$, justifying the difference between formulae \eqref{Eq:ExtFLGn=3} and \eqref{Eq:ExtFLGn=3special}.

It remains to calculate $\operatorname{h}^0(X,\SExt{1}{ }(\sF,\sF))$. As in the previous case, it is clear that $\SExt{1}{ }(\sF,\sF)$ is supported on $P_1,\dotsc,P_r$ and that it can be decomposed as $\bigoplus_{j=1}^{r}\Ext{1}(\sF_{P_j},\sF_{P_j})$.

In the following lines we will show that  $\dim(\Ext{1}(\sF_{P_j},\sF_{P_j}))=2\bp{2}{P_j}+2\min\{\bp{1}{P_j},\bp{2}{P_j}-\bp{1}{P_j}\}$; therefore, formulae \eqref{Eq:ExtFLGn=3} and \eqref{Eq:ExtFLGn=3special} hold, as desired.

In order to do the explicit computations, it is useful to distinguish three different cases, according to the indices of $\sF$ in the point $P_j$:
\begin{enumerate}
\item\label{Proof:Lem:ExtFLGn=3:1} $0=\bp{1}{P_j}<\bp{2}{P_j}$;
\item\label{Proof:Lem:ExtFLGn=3:2} $0<\bp{1}{P_j}=\bp{2}{P_j}$;
\item\label{Proof:Lem:ExtFLGn=3:3} $0<\bp{1}{P_j}<\bp{2}{P_j}$.
\end{enumerate} 
The first two possibilities are special cases of the modules considered in arbitrary multiplicity in the proof of Lemma \ref{Lem:ExtFLGn}. So, we have to prove only \ref{Proof:Lem:ExtFLGn=3:3}. In this case, by Local Structure Theorem, $I\cong(x^{\be{2}}+\alpha y,\,x^{\be{2}-\be{1}}y,\,y^2)$ (observe that it is possible to assume that not only $x^{\be{2}-\be{1}}$ and $y$ do not divide $\alpha$ but also that $x^{\be{1}}$ does not divide it: indeed, it holds that $(x^{\be{2}}+x^{\be{1}}\epsilon y,\,x^{\be{2}-\be{1}}y,\,y^2)\cong(x^{\be{2}},\,x^{\be{2}-\be{1}}y,\,y^2)$, for any $\epsilon\in A$). The method of calculation is similar to the other cases, but the computations are harder, having one more generator. The following is a periodic free resolution of $I$: 
\[
\dotsb\longrightarrow A^2\overset{M_2}{\longrightarrow}A^2\overset{M_1}{\longrightarrow}A^2\overset{f}{\longrightarrow}I\longrightarrow 0,
\] 
where
\[
M_1 =
\left( \begin{array}{ccc}
y & -x^{\be{2}-\be{1}}y &-\alpha^{\phantom{\be{1}}}\\
0 & y &-x^{\be{1}}\\
0 & 0 &\phantom{-} y^{\phantom{\be{1}}} 
\end{array} \right) \text{and }
M_2 =
\left( \begin{array}{ccc}
y^2 & x^{\be{2}-\be{1}}y & x^{\be{2}}+\alpha y\\
0 & y^2 & -x^{\be{1}}y\\
0 & 0^{\phantom{2}} & y^2 
\end{array}
 \right)
 \]
while $f((1,0,0))=y^2$, $f((0,1,0))=x^{\be{2}-\be{1}}y$ and $f((0,0,1))=x^{\be{2}}+\alpha y$.

From the resolution one gets the complex
\[
\dotsb\longleftarrow\Hom(A^3,I)\overset{\phantom{x}a_2}{\longleftarrow}\Hom(A^3,I)\overset{\phantom{x}a_1}{\longleftarrow}\Hom(A^3,I),
\] 
where $a_i$ is the homomorphism induced by multiplication by $M_i$, for $i=1,2$. By definition, $\Ext{1}(I,I)=\ker{a_2}/\im{a_1}$. It holds that $\varphi\in\im{a_1}$ if and only if
\[
\left\{
\begin{aligned}
&\varphi((1,0,0))=\beta_1x^{\be{2}-\be{1}}y^2+\beta_2x^{\be{2}}y\\
&
\begin{split}
\varphi((0,1,0))=&(\beta_3x^{\be{2}-\be{1}}+\beta_4\alpha)y^2+(-\beta_1x^{\be{2}-\be{1}}+\beta_4x^{\be{1}})x^{\be{2}-\be{1}}y
\\
&-\beta_2x^{\be{2}-\be{1}}(x^{\be{2}}+\alpha y)
\end{split}
\\
&
\begin{split}
\varphi((0,0,1))=&(\beta_5\alpha+\beta_6x^{\be{1}}+\beta_7x^{\be{2}-\be{1}})y^2+(-\beta_1\alpha+(\beta_5-\beta_3)x^{\be{1}})x^{\be{2}-\be{1}}y\\
&-(\beta_2\alpha+\beta_4x^{\be{1}})(x^{\be{2}}+\alpha y),
\end{split}
\end{aligned}
\right.
\]
with $\beta_i\in A$, for any $1\le i\le 7$; 
on the other side $\psi\in\ker{a_2}$ if and only if
\[
\left\{
\begin{aligned}
&\psi((1,0,0))=\gamma_1y^2+\gamma_2x^{\max\{0,\,2\be{1}-\be{2}\}}x^{\be{2}-\be{1}}y\\
&\psi((0,1,0))=-\gamma_2x^{\max\{\be{2}-2\be{1},\,0\}}(x^{\be{2}}+\alpha y)+\gamma_3 y^2+\gamma_4x^{\be{2}-\be{1}}y\\
&\psi((0,0,1))=(-\gamma_1-\gamma_4)(x^{\be{2}}+\alpha y)+\gamma_5y^2+\gamma_6x^{\be{2}-\be{1}}y,
\end{aligned}
\right.
\]
with $\gamma_i\in A$, for $1\le i\le 6$. Hence, the desired result follows from these direct computations (observing that each $\beta_i$ can be used to \emph{limit} almost one $\gamma_j$).
\end{proof}

\paragraph*{Acknowledgements}
This paper is extracted from my forthcoming doctoral thesis \cite{S1}, so I would like to express my gratitude to my advisor, Filippo Viviani, for his guidance.\\
\emph{Ad maiorem Dei gloriam}

\end{document}